%% file: ms.tex
\documentclass[a4paper, 11pt]{article}
\usepackage{etoolbox}
\input{preamble_ms.tex}

\title{Elastic wave propagation in anisotropic solids using energy-stable finite differences with weakly enforced boundary and interface conditions}

\author{Martin Almquist\thanks{Department of Geophysics, Stanford University, Stanford, CA, USA} \and Eric M.\ Dunham$^{*,}$\thanks{Institute for Computational and Mathematical Engineering, Stanford University, Stanford, CA, USA} }

\date{}

\begin{document}
\maketitle

\begin{abstract}
  Summation-by-parts (SBP) finite difference methods have several desirable properties for second-order wave equations. They combine the computational efficiency of narrow-stencil finite difference operators with provable stability on curvilinear multiblock grids. While several techniques for boundary and interface conditions exist, weak imposition via simultaneous approximation terms (SATs) is perhaps the most flexible one. Although SBP methods have been applied to elastic wave equations many times, an SBP-SAT method for general anisotropic elastic wave equations has not yet been presented in the literature. We fill this gap by deriving energy-stable self-adjoint SBP-SAT methods for general anisotropic materials on curvilinear multiblock grids. The methods are based on fully compatible SBP operators. \refone{Although this paper focuses on classical SBP finite difference operators, the presented boundary and interface treatments are general and apply to a range of methods that satisfy an SBP property}. We demonstrate the stability and accuracy properties of a particular set of fully compatible SBP-SAT schemes using the method of manufactured solutions. We also demonstrate the \otherchange{utility} of the new method in elastodynamic cloaking and seismic imaging in mountainous regions.
\end{abstract}

\section{Introduction}
This paper considers numerical solution of elastic wave equations in complex geometries. We deal with the most general form of the anisotropic elastic wave equation (AEWE), which includes the isotropic elastic wave equation (IEWE) as a special case. Generally speaking, high-order finite difference methods are computationally efficient for wave-dominated equations with smooth solutions \cite{KreissOliger72}.  Finite difference (FD) operators with the summation-by-parts (SBP) property \cite{KreissScherer74} lead to energy-stable discretizations on curvilinear multiblock grids when combined with suitable methods for imposing boundary and interface conditions. SBP FD methods may be used alone in moderately complex geometries, or as part of efficient hybrid solvers \cite{Lundquist2018,Gao2019} when unstructured meshing capabilities are required in parts of the domain. Recent applications of SBP methods to elastic wave equations include \cite{Rao2018}, which applied a second-order accurate scheme to tilted transversely isotropic media, and \cite{Sun2020}, which solved the first-order form of the IEWE. Another noteworthy contribution \cite{Dovgilovich15} introduced dual first-derivative SBP operators to solve the AEWE.

To minimize the number of unknowns, this paper discretizes the second-order form of the AEWE. For second order equations, narrow-stencil second-derivative SBP operators \cite{MattssonNordstrom04,Mattsson11} typically provide superior accuracy compared to applying a first-derivative operator twice. As a rule of thumb, the global convergence rate is one order higher \cite{MattssonNordstrom04} and the numerical dispersion relation mimics the exact dispersion relation better for marginally resolved modes \cite{Kreiss_wave}. Hence, we only consider narrow-stencil operators in this paper.

While SBP operators may be combined with various techniques for imposing boundary and interface conditions, weak enforcement via simultaneous approximation terms (SATs) \cite{CarpenterGottlieb94} has proven competitive in a wide range of applications \cite{DelReyFernandez2014a,Svard2014}. \refone{The SAT boundary and interface treatment does not introduce systems of equations to solve} and extends naturally to nonconforming grid blocks (see for example \cite{MattssonCarpenter09,Lundquist2018,AlmquistWangWerpers2019}) and nonlinear frictional interface conditions (see for example \cite{Kozdon2011,Kozdon2013}). Previous works that used narrow-stencil SBP operators combined with other boundary treatments include \cite{Petersson2015,Duru201437,Duru2014}. In \cite{Petersson2015}, Petersson and Sj{\"o}green presented a fourth order SBP scheme for the AEWE on curvilinear single-block grids. Boundary conditions were imposed with a ghost-point technique, \refone{which requires solving small linear systems for the ghost-point values on boundaries and interfaces}. In \cite{Duru201437}, Duru and Virta presented an SBP-SAT scheme for the IEWE on curvilinear multiblock grids. Traction boundary conditions were imposed using SATs and displacement boundary conditions were strongly enforced, using the injection method \cite{Duru2014}. Herein, we construct an SBP-SAT method for the AEWE on curvilinear multiblock grids. Robin boundary conditions (which include traction conditions), displacement boundary conditions, and interface conditions, are all imposed using SATs. We prove that the spatial discretization is energy stable and self-adjoint. \refone{While we only perform numerical experiments with classical finite difference SBP operators on uniform grids, the presented methodology is general and may be applied to a wide range of methods that satisfy the SBP property. Examples include finite differences on non-uniform grids \cite{MattssonAlmquistCarpenter13,DelReyFernandez2014}, multidimensional finite differences \cite{Hicken2016}, and discontinuous Galerkin spectral element methods \cite{Gassner2013, DelReyFernandez2014}. }

The methods derived in this paper are based on \emph{fully compatible} diagonal-norm second-derivative SBP operators \cite{MattssonParisi09} (see Section \ref{sec:sbp} for the definition). The assumption of full compatibility greatly simplifies the stability analysis when using SATs to impose displacement boundary conditions and inter-block couplings. The significant simplifications facilitated by the fully compatible operators were noted in \cite{Duru201437} for the IEWE, and later in \cite{AlmquistDunham2020} for the acoustic wave equation. The fully compatible operators are to be contrasted with \emph{compatible} operators, which are more commonly used. The compatible operators constructed by Mattsson in \cite{Mattsson11} with interior order $2q$ have boundary closures of order $q$ and boundary derivative operators of order $q+1$, yielding $(q+2)$th order global accuracy in most numerical experiments. Fisher and Carpenter \cite{FisherCarpenter2013} constructed a fully compatible $2q=4$ operator with $q$th order closures and boundary derivatives of order $q$. The reduction of the boundary derivative order (compared to Mattsson's compatible operators) increases the local truncation error by one order for Neumann-type boundary conditions and inter-block couplings. To the best of our knowledge, fully compatible operators for variable coefficients with $q$th order boundary closures and $(q+1)$th order boundary derivative operators are not yet available in the literature. We strongly encourage efforts to construct such operators. Until they become available, we resort to so-called \emph{adapted} fully compatible operators, which can be constructed from any set of compatible operators \cite{Duru2014} (see Section \ref{sec:sbp}). The adapted operators are identical to the original operators except at the first and last grid points, where the accuracy is reduced to $(q-1)$th order. By the general result in \cite{SvardNordstrom06}, we expect the $\ell^2$ error of pointwise stable schemes to be of order $\min(q_b+2,2q)$, where $q_b$ denotes the boundary accuracy. This implies that the adapted operators might yield up to one order lower convergence rates than the corresponding compatible operators. Remarkably, however, experiments with the IEWE in \cite{Duru201437} showed no loss in convergence rates. The adapted $2q=6$ operator even yielded smaller errors than the original operator. Although a theoretical explanation of this super convergence is currently lacking, the adapted operators seem attractive from a practical point of view. In this paper, we investigate how the adapted operators fare when applied to the AEWE.

The rest of this paper is organized as follows. We introduce notational conventions in Section \ref{sec:notation}. In Section \ref{sec:elasticity}, we review the equations of linear anisotropic elasticity and discuss how they change under coordinate transformations. We introduce compatible and adapted fully compatible SBP operators in Section \ref{sec:sbp}, and combine them with proper SATs to construct energy-stable self-adjoint schemes for Robin and displacement boundary conditions in Section \ref{sec:bc}. In Section \ref{sec:interfaces}, we derive SATs for grid-block couplings. Numerical experiments are presented in Section \ref{sec:num_exp}. We evaluate the convergence rates of the new multi-block SBP-SAT scheme against a manufactured solution and show the applicability of the scheme in elastodynamic cloaking and seismic imaging of the Earth. Conclusions follow in Section \ref{sec:conclusions}.

\section{Notation conventions} \label{sec:notation}
Let $\Omega \subset \R^d$ denote a bounded domain in $d$ dimensions and let $u,v \in L^2(\Omega)$. We use the $L^2$ inner product:
\begin{equation}
  \volintphys{u}{v} = \myint{\physdom}{}{uv}{\physdom}.
\end{equation}
Similarly, we use the notation
\begin{equation}
  \surfintphys{u}{v} = \myint{\physboundary}{}{uv}{S}
\end{equation}
for surface integrals.
Note, however, that $\surfintphys{\cdot}{\cdot}$ is not an inner product but a bilinear form. We use the summation convention for repeated subscript indices so that
\begin{equation}
  u_i v_i = \sum \limits_{i=1}^d u_i v_i.
\end{equation}
The summation convention applies to inner products too, i.e.,
\begin{equation}
  \ip{u_i}{v_i}{} = \sum \limits_{i=1}^d \ip{u_i}{v_i}{} .
\end{equation}
The summation convention only applies to indices $i,j,k,\ell,m,\il,\jl,\kl,$ and $\elll$. In particular, it does not apply to $x$, $\gamma$, or $N$.

Boldface font is reserved for vectors $\bfu$ whose elements approximate some scalar field $u$ evaluated on the grid. We will later define discrete inner products and use the summation convention in the discrete setting too, so that
\begin{equation}
  \ip{\bfu_i}{\bfv_i}{} = \sum \limits_{i=1}^d \ip{\bfu_i}{\bfv_i}{} .
\end{equation}

For all spatially variable coefficients, we use the same symbol also in the discrete case, which then is understood to denote a diagonal matrix with the values of that coefficient on the diagonal. The outward unit normals $\hat{n}$ and $\hat{\nu}$ (see Figure \ref{fig:transform}) are regarded as variable coefficients that take non-zero values only at boundary points. In the discrete setting, the values of $\hat{n}$ and $\hat{\nu}$ at edge and corner points change with context. When integrating over a face, $\hat{n}$ (or $\hat{\nu}$) is understood to denote the unit normal to that face even at edge and corner points. The same convention applies to the surface area scale factor $\surfjacobian$.

\section{Equations of linear elasticity} \label{sec:elasticity}
Let $\{\vec{E}_\is\}$ denote an orthonormal basis in $\R^d$, let $\vec{X} = X_\is \vec{E}_\is$, and let $\ddxi = \pd/\pd X_\is$. The generalized Hooke's law for an elastic medium relates stress to strain and reads
\begin{equation} \label{eq:hooke}
  \sigma_{\is \js} = \stiffphys_{\idx{IJKL}} \ddxk u_\ls,
\end{equation}
where $u_\ls$ is the displacement vector, $\sigma_{\is \js}$ is the stress tensor, and $\stiffphys_{\idx{IJKL}}$ is the elastic stiffness tensor. \refone{Note that all indices range from $1$ to $d$.} The stiffness tensor has the major symmetry
\begin{equation} \label{eq:symmetry_phys}
    \stiffphys_{\idx{IJKL}} = \stiffphys_{\idx{KLIJ}}.
\end{equation}
Normal elastic materials also have the minor symmetry
\begin{equation} \label{eq:symmetry_minor}
  \stiffphys_{\idx{IJKL}} = \stiffphys_{\idx{JIKL}},
\end{equation}
which implies that the stress tensor is symmetric, i.e., $\sigma_{\is \js} = \sigma_{\js \is}$. In this paper, we consider the more general theory of Cosserat elasticity \cite{Cosserat}, in which stress is not necessarily symmetric. That is, we do not assume that the stiffness tensor has the minor symmetry \eqref{eq:symmetry_minor}. Requiring a non-negative elastic strain energy density results in the condition
\begin{equation} \label{eq:definite_phys}
    S_{\idx{IJ}} \stiffphys_{\idx{IJKL}} S_{\idx{KL}} \geq 0 \quad \forall S_{\idx{IJ}},
\end{equation}
which we assume that $\stiffphys_{\idx{IJKL}}$ satisfies. The momentum balance reads
\begin{equation} \label{eq:newton}
  \rho \ddot{u}_{\js} = \ddxi \sigma_{\is \js} + f_\js,
\end{equation}
where $\rho$ is density and $f_\js$ denotes external body forces. Substituting Hooke's law \eqref{eq:hooke} in \eqref{eq:newton} yields the elastic wave equation for displacements,
\begin{equation}
\begin{array}{ll} \label{eq:wave_eq_general}
  \rho \ddot{u}_{\js} = \ddxi \stiffphys_{\idx{IJKL}} \ddxk u_{\ls} + f_\js, & \vec{X} \in \physdom, \\ 
  L_{\idx{IJ}} u_\js = 0, & \vec{X} \in \physboundary, \\ 
\end{array}
\end{equation}
where $\physdom \subset \R^d$ is a bounded domain with outward unit normal $\hat{n} = n_\is \vec{E}_\is$ and the linear operator $L_{\is \js}$ represents well-posed boundary conditions. The traction vector $\vec{\tau} = \tau_\js \vec{E}_\js$ acting on $\physboundary$ is
\begin{equation} \label{eq:traction_def}
  \tau_\js = n_\is \sigma_{\is \js} = n_\is \stiffphys_{\idx{IJKL}} \ddxk u_{\ls} .
\end{equation}
For future use we define the traction operator
\begin{equation} \label{eq:tractionop_cont_def}
  \ctractionop_{\js \ls} = n_\is \stiffphys_{\idx{IJKL}} \ddxk
\end{equation}
such that $\tau_\js = \ctractionop_{\js \ls} u_{\ls}$.

In the absence of external body forces, the energy method, which amounts to multiplying the first equation in \eqref{eq:wave_eq_general} by $\dot{u}_\js$ and integrating over $\physdom$, leads to
\begin{equation} \label{eq:phys_energy_ibp}
\begin{aligned}
    \volintphys{\dot{u}_\js}{\rho \ddot{u}_\js} &= \volintphys{\dot{u}_\js}{\ddxi \stiffphys_{\idx{IJKL}} \ddxk u_{\ls}} \\
    &= \surfintphys{\dot{u}_\js}{n_\is \stiffphys_{\idx{IJKL}} \ddxk u_{\ls}} - \volintphys{\ddxi \dot{u}_\js}{\stiffphys_{\idx{IJKL}} \ddxk u_{\ls}} \\
    &= \surfintphys{\dot{u}_\js}{\tau_\js} - \volintphys{\ddxi \dot{u}_\js}{\stiffphys_{\idx{IJKL}} \ddxk u_{\ls}},
\end{aligned}
\end{equation}
where we used integration by parts and the definition of $\tau_\js$.
We have
\begin{equation}
  \volintphys{\dot{u}_\js}{\rho \ddot{u}_\js} = \frac{1}{2} \dd{}{t} \volintphys{\dot{u}_\js}{\rho \dot{u}_\js}.
\end{equation}
The major symmetry of the stiffness tensor \eqref{eq:symmetry_phys} yields
\begin{equation}
  \volintphys{\ddxi \dot{u}_\js}{\stiffphys_{\idx{IJKL}} \ddxk u_{\ls}} = \volintphys{\ddxk \dot{u}_{\ls}}{\stiffphys_{\idx{IJKL}} \ddxi u_\js} = \frac{1}{2} \dd{}{t} \volintphys{\ddxi u_\js}{\stiffphys_{\idx{IJKL}} \ddxk u_{\ls}}.
\end{equation}
The total energy $\mathcal{E}$ is the sum of kinetic and strain energy,
\begin{equation} \label{eq:energy_cont}
  \mathcal{E} = \frac{1}{2} \volintphys{\dot{u}_\js}{\rho \dot{u}_\js} + \frac{1}{2} \volintphys{\ddxi u_\js}{\stiffphys_{\idx{IJKL}} \ddxk u_{\ls}}.
\end{equation}
The positive semidefiniteness of the stiffness tensor \eqref{eq:definite_phys} ensures that the strain energy is non-negative. Rearranging terms in \eqref{eq:phys_energy_ibp} leads to the energy rate
\begin{equation} \label{eq:energy_rate_phys}
    \dd{\mathcal{E}}{t} = \surfintphys{\dot{u}_\js}{\tau_\js}.
\end{equation}
We note that homogeneous displacement boundary conditions ($u_\js = 0$) and homogeneous traction boundary conditions ($\tau_\js = 0$) both yield energy conservation.

\otherchange{
\subsection{Wave speeds in anisotropic solids} \label{sec:plane_wave}
A plane wave propagating in unit direction $\vec{\xi}$ can be described by the equivalent expressions
\begin{equation}
  u_\js = U_\js e^{i(k_\is X_\is - \varphi t)} = U_\js e^{ik (\xi_\is X_\is - vt)} = U_\js e^{i \varphi(s_\is X_\is - t)},
\end{equation}
where $U_\js$ is the polarization vector, $k_\is=k \xi_\is$ is the wave vector, $\varphi$ is the angular frequency, $v$ is the phase velocity, and $s_\is = k_\is/\varphi$ is the slowness vector.
Assuming a homogeneous solid and no external forces, inserting the plane wave ansatz into the elastic wave equation (the first equation in \eqref{eq:wave_eq_general}) yields the Christoffel equation \cite{Synge1956,Achenbach}
\begin{equation} \label{eq:christoffel}
  \left(v^2 \delta_{\js \ls} - \rho^{-1} \xi_\is C_{\is \js \ks \ls} \xi_\ks\right) U_{\ls} = 0.
\end{equation}
For nontrivial solutions to exist we must have $\det\left(v^2 \delta_{\js \ls} - \rho^{-1} \xi_\is C_{\is \js \ks \ls} \xi_\ks\right)=0$, which is the dispersion relation. The phase velocities are the positive square roots of the eigenvalues of the operator $\rho^{-1}\xi_\is C_{\is \js \ks \ls} \xi_\ks$. The eigenvalues depend on the direction of propagation, $v = v(\vec{\xi})$. In two spatial dimensions there are two body-wave solutions to \eqref{eq:christoffel}: the quasi-P-wave and the quasi-S-wave. In the isotropic case these waves reduce to the P- and S-waves.

In Section \ref{sec:cloaking} we will use slowness surfaces to illustrate the properties of an anisotropic medium. Slowness surfaces are drawn in the slowness vector space and satisfy
\begin{equation}
  s_\is = \frac{\xi_\is}{v(\vec{\xi})},
\end{equation}
where $v$ is one of the phase velocities. For isotropic solids, the phase velocities are direction-independent and hence the slowness surfaces are spherical (circular in two dimensions). The faster the wave, the smaller the radius of the slowness surface.
}

\subsection{Coordinate transformation} \label{sec:transform}
Let $\{\vec{e}_i\}$ denote an orthonormal basis in $\R^d$ and let $\vec{x} = x_i \vec{e}_i$. Introduce a smooth one-to-one mapping $X_\is = X_{\is}(x_1,...,x_d)$ from the reference domain $\refdom = [0, \;1]^d$ to the physical domain $\physdom$, as illustrated in Figure \ref{fig:transform}.
\begin{figure}[h]
    \centering
    \includegraphics[width=0.55\linewidth, trim=0cm 0cm 0cm 0cm, clip=true]{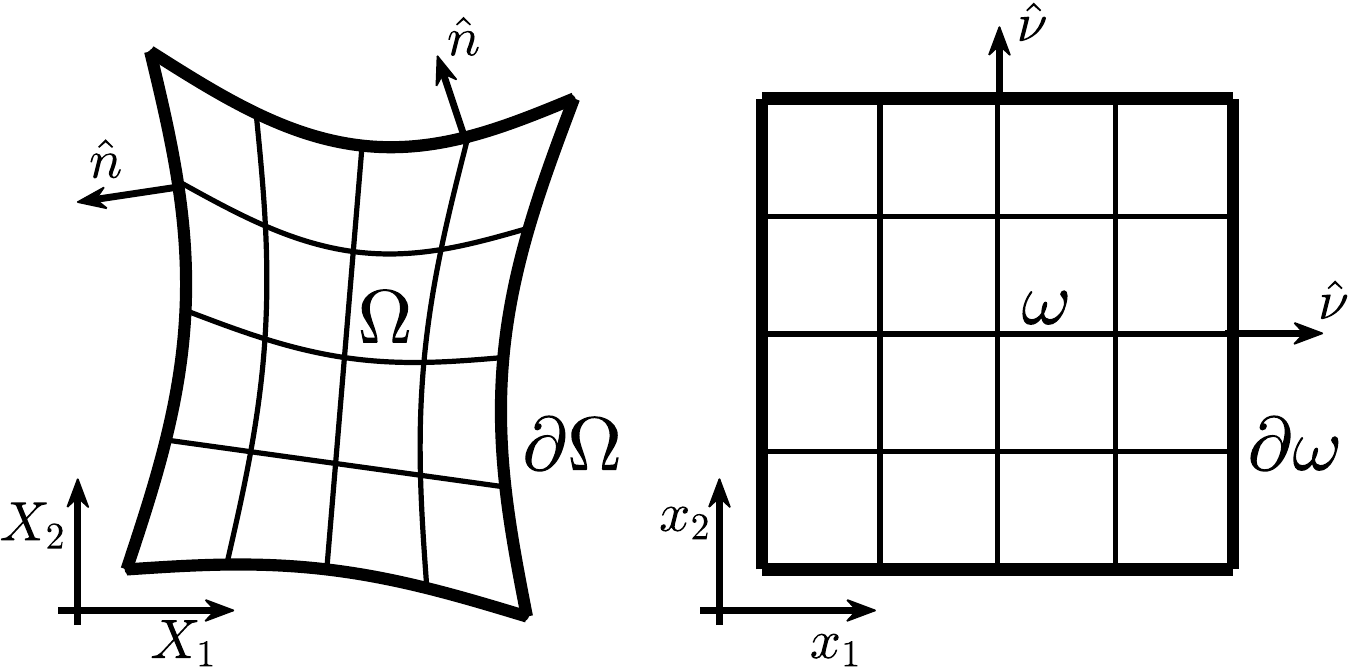}
    \caption{Schematic of the physical domain $\physdom$ and the reference domain $\refdom$}\label{fig:transform}
\end{figure}
We will use uppercase letters for quantities related to the physical domain and lowercase letters for similar quantities in the reference domain. We define $\ddxii = \pd / \pd x_i$. Let
\begin{equation}
  \K_{\is i} = \pd x_i / \pd X_{\is}
\end{equation}
denote the transformation gradient. \refone{Note that the object $\K_{\is i}$ is not a second order tensor because it maps from one domain to the other \cite{Norris2011}. By the chain rule,
\begin{equation} \label{eq:transf_chain}
    \ddxi = \K_{\is i} \ddxii.
\end{equation}
Further, let
\begin{equation}
  J = \det[(\K^{-1})_{i \is} ]
\end{equation}
denote the Jacobian determinant of the mapping from $\refdom$ to $\physdom$.} We assume $J>0$.
The following metric identities are well known (see \cite{thompson1985numerical}):
\begin{equation} \label{eq:identity}
    J \K_{\is i} \ddxii = \ddxii J \K_{\is i}.
\end{equation}
Let $\vec{a}_i$ denote the covariant basis vectors:
\begin{equation} \label{eq:cov_base_def}
  \vec{a}_i = \ddi \vec{X} = \ddi X_{\is} \vec{E}_{\is} = (\K^{-1})_{i \is} \vec{E}_{\is}.
\end{equation}

\subsubsection{Transforming the PDE}
Using first \eqref{eq:transf_chain} and then \eqref{eq:identity}, we have
\begin{equation} \label{eq:cont_transf_operator_first}
\begin{aligned}
    \ddxi \stiffphys_{\idx{IJKL}} \ddxk  &= \K_{\is i} \ddxii \stiffphys_{\idx{IJKL}} \K_{\ks k} \ddxik = J^{-1} \ddxii \K_{\is i} J \stiffphys_{\idx{IJKL}} \K_{\ks k} \ddxik .
\end{aligned}
\end{equation}
Introduce a change of variables
\begin{equation}
  u_\is = \gauge_{\is i} u_i, \quad \tau_\is = \gauge_{\is i} \tau_i,
\end{equation}
for some $A_{\is i}$ to be discussed later. We can now write the equations of motion in \eqref{eq:wave_eq_general} as
\begin{equation}
  J \rho  \ddot{u}_j = (\gauge^{-1})_{j \js} \ddxii \K_{\is i} J \stiffphys_{\idx{IJKL}} \K_{\ks k} \ddxik \gauge_{\ls \ell} u_{\ell} + J (\gauge^{-1})_{j\js} f_{\js}.
\end{equation}
In this paper we will use the trivial change of variables
\begin{equation}
  \gauge_{\is i} = \delta_{\is i},
\end{equation}
which yields the equations of motion
\begin{equation}
  J \rho  \ddot{u}_j = \ddxii \K_{\is i} J \stiffphys_{\is j \ks \ell} \K_{\ks k} \ddxik u_{\ell} + J f_j.
\end{equation}
Define the transformed density and stiffness tensor
\begin{equation} \label{eq:transf_density_stiffness}
  \rhoref = J \rho, \quad \stiffref_{i j k \ell} = \K_{\is i} J \stiffphys_{\is j \ks \ell} \K_{\ks k}.
\end{equation}
The transformed equation, posed on the unit cube $\refdom$, reads
\begin{equation} \label{eq:model_problem_transf}
\begin{array}{ll}
  \rhoref \ddot{u}_j = \ddi \stiffref_{i j k \ell} \ddk u_{\ell} + Jf_j, & \vec{x} \in \refdom, \\ 
  \Lambda_{i j} u_j = 0, & \vec{x} \in \refboundary, \\ 
\end{array}
\end{equation}
where $\Lambda_{ij}$ denotes the transformation of $L_{\is \js}$. Using the definition of $\stiffref_{ijk\ell}$ \refone{and} \eqref{eq:cont_transf_operator_first} shows that
\begin{equation} \label{eq:cont_transf_operator_clean}
  \ddxi \stiffphys_{\idx{IJKL}} \ddxk =  J^{-1} \ddxii \stiffref_{i \js k \ls} \ddxik.
\end{equation}
In Section \ref{sec:sbp} we use formula \eqref{eq:cont_transf_operator_clean} to construct an SBP operator that approximates $\ddxi \stiffphys_{\idx{IJKL}} \ddxk$.

The transformed stiffness tensor retains the major symmetry,
\begin{equation} \label{eq:C_symmetry}
\begin{aligned}
   \stiffref_{k \ell i j} = \K_{\is k} J \stiffphys_{\is \ell \ks j} \K_{\ks i} = \K_{\is k} J \stiffphys_{\ks j \is \ell} \K_{\ks i}
    = \K_{\ks k} J \stiffphys_{\is j \ks \ell} \K_{\is i} = \stiffref_{i j k \ell},
\end{aligned}
\end{equation}
and the semidefiniteness
\begin{equation} \label{eq:C_semidef}
    s_{i j} \stiffref_{i j k \ell} s_{k \ell} = \underbrace{ s_{i j} \K_{\is i} }_{=:S_{\is j}} J \stiffphys_{\is j \ks \ell} \underbrace{\K_{\ks k} s_{k \ell}}_{=:S_{\ks \ell}} \geq 0 \quad \forall s_{ij},
\end{equation}
where we used the semidefiniteness of $\stiffphys_{\idx{IJKL}}$ \eqref{eq:definite_phys} and the positivity of $J$. We conclude that the transformed PDE is of the same form as the original PDE in \eqref{eq:wave_eq_general}. However, even if $\stiffphys_{\idx{IJKL}}$ has the minor symmetry \eqref{eq:symmetry_minor}, the transformed stiffness tensor generally does not, because
\begin{equation}
\begin{aligned}
  \stiffref_{i j k \ell} - \stiffref_{j i k \ell} &= \K_{\is i} J \stiffphys_{\is j \ks \ell} \K_{\ks k} - \K_{\is j} J \stiffphys_{\is i \ks \ell} \K_{\ks k} \\
  &=  \left(\K_{\is i} \stiffphys_{\is j \ks \ell} - \K_{\is j} \stiffphys_{\is i \ks \ell} \right) J \K_{\ks k} \\
  &=  \left(\K_{\is i} \stiffphys_{\is j \ks \ell} - \K_{\is j} \stiffphys_{i \is \ks \ell} \right) J \K_{\ks k} ,
\end{aligned}
\end{equation}
which is nonzero, in general. Hence, the equations of Cosserat materials are invariant under coordinate transformations, but the equations of normal materials are not. It is, however, possible to symmetrize the effective transformed stress tensor by setting (see \cite{Norris2011} for a thorough discussion of coordinate transformations in elastic wave equations)
\begin{equation}
  \gauge_{\is i} = \K_{\is i}.
\end{equation}
This approach introduces additional terms in the transformed equations of motion, similar to those required for Willis materials \cite{Milton2006,Milton2007}, and will not be pursued in the present study.

In the semidiscrete stability proof we will make use of the property
\begin{equation} \label{eq:C_semidef_sum}
  u_j \stiffref_{mjm\ell} u_{\ell} \geq 0 \quad \forall  u_j,
\end{equation}
which follows from \eqref{eq:C_semidef}, because
\begin{equation}
  u_j \stiffref_{mjm\ell} u_{\ell} = \underbrace{u_j \delta_{im}}_{=:U_{ijm}} \stiffref_{ijk\ell} \underbrace{u_{\ell} \delta_{km}}_{=:U_{k \ell m}} =\sum \limits_{m} \underbrace{  U_{ijm} \stiffref_{ijk\ell} U_{k \ell m} }_{\geq 0 \; \forall m } \geq 0 .
\end{equation}

\subsubsection{Integrals and normals}
Since $J \mathrm{d} \refdom$ is the volume element, we have $\mathrm{d} \physdom = J \mathrm{d} \refdom$, and hence
\begin{equation} \label{eq:volint_J}
  \volintphys{u}{v} = \volint{u}{Jv}.
\end{equation}
Similarly, we let $\surfjacobian$ denote the surface area scale factor such that
\begin{equation} \label{eq:surfint_J}
 \surfintphys{u}{v} = \surfint{u}{\surfjacobian v}.
\end{equation}
The surface area scale factor $\surfjacobian$ is related to the covariant basis vectors $\vec{a}_i$ defined in \eqref{eq:cov_base_def} as follows.
In two space dimensions
\begin{equation}
  \surfjacobian = \abs{\vec{a}_i}, \quad x_{j} \in \{0, 1\}, \quad  i, j \mbox{ cyclic},
\end{equation}
and in three space dimensions
\begin{equation}
  \surfjacobian = \abs{\vec{a}_i \times \vec{a}_j} , \quad x_k \in \{0,1\}, \quad  i, j, k \mbox{ cyclic}.
\end{equation}
Let $\hat{\nu} = \nu_i \vec{e}_i$ denote the unit normal to $\refdom$. The normals $\hat{n}$ and $\hat{\nu}$ are related by Nanson's formula \cite{Malvern},
\begin{equation} \label{eq:nanson}
  \surfjacobian n_\is = J \K_{\is i} \nu_i.
\end{equation}

\subsection{Numerical approximation of the transformation gradient}
In this subsection we comment briefly on how numerical approximations of properties of the coordinate transformation may be computed. We compute an approximation $\Kapprox_{\is i} \approx \K_{\is i}$ of the transformation gradient by applying derivative approximations to a given grid. To retain the order of accuracy, $\Kapprox_{\is i}$ needs to be at least as accurate as the finite difference operators used to discretize the PDE. Higher-order approximations, or even the exact $\K_{\is i}$, if available, could also be used. For all numerical experiments in this paper, we compute $\Kapprox_{\is i}$ using first-derivative SBP operators of the same order as we use to solve the PDE. That is, $\Kapprox_{\is i}$ is computed to order $q$ near boundaries and order $2q$ in the interior.

Once $\Kapprox_{\is i}$ is computed, we use relations between the corresponding continuous quantities to define all other approximations. We set
\begin{equation}
  \refone{\Japprox = \det [(\Kapprox^{-1})_{i \is}],}
\end{equation}
\begin{equation}
  \stiffrefapprox_{ijk\ell} = \Kapprox_{\is i} \Japprox \stiffphys_{\is j \ks \ell} \Kapprox_{\ks k},
\end{equation}
\begin{equation}
  \munderbar{\vec{a}}_i = (\Kapprox^{-1})_{i \is} \vec{E}_{\is},
\end{equation}
\begin{equation} \label{eq:surfj_app_2d}
  \surfjacobianapprox = \abs{\munderbar{\vec{a}}_i}, \quad x_{j} \in \{0, 1\}, \quad  i, j \mbox{ cyclic}, \quad \mbox{(in 2D)},
\end{equation}
or
\begin{equation} \label{eq:surfj_app_3d}
  \surfjacobianapprox = \abs{\munderbar{\vec{a}}_i \times \munderbar{\vec{a}}_j} , \quad x_k \in \{0,1\}, \quad  i, j, k \mbox{ cyclic}, \quad \mbox{(in 3D)},
\end{equation}
and
\begin{equation}
  \munderbar{n}_\is = \surfjacobianapprox^{-1} \Japprox \Kapprox_{\is i} \nu_i.
\end{equation}
The only requirements for stability of the semidiscrete scheme (to be introduced later) are $\Japprox > 0$, $\stiffrefapprox_{ijk\ell} = \stiffrefapprox_{k\ell ij}$, $s_{i j} \stiffrefapprox_{i j k \ell} s_{k \ell} \geq 0 \; \forall  s_{ij}$, and $\surfjacobianapprox > 0$. We suggest checking the condition $\Japprox > 0$, which could be violated due to truncation errors. Assuming $\Japprox >0$, the remaining three conditions are guaranteed to be satisfied, regardless of how $\Kapprox_{\is i}$ was computed, because
\begin{equation}
\begin{aligned}
   \stiffrefapprox_{k \ell i j} = \Kapprox_{\is k} \Japprox \stiffphys_{\is \ell \ks j} \Kapprox_{\ks i} = \Kapprox_{\is k} \Japprox \stiffphys_{\ks j \is \ell} \Kapprox_{\ks i}
    = \Kapprox_{\ks k} \Japprox \stiffphys_{\is j \ks \ell} \Kapprox_{\is i} = \stiffrefapprox_{i j k \ell},
\end{aligned}
\end{equation}
\begin{equation}
    s_{i j} \stiffrefapprox_{i j k \ell} s_{k \ell} = \underbrace{ s_{i j} \Kapprox_{\is i} }_{=:S_{\is j}} \Japprox \stiffphys_{\is j \ks \ell} \underbrace{\Kapprox_{\ks k} s_{k \ell}}_{=:S_{\ks \ell}} \geq 0 \quad \forall s_{ij},
\end{equation}
and $\surfjacobianapprox > 0$ follows from formulas \eqref{eq:surfj_app_2d} and \eqref{eq:surfj_app_3d}, combined with the assumption $\Japprox > 0$, which implies that $\Kapprox_{\is i}$ is nonsingular and thus guarantees $\munderbar{\vec{a}}_i \neq \vec{0}$.

Note that since we used Nanson's formula \eqref{eq:nanson} to define $\munderbar{\hat{n}}$, Nanson's formula holds identically for the approximated quantities. We conclude that $\Kapprox_{\is i}$ may be computed with any sufficiently accurate method, as long as the resulting Jacobian is positive.
With a slight abuse of notation, we henceforth drop the underline notation and let it be implied that we may be dealing with approximations in the discrete setting.

\subsubsection{The transformed stiffness tensor of isotropic materials}
Isotropic materials are characterized by the two Lam{\'e} parameters $\lambda$ and $\mu$ and have the stiffness tensor
\begin{equation}
  \stiffphys_{\idx{IJKL}} = \lambda \dijx \dklx + \mu \left( \dikx \djlx + \dilx \djkx \right).
\end{equation}
The isotropic stiffness tensor transforms into
\begin{equation}
\begin{aligned}
  \stiffref_{i j k \ell} &= \K_{\is i} J \stiffphys_{\is j \ks \ell} \K_{\ks k} = \K_{\is i} J \left[ \lambda \delta_{\is j} \delta_{\ks \ell} + \mu \left( \dikx \djl + \delta_{\is \ell} \delta_{j \ks} \right) \right] \K_{\ks k} \\
  &=  J \left[ \lambda \K_{j i} \K_{\ell k} + \mu \left( \K_{\ks i} \djl \K_{\ks k} + \K_{\ell i} \K_{j k} \right) \right] .
\end{aligned}
\end{equation}
In 3D, there are 9 independent parameters in $\K_{\is i}$, which leads to a total of 11 independent parameters in $\stiffref_{i j k \ell}$. In general, the transformed stiffness tensor does not have the minor symmetry even in the isotropic case, because
\begin{equation}
\begin{aligned}
  \stiffref_{i j k \ell} - \stiffref_{j i k \ell} = &J \lambda \left(\K_{ji} - \K_{ij} \right) J\K_{\ell k} \\
  + &J \mu \left[ \left(\K_{\ks i} \djl - \K_{\ks j} \dil\right) \K_{\ks k} + \K_{\ell i} \K_{j k} - \K_{\ell j} \K_{i k} \right],
\end{aligned}
\end{equation}
which is nonzero in general.

\section{Summation-by-parts operators} \label{sec:sbp}
Most of the definitions in this section are not new but are restated here for completeness. The notation follows \cite{AlmquistDunham2020} closely. We consider only diagonal-norm SBP operators. That is, the so-called norm matrix $H_x$ has the structure
\begin{equation} \label{eq:structure_H}
    H_{x} = \mbox{diag}(h_1, h_2, \ldots, h_2, h_1),
\end{equation}
where all $h_i$ are proportional to the grid spacing $h$. The first-derivative SBP operators $D_x \approx \partial_x$ have the integration-by-parts-mimicking property
\begin{equation}
  H_x D_x = -D_x^T H_x - e_{0} e_{0}^T + e_N e_N^T ,
\end{equation}
where the vectors $e_0$ and $e_N$ interpolate or extrapolate to the left and right boundaries, respectively. We herein restrict our attention to grids that include the boundary points of the interval $[x_{L}, x_{R}]$, in which case one may set
\begin{equation} \label{eq:el_er}
e_{0} = \begin{bmatrix} 1, 0 , \ldots, 0\end{bmatrix}^T, \quad e_{N} = \begin{bmatrix} 0, \ldots, 0, 1 \end{bmatrix}^T.
\end{equation}
 We will use the first-derivative operators presented in \cite{MattssonNordstrom04}, which (for orders $2q \geq 6$) correspond to a particular choice of the free parameters in the operators developed in \cite{KreissScherer74,SchererThesis77,OlssonThesis,Strand94}. \refone{These operators have a repeating interior stencil of order $2q$ and boundary closures of order $q$.} The \emph{compatible} narrow-stencil second-derivative operators $D_{xx}(b) \approx \partial_x b \partial_x$ derived in \cite{Mattsson11} are based on the same norm matrix $H_x$ and have the property
\begin{equation} \label{eq:sbp_property_2nd_der}
   H_x D_{xx}(b) = -D_x^T H_x b D_x - R_{xx}(b) - e_{0} e_{0}^T b \db_x + e_N e_N^T b \db_x,
\end{equation}
where the first and last rows of $\db_x$ approximate the first derivative and the interior of $\db_x$ is zero ($\db_x$ was denoted $S$ in \cite{Mattsson11}). \refone{Just like $D_x$, $D_{xx}$ is $q$th order accurate in the boundary closures and $2q$th order accurate in the interior.}
Note that for the SBP operators derived in \cite{Mattsson11}, $e_{0,N}^T D_x \neq e_{0,N}^T \db_x$. If $e_{0,N}^T D_x= e_{0,N}^T\db_x$, then the SBP operators $D_x$ and $D_{xx}$ are said to be \emph{fully compatible} \cite{MattssonParisi09}. The SBP operators derived in \cite{Mattsson11} have $\db_x$ that are accurate of order $q+1$, i.e., one order higher than the boundary closure of $D_x$.

The matrix $R_{xx}(b)$ is symmetric positive semidefinite and consists of undivided difference approximations in such a way that $\bfu^T R_{xx}(b) \bfv$ is zero to order $2q$ \cite{Mattsson11}. Its structure is
\begin{equation}
    R_{xx}(b) =  \sa h^{2\alpha - 2} D_{x^\alpha}^T E_{\alpha}^T H_x B_{\alpha}(b) E_{\alpha} D_{x^\alpha},
\end{equation}
where $\alpha \geq q + 1$; $D_{x^\alpha} \approx \partial^{\alpha}/\partial x^\alpha$; the $E_\alpha$ are of order 1; and the $B_{\alpha}$ are diagonal matrices whose entries are convex combinations of $b(x)$ evaluated on the grid. Let $b_s$ denote $b$ evaluated at the $s$th grid point, and let $(B_{\alpha})_s$ denote the entry in $B_{\alpha}$ associated with the $s$th grid point. The structure of $B_{\alpha}(b)$ is
\begin{equation} \label{eq:B_structure_1d}
  \left(B_{\alpha}(b) \right)_r = \sess \beta_{\alpha,r,s} b_{s}, \quad \beta_{\alpha,r,s} \geq 0 \; \forall \alpha,r,s.
\end{equation}
To simplify the notation we define
\begin{equation}
   \mathcal{D}_{x^\alpha} = h^{\alpha-1} H_x^{1/2} E_{\alpha} D_{x^\alpha}
\end{equation}
such that
\begin{equation}
  R_{xx}(b) =  \sa \mathcal{D}_{x^\alpha}^T  B_{\alpha}(b) \mathcal{D}_{x^\alpha}.
\end{equation}
For future use we prove the following lemma, which states that $R_{xx}$ preserves semidefiniteness of two-tensors.
\begin{lemma} \label{lemma:R_two_tensor}
If $u_{i} S_{ij} u_{j} \geq 0 \; \forall u_{i}$, then
\begin{equation}
  \bfu_{i} R_{xx}(S_{ij}) \bfu_{j} \geq 0 \; \forall \bfu_{i} .
\end{equation}
\begin{proof}
\begin{equation*}
\begin{aligned}
  \bfu_{i} R_{xx}(S_{ij}) \bfu_{j} &= \sa \left( \mathcal{D}_{x^\alpha} \bfu_{i} \right)^T  B_{\alpha}(S_{ij}) \mathcal{D}_{x^\alpha}  \bfu_{j}\\
   &= \sar \left( \mathcal{D}_{x^\alpha} \bfu_{i} \right)_r  \left(B_{\alpha}(S_{ij}) \right)_r \left( \mathcal{D}_{x^\alpha}  \bfu_{j} \right)_r \quad [\mbox{Use \eqref{eq:B_structure_1d}}] \\
   &= \sars \left( \mathcal{D}_{x^\alpha} \bfu_{i} \right)_r  \beta_{\alpha,r,s} (S_{ij})_{s} \left( \mathcal{D}_{x^\alpha}  \bfu_{j} \right)_r \quad [\mbox{Use $u_{i} S_{ij} u_{j} \geq 0$}] \\
   & \geq 0.
\end{aligned}
\end{equation*}
\end{proof}
\end{lemma}
In particular, Lemma \ref{lemma:R_two_tensor} shows that $R_{xx}$ preserves the semidefiniteness of the two-tensor $\stiffref_{mjm\ell}$ (cf.\ \eqref{eq:C_semidef_sum}):
\begin{equation}
  \bfu_{j} R_{xx}(\stiffref_{mjm\ell}) \bfu_{\ell} \geq 0 \; \forall \bfu_{j}.
\end{equation}

\subsection{Adapted fully compatible SBP operators}
Any compatible second-derivative operator can be turned into a fully compatible operator, here denoted $D_{xx}^{\idx{F}\idx{C}}$, by simply replacing the boundary derivatives $\db_x$ by $D_x$ \cite{Duru201437}. We refer to such operators as \emph{adapted} fully compatible operators. For the operators derived in \cite{Mattsson11}, swapping boundary derivatives amounts to adding terms of order $q-1$ at the grid end points,
\begin{equation}
  D_{xx}^{\idx{F}\idx{C}} = D_{xx} + \underbrace{ H_x^{-1} \left( e_{0} e_{0}^T b (\db_x - D_x) \right) }_{\mathcal{O}(h^{q-1})} - \underbrace{ H_x^{-1} \left( e_{N} e_{N}^T b (\db_x - D_x) \right) }_{\mathcal{O}(h^{q-1})} .
\end{equation}
Hence, the adapted fully compatible operators are one order less accurate than the original operators at precisely one grid point at each boundary. It is not obvious how the local reduction in accuracy affects the global convergence rate. A pessimist would expect reduction by a full order, but \cite{Duru201437} did not observe any reduction for isotropic elasticity. Our numerical experiments in Section \ref{sec:num_exp} indicate a reduction by half an order for orders $2q=4$ and $2q=6$, and no reduction for $2q=2$, for anisotropic materials.

In the following derivations we shall assume fully compatible operators. This assumption greatly simplifies the stability proofs (for a discussion on how non-fully compatible operators complicate the stability proofs for the acoustic wave equation, see \cite{AlmquistDunham2020}). In all numerical experiments we will use the adapted fully compatible operators.

\subsection{Positivity properties}
\refone{To prove stability for displacement boundary conditions and interface couplings in subsequent sections, we shall need to bound certain discrete volume integrals from below by discrete surface integrals. We refer to such bounds as \emph{positivity properties}. All positivity properties in this paper follow from the structure of the discrete quadrature $H_x$.}
It follows immediately from \eqref{eq:structure_H} and \eqref{eq:el_er} that we have
\begin{equation} \label{eq:pos_H_1d}
    H_{x} = \mbox{diag}(0,h_2, \ldots, h_2,0) + h_1 e_{0} e_{0}^T + h_1 e_{N} e_{N}^T  \geq h_1 e_{0} e_{0}^T + h_1 e_{N} e_{N}^T,
\end{equation}
or, equivalently,
\begin{equation} \label{eq:borrowing_H_1d}
   \bfu^T H_x \bfu \geq  h_1 (e_{0}^T \bfu)^2 + h_1 (e_{N}^T \bfu)^2 \quad \forall \bfu.
\end{equation}

\subsection{Multi-dimensional first-derivative operators}
Let operators with subscripts $x_i$ denote one-dimensional operators corresponding to coordinate direction $x_i$. The multi-dimensional first derivatives $D_i \approx \ddi $ are constructed using tensor products:
\begin{equation}
  \reftwo{D_i = I_{x_1} \otimes \cdots \otimes I_{x_{i-1}} \otimes D_{x_i} \otimes I_{x_{i+1}} \otimes \cdots \otimes I_{x_{d}} ,}
\end{equation}
where the $I_{x_i}$ are one-dimensional identity matrices of appropriate sizes. In analogy with the chain rule \eqref{eq:transf_chain}, we define
\begin{equation} \label{eq:chain_discrete}
    D_{\is} = \K_{\is i} D_i,
\end{equation}
where $D_{\is} \approx \ddxi$. Note that in the discrete setting, $\K_{\is i}$ is to be interpreted as a
diagonal matrix holding the grid-point values of the continuous coefficient $\K_{\is i}$ for each fixed $\il$ and $i$. Similarly, $D_i$ is a matrix for each fixed $i$. The implied summation in $\K_{\is i} D_i$ hence adds matrices in $\R^{N\times N}$, where $N$ denotes the total number of grid points, not elements of such matrices.

The multi-dimensional quadrature is
\begin{equation}
  H = H_{x_1} \otimes \cdots \otimes H_{x_d} .
\end{equation}
Let $\refboundary_{i}^-$ and $\refboundary_i^+$ denote the boundary faces where $x_i = 0$ and $x_i = 1$, respectively.
For integration over boundary faces, we define
\begin{equation}
  H_{\refboundary_i} = H_{x_1} \otimes \cdots \otimes H_{x_{i-1}} \otimes H_{x_{i+1}} \otimes \cdots \otimes H_{x_d}.
\end{equation}
Note that $H_{\refboundary_i}$ can be used to integrate over $\refboundary_i^+$ as well as $\refboundary_i^-$. For discrete integration over the volume, we define
\begin{equation}
  \dvolint{\bfu}{\bfv} = \bv{u}^T H \bv{v} . 
\end{equation}
We use the same inner product notation as in the continuous case without risk of confusion since the boldface font denotes discrete solution vectors.

Let $e_{f}^T$ denote a restriction operator that picks out only those solution values that reside on the face $f$. For discrete integration over the face $\refboundary_i^+$, for example, we write
\begin{equation}
  \dsurfintindexsign{\bv{u}}{\bv{v} }{i}{+} = (e_{\refboundary_i^+}^T \bv{u})^T H_{\refboundary_i} (e_{\refboundary_i^+}^T \bv{v}).
\end{equation}
Let $\setoffaces$ denote the set of all faces of $\refdom$,
\begin{equation} \label{eq:set_of_faces}
    \setoffaces = \{\refboundary_1^-,\ldots,\refboundary_d^-,\refboundary_1^+, \ldots, \refboundary_d^+\}.
\end{equation}
For integration over the entire boundary $\refboundary$, we define
\begin{equation}
    \dsurfint{\bfu}{\bfv} = \sum \limits_{f\in \setoffaces} \ip{\bfu}{\bfv}{f},
\end{equation}
 i.e., the integration is performed over one face at a time. If the integrand contains the unit normal or the scale factor $\surfjacobian$, their values at edges and corners are defined to be the same as on the remainder of that face.
In analogy with \eqref{eq:volint_J} and \eqref{eq:surfint_J}, we define
\begin{equation} \label{eq:discr_volint_curve}
         \dvolintphys{\bfu}{\bfv} = \dvolint{\bfu}{J \bfv}
 \end{equation}
 and
 \begin{equation} \label{eq:discr_surfint_curve}
     \dsurfintphys{\bfu}{\bfv} = \dsurfint{\bfu}{\surfjacobian \bfv}.
 \end{equation}
With the notation established in this section, we have the discrete integration-by-parts formula
\begin{equation} \label{eq:sbp_multid_mixed}
  \dvolint{\bfu }{D_i b D_j \bfv} = \dsurfint{\bfu }{\nu_i b D_j \bfv} - \dvolint{D_i \bfu }{b D_j \bfv}.
\end{equation}

\subsection{Multi-dimensional narrow-stencil second-derivative operators}
For any fixed $i$, we construct
\begin{equation}
  D_{ii}^{\idx{F}\idx{C}}(b) \approx \ddi b \ddi  \quad \mbox{(no sum over $i$)},
\end{equation}
by using the one-dimensional operator $D^{\idx{F}\idx{C}}_{xx}$ for each grid line. The multi-dimensional fully compatible SBP property for the second derivative that follows is
\begin{equation} \label{eq:sbp_multid_narrow}
\begin{aligned}
  \dvolint{\bfu }{D_{ii}^{\idx{F} \idx{C}}( b ) \bfv} &= \dsurfint{\bfu }{\nu_i b D_i \bfv} - \dvolint{D_i \bfu }{b D_i \bfv} - \bfu^T R_{ii}(b) \bfv \\
  & \quad \mbox{(no sum over $i$)},
\end{aligned}
\end{equation}
where the $R_{ii}$ matrices are multi-dimensional versions of $R_{xx}$. \reftwo{More precisely, the operator $H_{\refboundary_i}^{-1} R_{ii}(b)$ (no sum over $i$), is the operator that applies $R_{xx}(b)$ to each grid line in the $i$th coordinate direction. The $R_{ii}$ operators} inherit the symmetry and semidefiniteness-preserving properties of $R_{xx}$. In particular,
\begin{equation}
  R_{ii}(b) = R_{ii}^T(b) \quad \mbox{(no sum over $i$)}
\end{equation}
and
\begin{equation} \label{eq:semidef_R_multid}
  \bfu_{j} R_{ii}(\stiffref_{mjm\ell}) \bfu_{\ell} \geq 0 \; \forall \bfu_{j} \quad \mbox{(no sum over $i$)}.
\end{equation}

\subsection{Multi-dimensional positivity properties}
\refone{We here extend the one-dimensional positivity property \eqref{eq:borrowing_H_1d} to multiple dimensions}.
 To suppress unnecessary notation, we assume that the grid spacing in the reference domain is the same in each dimension (the analysis does not rely on this assumption). It follows from \eqref{eq:borrowing_H_1d}
that (see \cite{AlmquistDunham2020})
\begin{equation} \label{eq:borrowing_nd_H_local}
    \dvolint{\bfs_{ij}}{\stiffref_{ijk\ell} \bfs_{k\ell}} \geq h_1 \left( \ip{\bfs_{ij}}{\stiffref_{ijk\ell} \bfs_{k\ell}}{\refboundary_m^+} + \ip{\bfs_{ij}}{\stiffref_{ijk\ell} \bfs_{k\ell}}{\refboundary_m^-} \right) , 
\end{equation}
for $m = 1,\ldots,d$.
Using \eqref{eq:borrowing_nd_H_local} we can derive
\begin{equation}
\begin{aligned}
    \dvolint{\bfs_{ij}}{\stiffref_{ijk\ell} \bfs_{k\ell}} &= \frac{1}{d} \sum \limits_{m=1}^d \dvolint{\bfs_{ij}}{\stiffref_{ijk\ell} \bfs_{k\ell}} \\
    &\geq \frac{1}{d} \sum \limits_{m=1}^d h_1 \left( \ip{\bfs_{ij}}{\stiffref_{ijk\ell} \bfs_{k\ell}}{\refboundary_m^+} + \ip{\bfs_{ij}}{\stiffref_{ijk\ell} \bfs_{k\ell}}{\refboundary_m^-} \right) \\
    &= \frac{h_1}{d} \dsurfint{\bfs_{ij}}{\stiffref_{ijk\ell} \bfs_{k\ell}},
\end{aligned}
\end{equation}
which we summarize as
\begin{equation} \label{eq:borrowing_nd_H_global}
    \dvolint{\bfs_{ij}}{\stiffref_{ijk\ell} \bfs_{k\ell}} \geq \frac{h_1}{d} \dsurfint{\bfs_{ij}}{\stiffref_{ijk\ell} \bfs_{k\ell}} .
\end{equation}
Using \eqref{eq:borrowing_nd_H_global}, we can derive a similar property for integrals in the physical domain,
\begin{equation}
\begin{aligned}
  \dvolintphys{\bfs_{\is \js}}{\stiffphys_{\is \js \ks \ls} \bfs_{\ks \ls}} &= \dvolint{\bfs_{\is \js}}{J\stiffphys_{\is \js \ks \ls} \bfs_{\ks \ls}} \geq \frac{h_1}{d} \dsurfint{\bfs_{\is \js}}{J\stiffphys_{\is \js \ks \ls} \bfs_{\ks \ls}} \\
  &= \frac{h_1}{d} \dsurfintphys{\bfs_{\is \js}}{\frac{J}{\surfjacobian} \stiffphys_{\is \js \ks \ls} \bfs_{\ks \ls}},
\end{aligned}
\end{equation}
which we summarize as
\begin{equation} \label{eq:borrowing_phys}
  \dvolintphys{\bfs_{\is \js}}{\stiffphys_{\is \js \ks \ls} \bfs_{\ks \ls}} \geq \frac{h_1}{d} \dsurfintphys{\bfs_{\is \js}}{\surfjacobian^{-1} J \stiffphys_{\is \js \ks \ls} \bfs_{\ks \ls}}.
\end{equation}

\subsection{Combining narrow-stencil derivatives and mixed derivatives}
To discretize a term such as $\ddi b \ddj$, using narrow-stencil second derivatives when possible, we define the operator $\cD_{ij}$ as
\begin{equation} \label{eq:def_cd}
  \cD_{ij}(b) = \left\{ \begin{array}{cc} D^{\idx{F} \idx{C}}_{ij}(b), & i=j \\ D_i b D_j, & i\neq j \end{array} \right. .
\end{equation}
We use blackboard bold for discrete two-tensors such as $\cD_{ij}$ (where each tensor element is a square matrix).
Combining the two integration-by-parts formulas \eqref{eq:sbp_multid_mixed} and \eqref{eq:sbp_multid_narrow} leads to the integration-by-parts formula
\begin{equation} \label{eq:discr_ibp_multid_compatible}
\begin{aligned}
  &\dvolint{\bfu }{\cD_{ij}( b ) \bfv} = \\
  &\left\{
  \def\arraystretch{2.2}
  \begin{array}{ll} \displaystyle\dsurfint{\bfu }{\nu_i b D_j \bfv} - \dvolint{D_i \bfu }{b D_j \bfv}, & i\neq j \\
  \dsurfint{\bfu }{\nu_i b D_j \bfv}-\dvolint{D_i \bfu }{b D_j \bfv} - \bfu^T R_{ij}(b) \bfv, & i=j \end{array} \right. .
\end{aligned}
\end{equation}

\subsection{The discrete elastic operator}
The discrete operator that approximates $\ddi \stiffref_{ijk\ell} \ddk$
is $\cD_{ik}(\stiffref_{ijk\ell})$.
By \eqref{eq:discr_ibp_multid_compatible}, we have
\begin{equation} \label{eq:ibp_discr_operator}
\begin{aligned}
  \dvolint{\bfu_j}{\cD_{ik}( \stiffref_{ijk\ell}) \bfv_{\ell}} = &\dsurfint{\bfu_j}{\nu_i  \stiffref_{ijk\ell} D_k \bfv_{\ell}} - \dvolint{D_i \bfu_j}{ \stiffref_{ijk\ell} D_k \bfv_{\ell}} \\
  - &\sk \bfu_j^T R_{kk}(\stiffref_{kjk\ell}) \bfv_{\ell} .
\end{aligned}
\end{equation}
To simplify the notation in what follows, we define
\begin{equation}
  \Rmultid_{j\ell} = \sk R_{kk}(\stiffref_{kjk\ell}).
\end{equation}
Due to the major symmetry of $\stiffref_{ijk\ell}$ \eqref{eq:C_symmetry} and the symmetry $R_{kk} = R_{kk}^T$, we have
\begin{equation}
  \Rmultid_{j\ell} = \Rmultid_{\ell j} = \Rmultid_{j \ell}^T.
\end{equation}
By \eqref{eq:semidef_R_multid}, $\Rmultid_{j\ell}$ is positive semidefinite, i.e.,
\begin{equation}
  \bfu_j^T \Rmultid_{j\ell} \bfu_{\ell} \geq 0 \quad \forall \bfu_j.
\end{equation}
Another property that $\Rmultid_{j\ell}$ inherits from $R_{xx}$ is that it is zero to the order of accuracy in the sense that
\begin{equation} \label{eq:W_consistent}
  \bfu_j^T \Rmultid_{j\ell} \bfv_{\ell} = \mathcal{O}(h^{2q})
\end{equation}
for all $\bfu_j$, $\bfv_{\ell}$ that are restrictions of smooth functions to the grid. Thus, $\Rmultid_{j\ell}$ is a consistent approximation of the zero operator and we write $\Rmultid_{j\ell} \approx 0$.
We restate \eqref{eq:ibp_discr_operator} as
\begin{equation} \label{eq:ibp_discr_operator_clean}
\begin{aligned}
  \dvolint{\bfu_j}{\cD_{ik}(\stiffref_{ijk\ell}) \bfv_{\ell}} &= \dsurfint{\bfu_j}{\nu_i \stiffref_{ijk\ell} D_k \bfv_{\ell}}
  - \dvolint{D_i \bfu_j}{\stiffref_{ijk\ell} D_k \bfv_{\ell}} \\
  &- \bfu_j^T \Rmultid_{j\ell} \bfv_{\ell} .
\end{aligned}
\end{equation}
At this point, we introduce the following two new definitions, which extend the SBP concept to operators of the form $\ddi \stiffref_{ijk\ell} \ddk$.

\begin{defi} \label{def:sbpop}
Given a discrete inner product that approximates $\volint{\cdot}{\cdot}$ and a non-negative bilinear form that approximates $\surfint{\cdot}{\cdot}$, we say that $\cD_{ik}^{\idx{S} \idx{B} \idx{P}}(\stiffref_{ijk\ell})$ is an SBP operator for $\ddi \stiffref_{ijk\ell} \ddk$ on $\refdom$ if
\begin{equation} \label{eq:sbp_property_elastic}
\begin{aligned}
  \dvolint{\bfu_j}{\cD_{ik}^{\idx{S} \idx{B} \idx{P}}(\stiffref_{ijk\ell}) \bfv_{\ell}} &= \dsurfint{\bfu_j}{\nu_i \stiffref_{ijk\ell} \widetilde{D}_k \bfv_{\ell}} - \dvolint{D_i \bfu_j}{\stiffref_{ijk\ell} D_k \bfv_{\ell}} \\
  &- \bfu_j^T \Rmultid_{j\ell} \bfv_{\ell},
\end{aligned}
\end{equation}
where $D_i \approx \ddi$, $\widetilde{D}_i \approx \ddi$, $\Rmultid_{j\ell} = \Rmultid_{\ell j}^T \approx 0$, and $\bfu_j^T \Rmultid_{j\ell} \bfu_{\ell} \geq 0 \; \forall \bfu_j$.
\end{defi}
\begin{defi} \label{def:sbpop_fc}
An operator $\cD_{ik}^{\idx{S} \idx{B} \idx{P}}(\stiffref_{ijk\ell})$ is called a \emph{fully compatible} SBP operator for $\ddi \stiffref_{ijk\ell} \ddk$ on $\refdom$ if it satisfies \eqref{eq:sbp_property_elastic} with $\widetilde{D}_i = D_i$.
\end{defi}
The statement \eqref{eq:ibp_discr_operator_clean} shows that $\cD_{ik}(\stiffref_{ijk\ell})$, which was defined in \eqref{eq:def_cd} and is based on fully compatible one-dimensional SBP operators, is a fully compatible SBP operator for $\ddi \stiffref_{ijk\ell} \ddk$.

The following lemma shows that an SBP operator for $\ddi \stiffref_{ijk\ell} \ddk$ also mimics the formula that follows from using integration by parts twice:
\begin{equation} \label{eq:ibp2_cont}
\begin{aligned}
  \volint{u_j}{\ddi \stiffref_{ijk\ell} \ddk v_{\ell}} &= \surfint{u_j}{\nu_i \stiffref_{ijk\ell} \ddk v_{\ell}} - \surfint{\nu_i \stiffref_{ijk\ell} \ddk u_\ell}{ v_{j}} \\
  &+ \volint{\ddi \stiffref_{ijk\ell} \ddk u_{\ell}}{v_{j}}.
\end{aligned}
\end{equation}
\begin{lemma}
If $\cD_{ik}^{\idx{S} \idx{B} \idx{P}}(\stiffref_{ijk\ell})$ is an SBP operator for $\ddi \stiffref_{ijk\ell} \ddk$, then
\begin{equation}
\begin{aligned}
  \dvolint{\bfu_j}{\cD_{ik}^{\idx{S} \idx{B} \idx{P}}(\stiffref_{ijk\ell}) \bfv_{\ell}}&= \surfint{\bfu_j}{\nu_i \stiffref_{ijk\ell} \widetilde{D}_k \bfv_{\ell}} - \surfint{\nu_i \stiffref_{ijk\ell} \widetilde{D}_k \bfu_\ell}{ \bfv_{j}} \\
  &+ \volint{\cD_{ik}^{\idx{S} \idx{B} \idx{P}}(\stiffref_{ijk\ell}) \bfu_{\ell}}{\bfv_{j}}.
\end{aligned}
\end{equation}
\end{lemma}
\begin{proof}
By Definition \ref{def:sbpop},
\begin{equation} \label{eq:defi_restated}
\begin{aligned}
  \dvolint{\bfu_j}{\cD_{ik}^{\idx{S} \idx{B} \idx{P}}(\stiffref_{ijk\ell}) \bfv_{\ell}} &= \dsurfint{\bfu_j}{\nu_i \stiffref_{ijk\ell} \widetilde{D}_k \bfv_{\ell}} - \dvolint{D_i \bfu_j}{\stiffref_{ijk\ell} D_k \bfv_{\ell}} \\
  &- \bfu_j^T \Rmultid_{j\ell} \bfv_{\ell}.
\end{aligned}
\end{equation}
Using the symmetry of $\dvolint{\cdot}{\cdot}$ and $\dsurfint{\cdot}{\cdot}$, the major symmetry of $\stiffref_{ijk\ell}$ \eqref{eq:C_symmetry}, and $\Rmultid_{j\ell} = \Rmultid_{\ell j}^T$, we can write \eqref{eq:defi_restated} as
\begin{equation} \label{eq:defi_symm_used}
\begin{aligned}
  \dvolint{\cD_{ik}^{\idx{S} \idx{B} \idx{P}}(\stiffref_{ijk\ell}) \bfv_{\ell}}{\bfu_j} &= \dsurfint{\nu_i \stiffref_{ijk\ell} \widetilde{D}_k \bfv_{\ell}}{\bfu_j} - \dvolint{D_i \bfv_j}{\stiffref_{ijk\ell} D_k \bfu_{\ell}} \\
  &- \bfv_j^T \Rmultid_{j\ell} \bfu_{\ell}.
\end{aligned}
\end{equation}
Swapping $\bfu_j$ and $\bfv_j$ in \eqref{eq:defi_symm_used} leads to
\begin{equation} \label{eq:defi_swapped}
\begin{aligned}
  \dvolint{\cD_{ik}^{\idx{S} \idx{B} \idx{P}}(\stiffref_{ijk\ell}) \bfu_{\ell}}{\bfv_j} &= \dsurfint{\nu_i \stiffref_{ijk\ell} \widetilde{D}_k \bfu_{\ell}}{\bfv_j} - \dvolint{D_i \bfu_j}{\stiffref_{ijk\ell} D_k \bfv_{\ell}} \\
  &- \bfu_j^T \Rmultid_{j\ell} \bfv_{\ell}.
\end{aligned}
\end{equation}
Subtracting \eqref{eq:defi_swapped} from \eqref{eq:defi_restated} yields
\begin{equation}
\begin{aligned}
  \dvolint{\bfu_j}{\cD_{ik}^{\idx{S} \idx{B} \idx{P}}(\stiffref_{ijk\ell}) \bfv_{\ell}} - \dvolint{\cD_{ik}^{\idx{S} \idx{B} \idx{P}}(\stiffref_{ijk\ell}) \bfu_{\ell}}{\bfv_j} &= \dsurfint{\bfu_j}{\nu_i \stiffref_{ijk\ell} \widetilde{D}_k \bfv_{\ell}} \\
  &- \dsurfint{\nu_i \stiffref_{ijk\ell} \widetilde{D}_k \bfu_{\ell}}{\bfv_j}
\end{aligned}
\end{equation}
and the result follows after rearranging terms.
\end{proof}

We are now in position to use formula \eqref{eq:cont_transf_operator_clean} to construct an FD operator that approximates $\ddxi \stiffphys_{\is \js \ks \ls} \ddxk$. We define
\begin{equation}
  \cD^{\Omega}_{\is \ks }(\stiffphys_{\is \js \ks \ls}) := J^{-1} \cD_{ik}(\stiffref_{i \js k \ls}),
\end{equation}
\reftwo{where $\cD_{ik}$ is defined as in \eqref{eq:def_cd}, i.e., constructed from fully compatible second-derivative operators.} The main result of this section is stated in the following theorem.

\begin{theorem} \label{theorem:elastic_sbp_phys}
The operator $\cD^{\Omega}_{\is \ks }(\stiffphys_{\is \js \ks \ls}) = J^{-1} \cD_{ik}(\stiffref_{i \js k \ls})$ is a fully compatible SBP operator for $\ddxi \stiffphys_{\is \js \ks \ls} \ddxk$ on the physical domain $\physdom$.
\end{theorem}
\begin{proof}
We first derive a formula that simplifies the proof of the theorem.
Using first the definition of $\stiffref_{i j k \ell}$ and then Nanson's formula \eqref{eq:nanson}, we obtain
\begin{equation} \label{eq:n-C_conversion}
  \nu_i \stiffref_{i j k \ell} = \nu_i \K_{\is i} J \stiffphys_{\is j \ks \ell} \K_{\ks k} = J^{-1} \surfjacobian n_\is J \stiffphys_{\is j \ks \ell} \K_{\ks k} = \surfjacobian n_\is \stiffphys_{\is j \ks \ell} \K_{\ks k}.
\end{equation}
We are now ready to prove the result. We have
\begin{align*}
  &\dvolintphys{\bfu_\js}{\cD^{\Omega}_{\is \ks }(\stiffphys_{\is \js \ks \ls}) \bfv_{\ls}} = \dvolintphys{\bfu_\js}{J^{-1} \cD_{ik}(\stiffref_{i \js k \ls}) \bfv_{\ls}} \\
  =&\dvolint{\bfu_\js}{\cD_{ik}(\stiffref_{i \js k \ls}) \bfv_{\ls}} \tag{use \eqref{eq:ibp_discr_operator_clean}} \\
  = &\dsurfint{ \bfu_\js}{ \nu_i \stiffref_{i \js k \ls} D_k  \bfv_{\ls}} - \dvolint{D_i  \bfu_\js}{\stiffref_{i\js k\ls} D_k  \bfv_{\ls}} - \bfu_\js^T \Rmultid_{\js\ls}  \bfv_{\ls}  \tag{use \eqref{eq:n-C_conversion}\reftwo{,\eqref{eq:transf_density_stiffness}}} \\
  = &\dsurfint{ \bfu_\js}{\surfjacobian n_\is \stiffphys_{\is \js \ks \ls} \K_{\ks k}  D_k  \bfv_{\ls} } - \dvolint{D_i  \bfu_\js}{\K_{\is i} J \stiffphys_{\is \js \ks \ls} \K_{\ks k} D_k  \bfv_{\ls}} \\
  &- \bfu_\js^T \Rmultid_{\js \ls} \bfv_{\ls}  \tag{\reftwo{use \eqref{eq:chain_discrete},\eqref{eq:discr_volint_curve},\eqref{eq:discr_surfint_curve}}} \\
  = &\dsurfintphys{\bfu_\js}{n_\is \stiffphys_{\is \js \ks \ls} D_\ks \bfv_{\ls}} - \dvolintphys{D_\is \bfu_\js}{\stiffphys_{\is \js \ks \ls} D_\ks \bfv_{\ls}} - \bfu_\js^T \Rmultid_{\js \ls} \bfv_{\ls} .
\end{align*}
\end{proof}
In analogy with the continuous traction operator $\ctractionop_{\js \ls}$ defined in \eqref{eq:tractionop_cont_def}, we define the discrete traction operator
\begin{equation} \label{eq:discr_traction_op}
  \tractionop_{\js \ls} = n_\is \stiffphys_{\is \js \ks \ls} D_\ks.
\end{equation}
The integration-by-parts formulas satisfied by $\cD^{\Omega}_{\is \ks }(\stiffphys_{\is \js \ks \ls})$ now read
\begin{equation} \label{eq:ibp_phys_clean}
\begin{aligned}
  \dvolintphys{\bfu_\js}{\cD^{\Omega}_{\is \ks }(\stiffphys_{\is \js \ks \ls}) \bfv_{\ls}} &= \dsurfintphys{\bfu_\js}{\tractionop_{\js \ls} \bfv_{\ls}} - \dvolintphys{D_\is \bfu_\js}{\stiffphys_{\is \js \ks \ls} D_\ks \bfv_{\ls}} \\
  &- \bfu_\js^T \Rmultid_{\js \ls} \bfv_{\ls}
\end{aligned}
\end{equation}
and
\begin{equation} \label{eq:ibp2_phys_clean}
\begin{aligned}
  \dvolintphys{\bfu_\js}{\cD^{\Omega}_{\is \ks }(\stiffphys_{\is \js \ks \ls}) \bfv_{\ls}} &= \dsurfintphys{\bfu_\js}{\tractionop_{\js \ls} \bfv_{\ls}} - \dsurfintphys{\tractionop_{\js \ls} \bfu_\ls}{\bfv_{\js}} \\
  &+ \dvolintphys{\cD^{\Omega}_{\is \ks }(\stiffphys_{\is \js \ks \ls}) \bfu_\ls}{\bfv_{\js}}.
\end{aligned}
\end{equation}

\section{Energy-stable and self-adjoint boundary SATs}\label{sec:bc}
We discretize the problem \eqref{eq:wave_eq_general} in space as
\begin{equation} \label{eq:semidiscrete_general}
  \rho \ddot{\bv{u}}_\js = \cD^{\physdom}_{\is \ks}(\stiffphys_{\is \js \ks \ls}) \bfu_{\ls} + \bff_\js
  + SAT_\js,
\end{equation}
where the SATs in $SAT_{\js}$ impose the boundary conditions and will be specified later. For notational convenience we assume $f_\js = 0$ in the following analysis.
Multiplying \eqref{eq:semidiscrete_general} by $\bfphi^T_\js JH$, where $\bfphi_\js$ is an arbitrary test function, leads to the equivalent weak form:
\begin{equation}
  \dvolintphys{\bfphi_\js}{\rho \ddot{\bv{u}}_\js} = \dvolintphys{\bfphi_\js}{\cD^{\Omega}_{\is \ks}(\stiffphys_{\is \js \ks \ls}) \bv{u}_{\ls}}
    + \dvolintphys{\bfphi_\js }{ SAT_\js}.
\end{equation}
After using the integration-by-parts formula \eqref{eq:ibp_phys_clean}, the weak form reads
\begin{equation}
\begin{aligned}
  \dvolintphys{\bfphi_\js}{\rho \ddot{\bv{u}}_\js}&= \dsurfintphys{\bfphi_{\js}}{\tractionop_{\js \ls} \bfu_{\ls}}
  - \dvolintphys{D_\is \bfphi_{\js}}{\stiffphys_{\is \js \ks \ls} D_\ks \bfu_{\ls}} - \bfphi_{\js}^T \Rmultid_{\js \ls} \bfu_{\ls}
    \\
    &+ \dvolintphys{\bfphi_\js }{ SAT_\js}.
\end{aligned}
\end{equation}
Define the inner product
\begin{equation}
  M(\vec{\bfphi},\vec{\bfu}) = \dvolintphys{\bfphi_\js}{\rho \bfu_\js},
\end{equation}
the symmetric positive semidefinite bilinear form
\begin{equation}
  K(\vec{\bfphi},\vec{\bfu}) = \dvolintphys{D_\is \bfphi_{\js}}{\stiffphys_{\is \js \ks \ls} D_\ks \bfu_{\ls}} + \bfphi_{\js}^T \Rmultid_{\js \ls} \bfu_{\ls},
\end{equation}
and
\begin{equation} \label{eq:boundary_form_general}
  B(\vec{\bfphi},\vec{\bfu}) = \dsurfintphys{\bfphi_{\js}}{\tractionop_{\js \ls} \bfu_{\ls}} + \dvolintphys{\bfphi_\js }{ SAT_\js}.
\end{equation}
In this paper $SAT_\js$ is always linear in $\vec{\bfu}$ and thus $B(\cdot, \cdot)$ is a bilinear form in the case of homogeneous boundary conditions.
The weak form can now be written as
\begin{equation} \label{eq:weak_single}
  M(\vec{\bfphi},\ddot{\vec{\bfu}}) + K(\vec{\bfphi},\vec{\bfu}) = B(\vec{\bfphi},\vec{\bfu}).
\end{equation}
We define the discrete energy
\begin{equation}
\begin{aligned}
 E := &\frac{1}{2}\dvolintphys{\bfut_\js}{\rho \bfut_\js} + \frac{1}{2} \dvolintphys{D_\is \bfu_\js}{\stiffphys_{\is \js \ks \ls} D_\ks \bfu_{\ls}} + \frac{1}{2}  \bfu_\js^T \Rmultid_{\js\ls} \bfu_{\ls} \\
 = &\frac{1}{2} M(\dot{\vec{\bfu}},\dot{\vec{\bfu}}) + \frac{1}{2} K(\vec{\bfu},\vec{\bfu}).
 \end{aligned}
\end{equation}
Recalling that $\Rmultid_{\js \ls}$ is zero to order $2q$, we conclude that the discrete energy $E$ approximates the continuous energy $\mathcal{E}$ defined in \eqref{eq:energy_cont}. It follows from the non-negativity of $M$ and $K$ that $E$ is a non-negative quantity.
Setting $\vec{\bfphi} = \dot{\bfu}$ in \eqref{eq:weak_single} yields the discrete energy rate
\begin{equation} \label{eq:energy_rate_discr}
  \dd{E}{t} = B(\dot{\vec{\bfu}},\vec{\bfu}).
\end{equation}
For future use we note that the integration-by-parts formula \eqref{eq:ibp_phys_clean} can be written as
\begin{equation} \label{eq:ibp_K}
  \dvolintphys{\bfphi_\js}{\cD^{\Omega}_{\is \ks}(\stiffphys_{\is \js \ks \ls}) \bv{u}_{\ls}} = \dsurfintphys{\bfphi_{\js}}{\tractionop_{\js \ls} \bfu_{\ls}} - K(\vec{\bfphi},\vec{\bfu}).
\end{equation}

\subsection{Robin boundary conditions}
Consider Robin boundary conditions,
\begin{equation} \label{eq:bc_neumann_cont}
 \ctractionop_{\js \ls} u_{\ls} + U_{\js \ls} u_{\ls} = g_\js, \quad \vec{X} \in \physboundary,
\end{equation}
where $U_{\js \ls} = U_{\ls \js}$ and $u_\js U_{\js \ls} u_\ls \geq 0 \; \forall u_{\js}$. Robin conditions include the important case of traction conditions, obtained by setting $U_{\js \ls} = 0$ in \eqref{eq:bc_neumann_cont}.
It follows from \eqref{eq:energy_rate_phys} that, for $g_\js = 0$, the continuous solution satisfies the energy balance
\begin{equation} \label{eq:cont_eb_robin}
  \dd{\widetilde{\mathcal{E}}}{t} = 0,
\end{equation}
where
\begin{equation}
  \widetilde{\mathcal{E}} = \mathcal{E} + \frac{1}{2} \dsurfintphys{u_\js}{ U_{\js \ls} u_{\ls} }.
\end{equation}

If $SAT_\js$ satisfies
\begin{equation} \label{eq:energy_preserving_sat}
  \dvolintphys{\bfphi_\js }{ SAT_\js} = -\dsurfintphys{\bfphi_\js}{\tractionop_{\js \ls} \bfu_\ls + U_{\js \ls} \bfu_\ls - \boldsymbol{g}_\js},
\end{equation}
then, for $g_\js=0$, we obtain
\begin{equation} \label{eq:B_traction}
\begin{aligned}
  B(\vec{\bfphi},\vec{\bfu}) &= \dsurfintphys{\bfphi_{\js}}{\tractionop_{\js \ls} \bfu_{\ls}} -\dsurfintphys{\bfphi_\js}{\tractionop_{\js \ls} \bfu_\ls + U_{\js \ls} \bfu_\ls } \\
  &= -\dsurfintphys{\bfphi_\js}{U_{\js \ls} \bfu_\ls},
\end{aligned}
\end{equation}
which is a symmetric bilinear form. It follows that
\begin{equation}
  B(\dot{\vec{\bfu}},\vec{\bfu}) = -\frac{1}{2} \dd{}{t} \dsurfintphys{\bfu_\js}{ U_{\js \ls} \bfu_{\ls} } ,
\end{equation}
which yields the energy balance
\begin{equation} \label{eq:discr_eb_robin}
\dd{}{t}\widetilde{E} = 0,
\end{equation}
where
\begin{equation} \label{eq:energy_robin}
  \widetilde{E} = E + \frac{1}{2} \dsurfintphys{\bfu_\js}{ U_{\js \ls} \bfu_{\ls} } \geq 0,
\end{equation}
which shows that the scheme is energy stable.
We achieve \eqref{eq:energy_preserving_sat} by setting
\begin{equation} \label{eq:sat_traction}
  SAT_\js = -(JH)^{-1} \sum \limits_{f\in \setoffaces} e_f \surfjacobian H_f \left( e_f^T \left( \tractionop_{\js \ls} \bfu_\ls + U_{\js \ls}\bfu_\ls \right)- \boldsymbol{g}_\js \right) ,
\end{equation}
where $\setoffaces$ denotes the set of all faces and was defined in \eqref{eq:set_of_faces}. The SAT \eqref{eq:sat_traction} is the standard SAT for Robin boundary conditions, see \cite{Duru2014}.

\begin{remark}
Robin boundary conditions can be generalized by introducing an additional term $V_{\js \ls} \dot{u}_{\ls}$, where $u_{\js} V_{\js \ls} u_{\ls} \geq 0 \; \forall u_\js$, on the left-hand side of \eqref{eq:bc_neumann_cont}. This term introduces energy dissipation in the continuous problem. It is straightforward to generalize the SAT \eqref{eq:energy_preserving_sat} to such BC and obtain corresponding dissipation of discrete energy, see \cite{Duru2014}. To streamline the discussion of self-adjointness, however, we restrict our attention to Robin and displacement conditions in this paper.
\end{remark}

\subsection{Displacement boundary conditions}
We now consider displacement conditions,
\begin{equation} \label{eq:bc_dirichlet_cont}
  u_\js = g_\js, \quad \vec{X} \in \physboundary.
\end{equation}
The homogeneous conditions obtained by setting $g_\js=0$ are energy-conserving for the continuous equations. However, there are no consistent SATs that make $B(\vec{\bfphi},\vec{\bfu})$ vanish \refthree{(it is clear from e.g.\ \eqref{eq:boundary_form_general} that the unique SAT that makes $B$ vanish is \eqref{eq:sat_traction} with $U_{\js\ls} = 0$, $\boldsymbol{g}_\js=0$, which is consistent with homogeneous traction conditions)}. Instead, we shall choose SATs that symmetrize the form $B(\cdot,\cdot)$. Suppose that
\begin{equation} \label{eq:sat_displ_weak}
  \dvolintphys{\bfphi_\js }{ SAT_\js} = \dsurfintphys{\tractionop_{\ls \js} \bfphi_{\js}}{\bfu_{\ls} - \boldsymbol{g}_\ls} - \dsurfintphys{\satsym_{\ls \js} \bfphi_{\js}}{\bfu_{\ls} - \boldsymbol{g}_\ls},
\end{equation}
for some yet unspecified $\satsym_{\js \ls}$ that is symmetric with respect to the boundary quadrature in the sense that
\begin{equation} \label{eq:sat_symmetry_assumption}
  \dsurfintphys{\satsym_{\ls \js} \cdot }{ \cdot} = \dsurfintphys{\cdot }{ \satsym_{\js \ls} \cdot} .
\end{equation}
Then, for $g_\js = 0$, we obtain
\begin{equation} \label{eq:B_sym_displacement}
  B(\vec{\bfphi},\vec{\bfu}) = \dsurfintphys{\bfphi_{\js}}{\tractionop_{\js \ls} \bfu_{\ls}} + \dsurfintphys{\tractionop_{\ls \js} \bfphi_{\js}}{\bfu_{\ls}} - \dsurfintphys{\bfphi_{\js}}{\satsym_{\js \ls} \bfu_{\ls}},
\end{equation}
which is a symmetric bilinear form. It follows that
\begin{equation}
  \dd{E}{t} = B(\dot{\vec{\bfu}},\vec{\bfu}) = \frac{1}{2} \dd{}{t} B(\vec{\bfu},\vec{\bfu}).
\end{equation}
We obtain the energy balance
\begin{equation}
\dd{E_d}{t} = 0,
\end{equation}
where the modified energy $E_d$ is
\begin{equation} \label{eq:e_d}
  E_d = E - \frac{1}{2} B(\vec{\bfu},\vec{\bfu} ) = E - \dsurfintphys{\bfu_\js}{\tractionop_{\js \ls} \bfu_\ls } + \frac{1}{2} \dsurfintphys{\bfu_\js }{\satsym_{\js \ls} \bfu_\ls}.
\end{equation}
Note that $E_d$, just like $E$, is a high-order approximation of the continuous energy $\mathcal{E}$ because $B(\vec{\bfu},\vec{\bfu})$ is zero to the order of accuracy due to the boundary condition.

The SAT that satisfies \eqref{eq:sat_displ_weak} is
\begin{equation} \label{eq:dirichlet_ansatz}
  SAT_\js = (JH)^{-1} \sum \limits_{f \in \setoffaces} \left(\tractionop_{\ls \js} - \satsym_{\ls \js} \right)^T e_{f} \surfjacobian H_{f} (e_{f}^T \bfu_\ls - \boldsymbol{g}_\ls),
\end{equation}
where $\satsym_{\ls \js}$ remains unspecified at this point. The ansatz \eqref{eq:dirichlet_ansatz} ensures that the SATs are consistent with displacement boundary conditions. For fixed $\jl$ and $\elll$, $\satsym_{\ls \js}$ is an $N\times N$ matrix with units of force per unit volume. For $SAT_\js$ to have the same $h$-dependence as $\cD_{\is \ks}^{\Omega}(\stiffphys_{\is \js \ks \ls})$, which is a second derivative and hence scales as $h^{-2}$, the entries of $\satsym_{\ls \js}$ must be proportional to $h^{-1}$. Because the boundary quadrature operator is diagonal, the condition \eqref{eq:sat_symmetry_assumption} is satisfied if $\satsym_{\js \ls} = \satsym_{\ls \js}$ and $\satsym_{\js \ls}$ is diagonal for each $\jl$ and $\elll$.

To prove stability, it remains to prove that we can choose $\satsym_{\ls \js}$ so that $E_d$ is a non-negative quantity. To accomplish this, we use the positivity of $E$. Since the indefinite term in $E_d$ is a surface integral, we bound $E$ from below by a surface integral. We have
\begin{equation} \label{eq:borrowing_H_multiD}
\begin{aligned}
  2E &= \dvolintphys{\bfut_\js}{\rho \bfut_\js} + \dvolintphys{D_\is \bfu_\js}{\stiffphys_{\is \js \ks \ls} D_\ks \bfu_{\ls}} + \bfu_\js^T \Rmultid_{\js\ls} \bfu_{\ls} \\
  &\geq \frac{h_1}{d} \dsurfintphys{D_\is \bfu_\js}{\frac{J}{\surfjacobian} \stiffphys_{\is \js \ks \ls} D_\ks \bfu_{\ls}},
  \end{aligned}
 \end{equation}
where we used the positivity property \eqref{eq:borrowing_phys} in the last step.
Using \eqref{eq:borrowing_H_multiD} in \eqref{eq:e_d} yields
\begin{equation}
\begin{aligned}
2E_d &\geq \frac{h_1}{d} \dsurfintphys{D_\is \bfu_\js}{\surfjacobian^{-1} J \stiffphys_{\is \js \ks \ls} D_\ks \bfu_{\ls}} - 2 \dsurfintphys{\bfu_\js}{\tractionop_{\js \ls} \bfu_\ls} \\
&+ \dsurfintphys{\bfu_\js }{\satsym_{\js \ls} \bfu_\ls} .
\end{aligned}
\end{equation}
Recalling the definition of $\tractionop_{\js \ls}$ \eqref{eq:discr_traction_op} lets us write
\begin{equation} \label{eq:indefinite_term_expanded}
  \dsurfintphys{\bfu_\js}{\tractionop_{\js \ls} \bfu_\ls} =
    \dsurfintphys{\bfu_\js}{n_\is \stiffphys_{\is \js \ks \ls} D_\ks \bfu_{\ls}} = \dsurfintphys{n_\is \bfu_\js}{\stiffphys_{\is \js \ks \ls} D_\ks \bfu_{\ls}}.
\end{equation}
Due to the major symmetry of $\stiffphys_{\is \js \ks \ls}$ \eqref{eq:C_symmetry}, we have
\begin{equation}
  \dsurfintphys{n_\is \bfu_\js}{\stiffphys_{\is \js \ks \ls} D_\ks \bfu_{\ls}} = \dsurfintphys{n_\ks \bfu_\ls}{\stiffphys_{\is \js \ks \ls} D_\is \bfu_{\js}}.
\end{equation}
By completing the squares, we obtain
\begin{equation}
\begin{aligned}
    2E_d &\geq \frac{h_1}{d} \dsurfintphys{D_\is \bfu_\js}{\surfjacobian^{-1} J \stiffphys_{\is \js \ks \ls} D_\ks \bfu_{\ls}} - 2 \dsurfintphys{n_\is \bfu_\js}{\stiffphys_{\is \js \ks \ls} D_\ks \bfu_{\ls}} \\
    &+ \dsurfintphys{\bfu_\js }{\satsym_{\js \ls} \bfu_\ls} \\
    &= \frac{h_1}{d} \dsurfintphys{D_\is \bfu_{\js} - \frac{d\surfjacobian}{h_1 J} n_\is \bfu_\js}{\frac{J}{\surfjacobian} \stiffphys_{\is \js \ks \ls} \left( D_\ks \bfu_{\ls} - \frac{d \surfjacobian}{h_1 J} n_\ks \bfu_{\ls} \right) } \\
    &- \dsurfintphys{n_\is \bfu_\js}{ \frac{d \surfjacobian}{h_1 J} \stiffphys_{\is \js \ks \ls} n_\ks \bfu_{\ls} } + \dsurfintphys{\bfu_\js }{\satsym_{\js \ls} \bfu_{\ls}} \\
    &\geq \dsurfintphys{\bfu_\js }{ \left(\satsym_{\js\ls} - \frac{d \surfjacobian}{h_1 J} n_\is \stiffphys_{\is \js \ks \ls} n_\ks \right) \bfu_{\ls} }.
\end{aligned}
\end{equation}
We achieve $E_d \geq 0$ by setting
\begin{equation}
    \satsym_{\js \ls} = \beta \frac{d \surfjacobian}{h_1 J} n_\is \stiffphys_{\is \js \ks \ls} n_\ks, \quad \beta \geq 1.
\end{equation}
Since $\surfjacobian$, $J$, $n_\is$, and $\stiffphys_{\is \js \ks \ls}$ are diagonal matrices in the discrete case, the $\satsym_{\js \ls}$ are diagonal matrices. Using the major symmetry of $\stiffphys_{\is \js \ks \ls}$
\eqref{eq:symmetry_phys}, we have
\begin{equation}
  \satsym_{\js \ls} = \beta \frac{d \surfjacobian}{h_1 J} n_\is  \stiffphys_{\is \js \ks \ls} n_\ks = \beta \frac{d \surfjacobian}{h_1 J} n_\is  \stiffphys_{\ks \ls \is \js} n_\ks = \satsym_{\ls \js},
\end{equation}
which verifies that $\satsym_{\js \ls}$ satisfies the symmetry assumption \eqref{eq:sat_symmetry_assumption}.
We have now proven the following theorem.
\begin{theorem} \label{theorem:dirichlet}
The scheme
\begin{equation}
  \rho \ddot{\bfu}_\js = \cD^{\physdom}_{\is \ks}(\stiffphys_{\is \js \ks \ls}) \bfu_{\ls} + (JH)^{-1} \sum \limits_{f \in \setoffaces} \left(\tractionop_{\ls \js} - \satsym_{\ls \js} \right)^T e_{f} \surfjacobian H_{f} (e_{f}^T \bfu_\ls - \boldsymbol{g}_\ls),
\end{equation}
with
\begin{equation}
  \satsym_{\js \ls} = \beta \frac{d \surfjacobian}{h_1 J} n_\is \stiffphys_{\is \js \ks \ls} n_{\ks}
\end{equation}
is stable if $\beta \geq 1$.

\end{theorem}
In all simulations in this paper, we set $\beta = 1$, i.e., right on the limit of provable stability. The drawback of using larger values of $\beta$ is that this increases the spectral radius of the operator.

\reftwo{
\begin{remark}
Note that the analysis in this subsection would become significantly more involved without the assumption of fully compatible SBP operators.
Due to full compatibility, the positivity property
\begin{equation} \label{eq:borrowing_old}
  \dvolintphys{D_\is \bfu_\js}{\stiffphys_{\is \js \ks \ls} D_\ks \bfu_{\ls}} \geq \frac{h_1}{d} \dsurfintphys{D_\is \bfu_\js}{\frac{J}{\surfjacobian} \stiffphys_{\is \js \ks \ls} D_\ks \bfu_{\ls}}
\end{equation}
is sufficient to prove that $E_d \geq 0$. Without full compatibility, however, the discrete traction operator takes the form
\begin{equation}
  \tractionop_{\js \ls} = n_\is C_{\is \js \ks \ls} (D_\ks + \Delta D_\ks),
\end{equation}
where $\Delta D_\ks$ denotes the difference between compatible and fully compatible boundary derivatives and is of order $q$. Proving that $E_d \geq 0$ by completing the squares then requires, in addition to \eqref{eq:borrowing_old}, at least one of the following two positivity properties:
\begin{align}
\label{eq:borrowing_1}
  \dvolintphys{D_\is \bfu_\js}{\stiffphys_{\is \js \ks \ls} D_\ks \bfu_{\ls}} &\geq  \dsurfintphys{\Delta D_\is \bfu_\js}{\alpha C_{\is \js \ks \ls} \Delta D_\ks \bfu_{\ls}}  \quad \mbox{for some } \alpha > 0, \\
  \label{eq:borrowing_2}
  \bfu_\js^T \Rmultid_{\js\ls} \bfu_{\ls} &\geq \dsurfintphys{\Delta D_\is \bfu_\js}{\beta C_{\is \js \ks \ls} \Delta D_\ks \bfu_{\ls}} \quad \mbox{for some } \beta > 0.
\end{align}
Our numerical tests (not reported here) reveal that \eqref{eq:borrowing_1} does not hold for the operators derived in \cite{Mattsson11}. Thus, the only possibility appears to be to prove \eqref{eq:borrowing_2}. When solving the scalar wave equation,  \cite{AlmquistDunham2020} proved a simpler version of \eqref{eq:borrowing_2}. Extending that proof to the elastic operator remains an open problem. For the idealized case of constant material properties and affine coordinate transformations, \cite{Duru201437} claimed that \eqref{eq:borrowing_2} holds, but did not show a complete proof for multi-dimensional settings.

Not assuming full compatibility would complicate the interface analysis in Section \ref{sec:interfaces} in the same manner.
\end{remark}
}

\subsection{Self-adjointness}
The adjoint of the discrete operator plays an important role in PDE-constrained optimization problems such as seismic imaging, where the adjoint state method is frequently used to compute the gradient of the objective functional. The continuous elastic operator is self-adjoint, and this subsection is devoted to proving that the discrete elastic operator is also self-adjoint. A consequence of this property is that one may use the same solver for the forward and adjoint PDEs and still obtain the exact (up to roundoff error) gradient of a discrete objective functional (provided that the time-discretization is also adjoint-consistent).

Let $\mathcal{U}$ and $\Phi$ be subsets of $L^2(\physdom)$. We think of $\mathcal{U}$ as the primal space and $\Phi$ as the dual or adjoint space. The adjoint $\mathcal{L}_{\js \ls}^{\dagger}: \Phi \rightarrow L^2(\physdom) $ of a linear operator $\mathcal{L}_{\js \ls}: \mathcal{U} \rightarrow L^2(\physdom)$ satisfies
\begin{equation}
  \volintphys{\phi_{\js}}{\mathcal{L}_{\js \ls} u_{\ls}} = \volintphys{\mathcal{L}^{\dagger}_{\js \ls} \phi_{\ls}}{u_{\js}} \; \forall  u_{\js} \in \mathcal{U}, \; \phi_{\js} \in \Phi.
\end{equation}
The operator $\mathcal{L}_{\js \ls}$ is said to be self-adjoint if $\mathcal{L}^{\dagger}_{\js \ls} = \mathcal{L}_{\js \ls}$, which implies that $\Phi = \mathcal{U}$ \cite{Rudin1973}.

We here consider the elastic operator $\mathcal{D}_{\js \ls} = \ddxi \stiffphys_{\is \js \ks \ls} \ddxk$. For now, we leave the domain of $\mathcal{D}_{\js \ls}$ unspecified.
We define the space of admissible functions
\begin{equation}
  \mathcal{U} = \{ u_\js \in L^2(\physdom) \; | \; \mathcal{D}_{\js \ls} u_{\ls} \in L^2(\Omega) \} .
\end{equation}
We further assume that $u_\js$ satisfies either Robin boundary conditions \eqref{eq:bc_neumann_cont} or displacement boundary conditions \eqref{eq:bc_dirichlet_cont}.
Let $\mathcal{U}_R$ and $\mathcal{U}_D$ denote the corresponding spaces:
\begin{equation}
\begin{aligned}
\mathcal{U}_R &= \{ u_{\js} \in \mathcal{U} \; | \; T_{\js \ls} u_{\ls} + U_{\js \ls} u_\ls = 0 \mbox{ on } \physboundary  \}, \\
\mathcal{U}_D &= \{ u_{\js} \in \mathcal{U} \; | \; u_{\js} = 0 \mbox{ on } \physboundary  \}.
\end{aligned}
\end{equation}
Two partial integrations yield (cf.\ \eqref{eq:ibp2_cont})
\begin{equation}
\begin{aligned}
  \volintphys{\phi_\js}{\mathcal{D}_{\js \ls} u_{\ls}} = \surfintphys{\phi_{\js}}{\ctractionop_{\js \ls} u_{\ls}} - \surfintphys{\ctractionop_{\ls \js} \phi_{\js}}{u_{\ls}} + \volintphys{\mathcal{D}_{\js \ls} \phi_{\ls}}{u_{\js}}.
\end{aligned}
\end{equation}
It follows that
\begin{equation}
  \volintphys{\phi_\js}{\mathcal{D}_{\js \ls} u_{\ls}} = \volintphys{\mathcal{D}_{\js \ls} \phi_{\ls}}{u_{\js}} \; \forall  u_{\js} \in \mathcal{U}_R, \; \phi_{\js} \in \mathcal{U}_R
\end{equation}
and
\begin{equation}
  \volintphys{\phi_\js}{\mathcal{D}_{\js \ls} u_{\ls}} = \volintphys{\mathcal{D}_{\js \ls} \phi_{\ls}}{u_{\js}} \; \forall  u_{\js} \in \mathcal{U}_D, \; \phi_{\js} \in \mathcal{U}_D,
\end{equation}
which shows that $\mathcal{D}_{\js \ls}$ is self-adjoint both with domain $\mathcal{U}_R$ (Robin conditions) and with domain $\mathcal{U}_D$ (displacement conditions).

We now consider the total discrete elastic operator, including SATs for Robin or displacement boundary conditions. Assuming  homogeneous boundary conditions, we can define $\mathbb{S}_{\js \ls}$ such that
\begin{equation}
  SAT_\js = \mathbb{S}_{\js \ls} \bfu_\ls,
\end{equation}
and the total discrete elastic operator is
\begin{equation}
  \cD^{tot}_{\js \ls} = \cD^{\physdom}_{\is \ks}(\stiffphys_{\is \js \ks \ls}) + \mathbb{S}_{\js \ls}.
\end{equation}

\begin{theorem} \label{thm:selfadjoint_boundary}
The total discrete elastic operator, including SATs for Robin or displacement boundary conditions, is self-adjoint, i.e.,
\begin{equation}
  \dvolintphys{\bfphi_\js}{\cD_{\js \ls}^{tot}\bfu_{\ls}} = \dvolintphys{\cD_{\js \ls}^{tot}\bfphi_\ls}{\bfu_{\js}} \quad \forall \bfphi_\js , \bfu_\js .
\end{equation}
\end{theorem}
\begin{proof}
In deriving the weak form \eqref{eq:weak_single}, we showed that
\begin{equation}
  \dvolintphys{\bfphi_\js}{\cD_{\js \ls}^{tot}\bfu_{\ls}} = -K(\vec{\bfphi},\vec{\bfu}) + B(\vec{\bfphi},\vec{\bfu}),
\end{equation}
where $K$ is symmetric and $B$ is symmetric both in the case of Robin conditions (cf.\ \eqref{eq:B_traction}) and in the case of displacement conditions (cf.\ \eqref{eq:B_sym_displacement}). Hence, we have
\begin{equation}
\begin{aligned}
  \dvolintphys{\bfphi_\js}{\cD_{\js \ls}^{tot}\bfu_{\ls}} &= -K(\vec{\bfphi},\vec{\bfu}) + B(\vec{\bfphi},\vec{\bfu}) = -K(\vec{\bfu},\vec{\bfphi}) + B(\vec{\bfu},\vec{\bfphi}) \\
  &= \dvolintphys{\bfu_{\js}}{\cD_{\js \ls}^{tot}\bfphi_\ls} .
\end{aligned}
\end{equation}
After using the symmetry of $\dvolintphys{\cdot}{\cdot}$, the result follows.
\end{proof}

\section{Energy-stable and self-adjoint interface SATs} \label{sec:interfaces}
We may want to introduce multiple grid blocks to: handle discontinuous material parameters $\rho$ and $\stiffphys_{\idx{IJKL}}$, facilitate grid generation, or model earthquakes or fractures, in which case there are prescribed discontinuities in either displacement or traction. The discussion below covers all cases. The Jacobian $J$ and transformation gradient $\K_{\is i}$ may be discontinuous across the interface.

Let $\physinterface$ denote the interface between two domains $\physdom_u$ and $\physdom_v$. We use superscripts $u$ and $v$ to distinguish between quantities that correspond to the two different sides of the interface. We consider the problem
\begin{equation} \label{eq:interface_phys}
\begin{array}{rl}
  \rho^u \ddot{u}_\js - \ddxi \stiffphys_{\idx{IJKL}}^u \ddxk u_{\ls} = 0, & \vec{X} \in \physdom_u, \\
  \rho^v \ddot{v}_\js - \ddxi \stiffphys_{\idx{IJKL}}^v \ddxk v_{\ls} = 0, & \vec{X} \in \physdom_v, \\
  u_\js - v_\js = V_\js, & \vec{X} \in \physinterface, \\
  \tau_\js^u + \tau_\js^v = \Theta_\js, & \vec{X} \in \physinterface, \\
\end{array}
\end{equation}
augmented with suitable boundary conditions. The functions $V_\js$ and $\Theta_\js$ denote data for jumps in displacement and traction, respectively.
Define the energies
\begin{equation}
\begin{aligned}
    \mathcal{E}_u &= \frac{1}{2} \volintphysu{\dot{u}_\js}{\rho^u \dot{u}_\js} + \frac{1}{2} \volintphysu{ \ddxi u_\js}{\stiffphys_{\idx{IJKL}}^u \ddxk u_{\ls}}, \\
    \mathcal{E}_v &= \frac{1}{2} \volintphysv{\dot{v}_\js}{\rho^v \dot{v}_\js} + \frac{1}{2} \volintphysv{ \ddxi v_\js}{\stiffphys_{\idx{IJKL}}^v \ddxk v_{\ls}}.
\end{aligned}
\end{equation}
Assuming energy-conserving boundary conditions and $V_\js=\Theta_\js=0$, the energy method yields
\begin{equation}
    \dd{}{t} \left( \mathcal{E}_u + \mathcal{E}_v \right) = \intintphys{\dot{u}_\js}{\tau_\js^u} + \intintphys{\dot{v}_\js}{\tau_\js^v} = 0.
\end{equation}
We assume that the surface Jacobian $\surfjacobian$ is the same on the two sides of the interface so that grid points that coincide in the reference domain coincide also in the physical domain. In the following equations, we suppress superscripts $u$ and $v$ on the interface restriction operators $e_{\physinterface}$, because it is clear from context that, for example, $e^T_{\physinterface} \bfu_\ks$ denotes $(e^u_{\physinterface})^T \bfu_\ks$.

\reftwo{Omitting SATs for outer boundaries for convenience, we} discretize \eqref{eq:interface_phys} as
\begin{equation} \label{eq:scheme_interface}
\begin{aligned}
    \rho^u \ddot{\bfu}_\js  &= \cD^{\physdom_u}_{\is \ks}(\stiffphys_{\is \js \ks \ls}^u) \bfu_{\ls} -   (J^u H^u)^{-1} \satsym_{\ls \js}^T e_{\physinterface}  \surfjacobian H_{\physinterface} (e_{\physinterface}^T \bfu_\ls - e_{\physinterface}^T \bfv_\ls - \mathbf{V}_\ls)\\
    &+\frac{1}{2}(J^u H^u)^{-1}  \left(\tractionop_{\ls \js}^u \right)^T e_{\physinterface}  \surfjacobian H_{\physinterface} (e_{\physinterface}^T \bfu_\ls - e_{\physinterface}^T \bfv_\ls - \mathbf{V}_\ls)  \\
    &-\frac{1}{2}(J^u H^u)^{-1}  e_{\physinterface}  \surfjacobian H_{\physinterface} (e_{\physinterface}^T \tractionop_{\js \ls}^u \bfu_\ls + e_{\physinterface}^T \tractionop_{\js \ls}^v \bfv_\ls - \mathbf{\Theta}_\js),  \\ \\
    \rho^v \ddot{\bfv}_\js  &= \cD^{\physdom_v}_{\is \ks}(\stiffphys_{\is \js \ks \ls}^v) \bfv_{\ls} - (J^v H^v)^{-1} \satsym_{\ls \js}^T e_{\physinterface}  \surfjacobian H_{\physinterface} (e_{\physinterface}^T \bfv_\ls - e_{\physinterface}^T \bfu_\ls + \mathbf{V}_\ls)\\
    &+\frac{1}{2}(J^v H^v)^{-1}  \left(\tractionop_{\ls \js}^v \right)^T e_{\physinterface}  \surfjacobian H_{\physinterface} (e_{\physinterface}^T \bfv_\ls - e_{\physinterface}^T \bfu_\ls + \mathbf{V}_\ls)  \\
    &-\frac{1}{2}(J^v H^v)^{-1}  e_{\physinterface}  \surfjacobian H_{\physinterface} (e_{\physinterface}^T \tractionop_{\js \ls}^v \bfv_\ls + e_{\physinterface}^T \tractionop_{\js \ls}^u \bfu_\ls - \mathbf{\Theta}_\js),  \\
\end{aligned}
\end{equation}
where
\begin{equation}
 \satsym_{\js \ls} = \beta \frac{d}{4h_1} \surfjacobian \left( \frac{n_{\is}^u \stiffphys_{\is \js \ks \ls}^u n_{\ks}^u}{J^u} + \frac{n_{\is}^v \stiffphys_{\is \js \ks \ls}^v n_{\ks}^v}{J^v} \right), \quad \beta \geq 1.
\end{equation}
Note that $\satsym{\js \ls}$ satisfies $\dintintphys{\satsym_{\ls \js}\cdot}{\cdot} = \dintintphys{\cdot}{\satsym_{\js \ls} \cdot}$. The remainder of this section is devoted to proving that the scheme \eqref{eq:scheme_interface} is energy stable and self-adjoint.

To derive the weak form of \eqref{eq:scheme_interface}, we multiply the first equation by $\bfphi_\js^T J^u H^u$, which, with $\mathbf{V}_\js = \mathbf{\Theta}_\js = 0$ for convenience, yields
\begin{equation} \label{eq:weak_int_first}
\begin{aligned}
  \dvolintphysu{\bfphi_\js}{\rho^u \bfutt_\js} &= \dvolintphysu{\bfphi_\js}{\cD^{\physdom_u}_{\is \ks}(\stiffphys_{\is \js \ks \ls}^u) \bfu_{\ls}} - \dintintphys{\satsym_{\ls \js} \bfphi_\js}{\bfu_\ls - \bfv_\ls} \\
  &+ \frac{1}{2} \dintintphys{\tractionop^u_{\ls \js} \bfphi_\js}{\bfu_\ls - \bfv_\ls} - \frac{1}{2} \dintintphys{\bfphi_\js}{\tractionop^u_{\js \ls} \bfu_{\ls} + \tractionop^v_{\js \ls} \bfv_{\ls}}.
\end{aligned}
\end{equation}
Let
\begin{equation}
\begin{aligned}
M_u(\vec{\bfphi}, \vec{\bfu}) &= \dvolintphysu{\bfphi_\js}{\rho^u \bfu_\js} , \\
K_u(\vec{\bfphi}, \vec{\bfu}) &= \dvolintphysu{D_\is \bfphi_\js}{\stiffphys_{\is \js \ks \ls}^u D_\ks \bfu_{\ls}} + \bfphi^T_{\js} \Rmultid^u_{\js \ls} \bfu_{\ls}, \\
E_u &= \frac{1}{2} M_u(\dot{\vec{\bfu}}, \dot{\vec{\bfu}}) + \frac{1}{2}K_u(\vec{\bfu}, \vec{\bfu}),
\end{aligned}
\end{equation}
and define $M_v$, $K_v$, and $E_v$ analogously. Using the integration-by-parts formula \eqref{eq:ibp_K} in \eqref{eq:weak_int_first} yields
\begin{equation} \label{eq:weak_1}
\begin{aligned}
  M_u(\vec{\bfphi}, \ddot{\vec{\bfu}}) + K_u(\vec{\bfphi}, \vec{\bfu}) = &- \dintintphys{\satsym_{\ls \js} \bfphi_\js}{\bfu_\ls - \bfv_\ls} + \frac{1}{2} \dintintphys{\tractionop^u_{\ls \js} \bfphi_\js}{\bfu_\ls - \bfv_\ls} \\
  &+ \frac{1}{2} \dintintphys{\bfphi_\js}{\tractionop^u_{\js \ls} \bfu_{\ls} - \tractionop^v_{\js \ls} \bfv_{\ls}},
\end{aligned}
\end{equation}
\reftwo{where we have again omitted outer boundary terms.} Multiplying the second equation in \eqref{eq:scheme_interface} by $\bfchi_\js^T J^v H^v$ similarly leads to
\begin{equation} \label{eq:weak_2}
\begin{aligned}
  M_v(\vec{\bfchi}, \ddot{\vec{\bfv}}) + K_v(\vec{\bfchi}, \vec{\bfv}) = &- \dintintphys{\satsym_{\ls \js} \bfchi_\js}{\bfv_\ls - \bfu_\ls} + \frac{1}{2} \dintintphys{\tractionop^v_{\ls \js} \bfchi_\js}{\bfv_\ls - \bfu_\ls} \\
  &+ \frac{1}{2} \dintintphys{\bfchi_\js}{\tractionop^v_{\js \ls} \bfv_{\ls} - \tractionop^u_{\js \ls} \bfu_{\ls}}.
\end{aligned}
\end{equation}
We add \eqref{eq:weak_1} and \eqref{eq:weak_2} to obtain
\begin{equation} \label{eq:weak_added}
  M_u(\vec{\bfphi}, \ddot{\vec{\bfu}}) + K_u(\vec{\bfphi}, \vec{\bfu}) + M_v(\vec{\bfchi}, \ddot{\vec{\bfv}}) + K_v(\vec{\bfchi}, \vec{\bfv}) = I(\vec{\bfphi},\vec{\bfu},\vec{\bfchi},\vec{\bfv}),
\end{equation}
where $I$ is the sum of interface integrals,
\begin{equation}
\begin{aligned}
  I(\vec{\bfphi},\vec{\bfu},\vec{\bfchi},\vec{\bfv}) = &- \dintintphys{\bfphi_\js - \bfchi_{\js}}{\satsym_{\js \ls} (\bfu_\ls - \bfv_\ls)} \\
  &+ \frac{1}{2} \dintintphys{\tractionop^u_{\js \ls} \bfphi_\ls - \tractionop^v_{\js \ls} \bfchi_\ls}{\bfu_\js - \bfv_\js} \\
  &+ \frac{1}{2} \dintintphys{\bfphi_\js - \bfchi_\js}{\tractionop^u_{\js \ls} \bfu_{\ls} - \tractionop^v_{\js \ls} \bfv_{\ls}}.
\end{aligned}
\end{equation}
Note that $I$ is symmetric with respect to trial and test functions in the sense that
\begin{equation}
  I(\vec{\bfphi},\vec{\bfu},\vec{\bfchi},\vec{\bfv}) = I(\vec{\bfu},\vec{\bfphi},\vec{\bfv},\vec{\bfchi}).
\end{equation}
Setting $\vec{\bfphi} = \dot{\vec{\bfu}}$ and $\vec{\bfchi} = \dot{\vec{\bfv}}$ yields the energy rate
\begin{equation}
  \dd{}{t} \left( E_u + E_v\right) = I(\dot{\vec{\bfu}},\vec{\bfu},\dot{\vec{\bfv}},\vec{\bfv}) = \frac{1}{2} \dd{}{t} I(\vec{\bfu},\vec{\bfu},\vec{\bfv},\vec{\bfv}).
\end{equation}
We define the discrete energy
\begin{equation} \label{eq:energy_intf}
\begin{aligned}
  E_I = E_u + E_v &- \frac{1}{2} I(\vec{\bfu},\vec{\bfu},\vec{\bfv},\vec{\bfv}) \\
  = E_u + E_v &+ \frac{1}{2} \dintintphys{\bfu_\js - \bfv_{\js}}{\satsym_{\js \ls} (\bfu_\ls - \bfv_\ls)} \\
   &- \frac{1}{2} \dintintphys{\bfu_\js - \bfv_\js}{\tractionop^u_{\js \ls} \bfu_{\ls} - \tractionop^v_{\js \ls} \bfv_{\ls}},
\end{aligned}
\end{equation}
which satisfies
\begin{equation}
  \dd{E_I}{t} = 0.
\end{equation}
Note that $E_I$, just like $E_u + E_v$, approximates $\mathcal{E}_u + \mathcal{E}_v$, because the surface integrals in $I$ would be zero if the interface conditions were fulfilled exactly.

\begin{theorem}
The scheme \eqref{eq:scheme_interface} is stable.
\end{theorem}

\begin{proof}
We have shown that the scheme conserves the discrete energy $E_I$. It remains to prove that $E_I$ is a non-negative quantity. To keep the notation concise in the following, let
\begin{equation}
  \bfju_\js = \bfu_\js - \bfv_\js
\end{equation}
denote the jump in displacement.
The positivity property \eqref{eq:borrowing_phys} yields (cf. \eqref{eq:borrowing_H_multiD})
\begin{equation}
\begin{aligned}
  2E_u &\geq \frac{h_1}{d} \dintintphys{D_\is \bfu_\js}{\surfjacobian^{-1} J^u \stiffphys^u_{\is \js \ks \ls} D_\ks \bfu_{\ls}}, \\
  2E_v &\geq \frac{h_1}{d} \dintintphys{D_\is \bfv_\js}{\surfjacobian^{-1} J^v \stiffphys^v_{\is \js \ks \ls} D_\ks \bfv_{\ls}}.
\end{aligned}
\end{equation}
We set $\satsym_{\js \ks} = \satsym_{\js \ks}^u + \satsym_{\js \ks}^v$ and obtain
\begin{equation}
  2E_I \geq A_u + A_v,
\end{equation}
where
\begin{equation*}
\begin{aligned}
  A_u &= \frac{h_1}{d} \dintintphys{D_\is \bfu_\js}{\surfjacobian^{-1} J^u \stiffphys^u_{\is \js \ks \ls} D_\ks \bfu_{\ls}} + \dintintphys{\bfju_\js }{\satsym_{\js \ls}^u \bfju_\ls} - \dintintphys{\bfju_\js}{\tractionop^u_{\js \ls} \bfu_\ls}, \\
  A_v &= \frac{h_1}{d} \dintintphys{D_\is \bfv_\js}{\surfjacobian^{-1} J^v \stiffphys^v_{\is \js \ks \ls} D_\ks \bfv_{\ls}} + \dintintphys{\bfju_\js }{\satsym_{\js \ls}^v \bfju_\ls} + \dintintphys{\bfju_\js}{\tractionop^v_{\js \ls} \bfv_\ls}.
\end{aligned}
\end{equation*}
We choose $\satsym_{\js \ks}^u$ so that $A_u$ is non-negative. Using the definition of $\tractionop_{\js \ls}$ \eqref{eq:discr_traction_op} yields
\begin{equation}
\begin{aligned}
  \dintintphys{\bfju_\js}{\tractionop^u_{\js \ls} \bfu_\ls} = \dintintphys{\bfju_\js}{n^u_\is \stiffphys^u_{\is \js \ks \ls} D_\ks \bfu_{\ls}} = \dintintphys{n^u_\is  \bfju_\js}{\stiffphys^u_{\is \js \ks \ls} D_\ks \bfu_{\ls}}.
\end{aligned}
\end{equation}
Due to the major symmetry of $\stiffphys_{\is \js \ks \ls}$ \eqref{eq:C_symmetry}, we have
\begin{equation}
\dintintphys{n^u_\is  \bfju_\js}{\stiffphys^u_{\is \js \ks \ls} D_\ks \bfu_{\ls}}
= \dintintphys{n^u_\ks  \bfju_\ls}{\stiffphys^u_{\is \js \ks \ls} D_\is \bfu_{\js}}.
\end{equation}
Completing the squares in $A_u$ yields
\begin{equation*}
\begin{aligned}
  A_u &= \frac{h_1}{d} \dintintphys{D_\is \bfu_{\js} - \frac{d \surfjacobian }{2h_1 J^u} n^u_\is\bfju_\js}{ \frac{J^u}{\surfjacobian} \stiffphys^u_{\is \js \ks \ls} \left(D_\ks \bfu_\ls - \frac{d \surfjacobian }{2h_1 J^u} n^u_\ks\bfju_\ls \right)} \\
  &-  \dintintphys{n^u_\is\bfju_\js}{ \frac{d\surfjacobian}{4h_1 J^u} \stiffphys^u_{\is \js \ks \ls} n^u_\ks\bfju_\ls } + \dintintphys{\bfju_\js }{\satsym_{\js \ls}^u \bfju_\ls} \\
  &\geq \dintintphys{\bfju_{\js}}{ \left( \satsym^u_{\js \ls} -  \frac{d \surfjacobian }{4h_1 J^u} n^u_\is \stiffphys^u_{\is \js \ks \ls} n^u_\ks \right) \bfju_\ls },
\end{aligned}
\end{equation*}
which is non-negative if
\begin{equation}
  \satsym^u_{\js \ls} = \beta \frac{d \surfjacobian }{4h_1 J^u} n^u_\is \stiffphys^u_{\is \js \ks \ls} n^u_\ks, \quad \beta \geq 1.
\end{equation}
A similar derivation yields
\begin{equation}
  A_v \geq \dintintphys{\bfju_{\js}}{ \left( \satsym^v_{\js \ls} -  \frac{d \surfjacobian}{4h_1 J^v} n^v_\is \stiffphys^v_{\is \js \ks \ls} n^v_\ks \right) \bfju_\ls },
\end{equation}
which is non-negative if
\begin{equation}
  \satsym^v_{\js \ls} = \beta \frac{d \surfjacobian}{4h_1 J^v} n^v_\is \stiffphys^v_{\is \js \ks \ls} n^v_\ks, \quad \beta \geq 1.
\end{equation}
We conclude that $E_I$ is non-negative if
\begin{equation}
  \satsym_{\js \ls} = \satsym^u_{\js \ls} + \satsym^v_{\js \ls} = \beta \frac{d\surfjacobian}{4h_1} \left( \frac{n^u_\is \stiffphys^u_{\is \js \ks \ls} n^u_\ks}{J^u} + \frac{n^v_\is \stiffphys^v_{\is \js \ks \ls} n^v_\ks}{J^v} \right), \quad \beta \geq 1.
\end{equation}
\end{proof}
In all simulations in this paper, we set $\beta = 1$, i.e., right on the limit of provable stability.

\subsection{Self-adjointness}
Let $\physdomfull = \physdom_u \cup \physdom_v$ denote the full domain. Introduce the notation
\begin{equation}
  w_\js = \left\{
  \begin{array}{ll}
  u_\js, & \vec{X} \in \physdom_u \\
  v_\js, & \vec{X} \in \physdom_v
  \end{array}
  \right.
  \quad \mbox{and} \quad
  \psi_\js = \left\{
  \begin{array}{ll}
  \phi_\js, & \vec{X} \in \physdom_u \\
  \chi_\js, & \vec{X} \in \physdom_v
  \end{array}
  \right.,
\end{equation}
where we think of $w_\js$ as the primal field and $\psi_\js$ as the adjoint field. The continuous elastic operator satisfies
\begin{equation}
  \contOpInt_{\js \ls} w_{\ls} = \left\{
  \begin{array}{ll}
  \ddxi \stiffphys^u_{\is \js \ks \ls} \ddxk u_\ls, & \vec{X} \in \physdom_u \\
  \ddxi \stiffphys^v_{\is \js \ks \ls} \ddxk v_\ls, & \vec{X} \in \physdom_v
  \end{array}
  \right. .
\end{equation}
Requiring that $\contOpInt_{\js \ls} w_{\ls}$ be square-integrable over each subdomain leads us to define the space
\begin{equation}
  \mathcal{V} = \left\{ w_\js \in L^2(\physdomfull)
  \; \Bigg | \;
  \begin{array}{l}
  \smallskip
  \ddxi \stiffphys^u_{\is \js \ks \ls} \ddxk u_\ls \in L^2(\Omega_u) \\
  \ddxi \stiffphys^v_{\is \js \ks \ls} \ddxk v_\ls \in L^2(\Omega_v) \\
  \end{array}
  \right\}.
\end{equation}
We further require that $w_\js$ satisfies appropriate interface and boundary conditions. We define
\begin{equation}
  \mathcal{W} = \left\{ w_\js \in \mathcal{V}
  \; \Bigg | \;
  \begin{array}{rl}
  u_\js - v_\js = 0 &\mbox{on } \Gamma \\
  \ctractionop^u_{\js \ls} u_\ls + \ctractionop^v_{\js \ls} v_\ls = 0 &\mbox{on } \Gamma \\
  L_{\js \ls} w_\ls = 0 &\mbox{on } \physboundaryfull
  \end{array}
  \right\},
\end{equation}
where the boundary operator may be either $L_{\js \ls} = \ctractionop_{\js \ls} + U_{\js \ls}$, for Robin boundary conditions, or $L_{\js \ls} = \delta_{\js \ls}$, for displacement boundary conditions.
The operator $\contOpInt_{\js \ls} : \mathcal{W} \rightarrow L^2(\physdomfull)$ is self-adjoint, because integrating by parts twice yields
\begin{equation}
\begin{aligned}
  \volintphysfull{\psi_\js}{\contOpInt_{\js \ls} w_\ls}
  &= \intintphys{\phi_\js}{\ctractionop^u_{\js \ls} u_{\ls} } + \intintphys{\chi_\js}{\ctractionop^v_{\js \ls} v_{\ls} } \\
  &- \intintphys{\ctractionop^u_{\js \ls} \phi_\ls}{u_{\js} } - \intintphys{\ctractionop^v_{\js \ls} \chi_\ls}{ v_{\js} } \\
  &+ \volintphysu{\ddxi \stiffphys^u_{\is \js \ks \ls} \ddxk \phi_\ls}{u_\js} + \volintphysv{\ddxi \stiffphys^v_{\is \js \ks \ls} \ddxk \chi_\ls}{ v_\js} \\
  &= \volintphysfull{\contOpInt_{\js \ls} \psi_\ls}{w_\js}  \; \forall  w_\js, \psi_\js \in \mathcal{W}.
\end{aligned}
\end{equation}

Now consider the discrete elastic operator, including interface SATs. Let
\begin{equation}
\bfw_\js = \begin{bmatrix} \bfu_\js \\ \bfv_\js \end{bmatrix}, \quad \bfpsi_\js = \begin{bmatrix} \bfphi_\js \\ \bfchi_\js \end{bmatrix} .
\end{equation}
Omitting SATs for boundary conditions, the total discrete elastic operator in \eqref{eq:scheme_interface} can be written as
\begin{equation}
  \cD^{tot}_{\js \ls} =
  \begin{bmatrix}
  \cD_{\is \ks}^{\physdom_u}(\stiffphys^u_{\is \js \ks \ls}) + \mathbb{S}_{\js \ls}^{uu} & \mathbb{S}_{\js \ls}^{uv} \\
  \mathbb{S}_{\js \ls}^{vu} & \cD_{\is \ks}^{\physdom_v}(\stiffphys^v_{\is \js \ks \ls}) + \mathbb{S}_{\js \ls}^{vv}
  \end{bmatrix},
\end{equation}
where the $\mathbb{S}_{\js \ls}$ operators correspond to the interface SATs.
We define discrete integrals over the full domain as the sum of integrals over the subdomains,
\begin{equation}
  \dvolintphysfull{\bfpsi_\js}{\bfw_\js} = \dvolintphysu{\bfphi_\js}{\bfu_\js} + \dvolintphysv{\bfchi_\js}{\bfv_\js} .
\end{equation}

\begin{theorem} \label{thm:selfadjoint_interface}
The total discrete elastic operator $\cD^{tot}_{\js \ls}$, corresponding to the\\ scheme \eqref{eq:scheme_interface}, including interface SATs, is self-adjoint, i.e.,
\begin{equation}
  \dvolintphysfull{\bfpsi_\js}{\cD^{tot}_{\js \ls} \bfw_\ls} = \dvolintphysfull{\cD^{tot}_{\js \ls} \bfpsi_\ls}{\bfw_\js} \quad \forall \bfpsi_\js, \bfw_\js .
\end{equation}
\end{theorem}

\begin{proof}
In deriving the weak form \eqref{eq:weak_added}, we showed that
\begin{equation}
\begin{aligned}
  \dvolintphysfull{\bfpsi_\js}{\cD^{tot}_{\js \ls} \bfw_\ls} &= \dvolintphysu{\bfphi_\js}{(\cD_{\is \ks}^{\physdom_u}(\stiffphys^u_{\is \js \ks \ls}) + \mathbb{S}_{\js \ls}^{uu}) \bfu_\ls + \mathbb{S}_{\js \ls}^{uv} \bfv_\ls} \\
  &+ \dvolintphysv{\bfchi_\js}{(\cD_{\is \ks}^{\physdom_v}(\stiffphys^v_{\is \js \ks \ls}) + \mathbb{S}_{\js \ls}^{vv}) \bfv_\ls + \mathbb{S}_{\js \ls}^{vu} \bfu_\ls} \\
  &= - K_u(\vec{\bfphi}, \vec{\bfu}) - K_v(\vec{\bfchi}, \vec{\bfv}) + I(\vec{\bfphi},\vec{\bfu},\vec{\bfchi},\vec{\bfv}),
\end{aligned}
\end{equation}
where we are omitting all terms corresponding to outer boundaries for convenience.
Using the symmetries of $K_{u,v}$ and $I$ and the symmetry of $\dvolintphysfull{\cdot}{\cdot}$ yields
\begin{equation}
\begin{aligned}
  \dvolintphysfull{\bfpsi_\js}{\cD^{tot}_{\js \ls} \bfw_\ls} &= - K_u(\vec{\bfphi}, \vec{\bfu}) - K_v(\vec{\bfchi}, \vec{\bfv}) + I(\vec{\bfphi},\vec{\bfu},\vec{\bfchi},\vec{\bfv}) \\
  &= - K_u(\vec{\bfu}, \vec{\bfphi}) - K_v(\vec{\bfv}, \vec{\bfchi}) + I(\vec{\bfu},\vec{\bfphi},\vec{\bfv},\vec{\bfchi}) \\
  &= \dvolintphysfull{\bfw_\js}{\cD^{tot}_{\js \ls} \bfpsi_\ls} = \dvolintphysfull{\cD^{tot}_{\js \ls} \bfpsi_\ls}{\bfw_\js}.
\end{aligned}
\end{equation}
\end{proof}

\section{Numerical experiments} \label{sec:num_exp}
This section contains three numerical
experiments. First, we use the method of manufactured solutions to assess the global convergence rates of the new SBP-SAT schemes based on the fully compatible operators adapted from Mattsson's operators \cite{Mattsson11}. Second, we use the new methods to evaluate the performance of an elastodynamic cloak. Third, we solve an application problem inspired by seismic exploration in mountainous regions.

 Before presenting the numerical experiments, we briefly discuss how the time step is selected. \refthree{In general we cannot rely on von Neumann analysis, which is restricted to constant coefficient problems in unbounded domains; the maximum stable time step is often influenced by boundary closures and penalty terms. Computing eigenvalues of the spatial discretization allows one to precisely determine the stability limit, but eigenvalue computations are prohibitively expensive for large problems. Instead, we seek a relatively cheap procedure that, for the majority of cases, yields a stable time step close to the stability limit; one can then adjust the time step around this estimate through trial-and-error. We consider the transformed problem \eqref{eq:model_problem_transf} in the reference domain $\refdom$, because the grid spacing $h$ is well defined in $\refdom$. We choose the time step according to
 \begin{equation}
  \Delta t = \mbox{CFL} \times \min \limits_{\mbox{\tiny{all gridpoints}}} \; \frac{h}{v_{max}},
\end{equation}
where $v_{max}$ denotes the largest quasi-P-wave speed. Given a direction of propagation, computing $v_{max}$ amounts to inserting the transformed material properties into the Christoffel equation \eqref{eq:christoffel} and finding the largest root of a degree $d$ polynomial whose coefficients are functions of density and stiffness \cite{Synge1956}. Finding the direction of fastest propagation is a $(d-1)$-dimensional optimization problem, which we here solve using MATLAB's fmincon.
The dimensionless constant $\mbox{CFL}$ depends on the order of accuracy. Appropriate values of $\mbox{CFL}$ are determined empirically but are generally $\mathcal{O}(1)$.
 }

\subsection{Convergence studies} \label{sec:conv}
Consider the domain depicted in Figure \ref{fig:convDomain}.
We use the method of manufactured solutions and choose the exact solution
\begin{equation} \label{eq:mms_exact}
  u_1 = \sin(2X_1 + 3X_2 - t), \quad u_2 = \sin(3X_1 + 2X_2 - 2t),
\end{equation}
and the material parameters
\begin{equation}
  \rho = 2 + \sin\left(\reftwo{\frac{X_1+X_2}{2}} \right), \quad \stiffphys_{\il \jl \kl \elll} = \left\{ \begin{array}{ll}
  \medskip
  \alpha_{\is\js\ks\ls}, & \il=\kl \mbox{ and } \jl=\elll \\
  \beta_{\is \js \ks \ls}, & \mbox{otherwise} \\
  \end{array}
  \right. ,
\end{equation}
where
\begin{equation}
  \alpha_{\is\js\ks\ls} = 8 + \sin(\il X_1 + \jl X_2) + \frac{1}{2} \sin(\kl X_1- \elll X_2)
\end{equation}
and
\begin{equation}
  \begin{aligned}
  \beta_{\is \js \ks \ls} &= \frac{1}{8} \left( \sin(\il X_1 + \jl X_2) + \sin(\kl X_1 + \elll X_2)  \right) \\
  &+ \frac{1}{16} \left( \sin(\il X_1-\jl X_2) + \sin(\kl X_1-\elll X_2) \right).
  \end{aligned}
\end{equation}
We impose traction conditions on the outer boundaries and displacement conditions on the interior scatterer, and use the exact solution as boundary and initial data. For time-integration we use the classical fourth order Runge--Kutta method with $\mbox{CFL} = 0.5$, which proved to be small enough to make the spatial errors dominate. We set $T=1$ as the final time. \refone{Table \ref{table:conv} and Figure \ref{fig:convergence} show} the $\ell^2$ errors as functions of $h$, where $h$ denotes the grid spacing in the reference domain. \refthree{Table \ref{table:conv} also shows the number of grid points per solution wavelength (PPWL) used near the outer boundaries, where the grid spacing is the largest. The exact solution \eqref{eq:mms_exact} is not a plane wave but both $u_1$ and $u_2$ equal waves with wavelength $2 \pi / \sqrt{13}$; hence we use $2 \pi/ \sqrt{13}$ as the ``wavelength'' when computing the PPWL.}
\begin{figure}[h]
    \begin{subfigure}[b]{.49\linewidth}
        \centering
        \includegraphics[width= 0.75\linewidth]{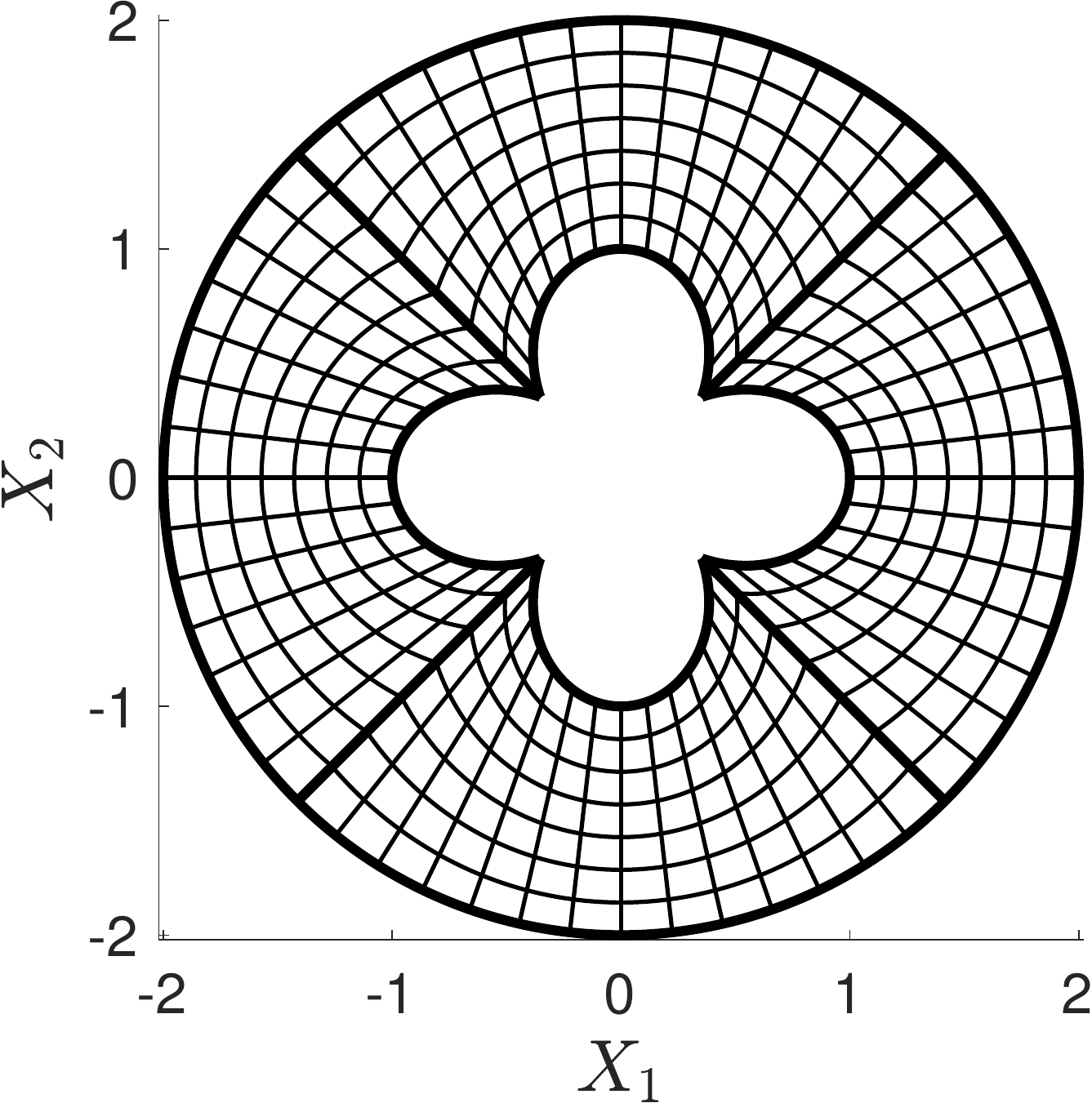}
        \caption{Domain and grid configuration}\label{fig:convDomain}
    \end{subfigure}
    \begin{subfigure}[b]{.49\linewidth}
        \centering
        \includegraphics[width=0.99\linewidth]{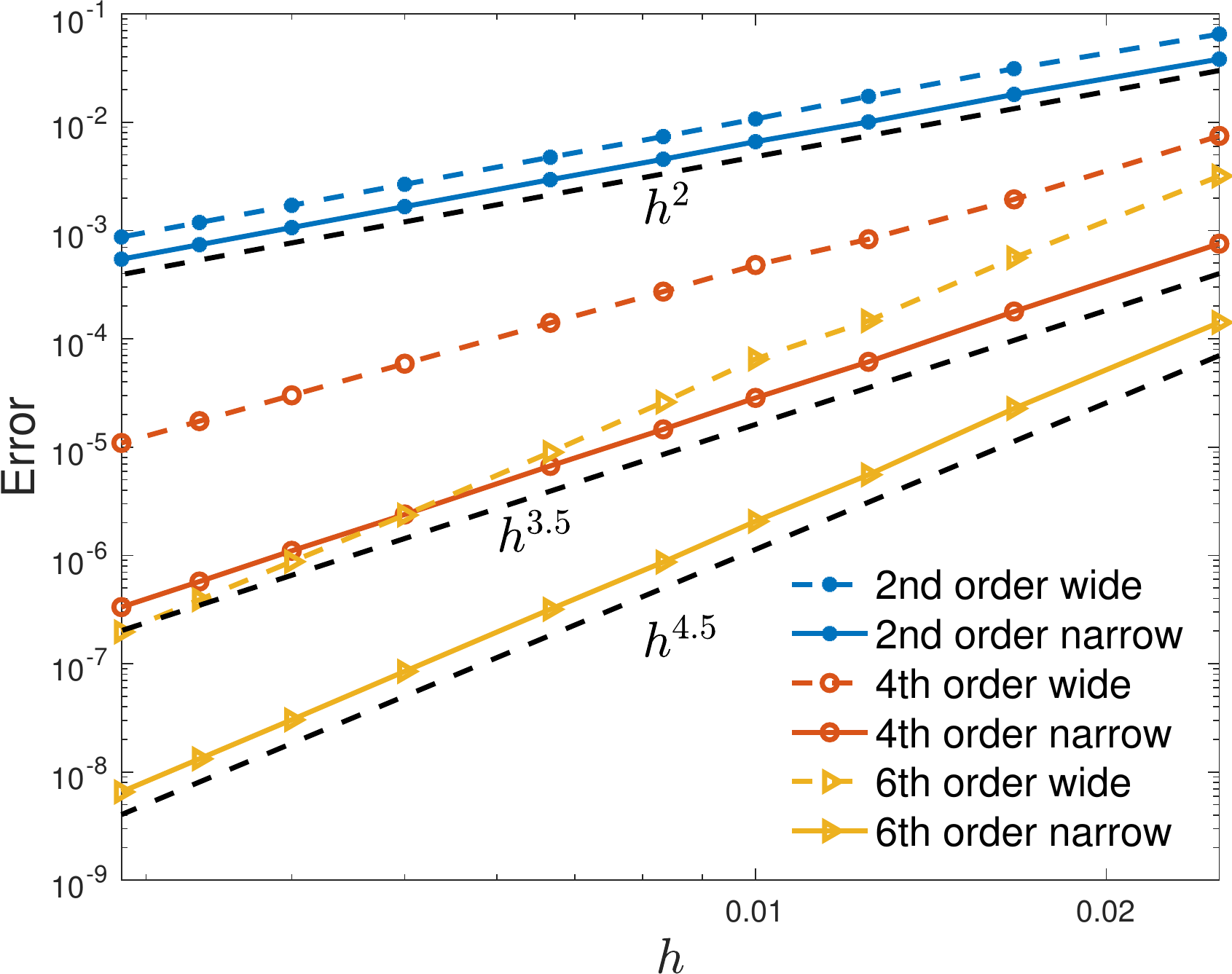}
        \caption{Convergence plot}\label{fig:convergence}
    \end{subfigure}
    \caption{(a) Multiblock grid used in the computations. (b) Convergence plot\reftwo{, comparing the narrow and wide stencils}. $h$ denotes the
smallest grid spacing in the reference domain. }
\end{figure}
\begin{table}[h]
\centering
\begin{tabular}{c|c|cc|cc|cc} & &
\multicolumn{2}{|c}{second order}& \multicolumn{2}{|c}{fourth order}& \multicolumn{2}{|c}{sixth order} \\
\hline
$h^{-1}$& \refthree{PPWL} &$log_{10}(\mbox{error})$ & $r$& $log_{10}(\mbox{error})$ & $r$& $log_{10}(\mbox{error})$ & $r$ \\
\hline
40 & \refthree{22} &   -1.42 &          &    -3.12 &          &    -3.85 &         \\
60 & \refthree{33} &   -1.74 &     1.85 &    -3.75 &     3.56 &    -4.64 &     4.51\\
80 & \refthree{44} &   -2.00 &     2.03 &    -4.21 &     3.72 &    -5.26 &     4.90\\
100 & \refthree{55} &   -2.18 &     1.88 &    -4.55 &     3.44 &    -5.69 &     4.44\\
120 & \refthree{67} &   -2.34 &     2.05 &    -4.84 &     3.67 &    -6.06 &     4.75\\
150 & \refthree{83} &   -2.53 &     1.94 &    -5.17 &     3.48 &    -6.50 &     4.50\\
200 & \refthree{111} &   -2.78 &     1.99 &    -5.62 &     3.59 &    -7.07 &     4.59\\
250 & \refthree{139} &   -2.97 &     2.00 &    -5.96 &     3.46 &    -7.52 &     4.61\\
300 & \refthree{166} &   -3.13 &     2.00 &    -6.24 &     3.58 &    -7.88 &     4.55\\
350 & \refthree{194} &   -3.26 &     2.00 &    -6.48 &     3.52 &    -8.18 &     4.56\\
\hline
\multicolumn{2}{c|}{\reftwo{avg.\ rate}} & & \reftwo{1.96} & & \reftwo{3.56} & & \reftwo{4.60}
\end{tabular}
\caption{$\ell^2$ errors and convergence rates $r$ \reftwo{for the anisotropic problem, using the narrow stencil.} }
\label{table:conv}
\end{table}
The convergence rates appear to be 2, 3.5, and 4.5, for interior orders two, four, and six. Recall that the adapted operators used here have reduced boundary accuracy $q_b = q-1$. In numerical experiments with second-derivative SBP operators, the convergence rate is often observed to be $\min(q_b+2,2q)$. For the adapted operators, this rule of thumb predicts rates 2, 3, and 4, and for operators with full boundary accuracy $q_b = q$, it predicts rates 2, 4, and 5. The second order adapted operator yields rate 2, as predicted by the rule of thumb. For orders four and six, the adapted operators suffer from a reduction by only half an order compared to their $q_b=q$ counterparts. Their rates are half an order higher than predicted by the rule of thumb. Explaining this ``super convergence'' will have to be the topic of another paper. For now, we conclude that---as fully compatible operators with full boundary accuracy are currently lacking---the adapted operators provide a reasonable compromise that allows for a straightforward stability proof at the cost of no more than half an order reduction of global accuracy.

\refone{To demonstrate the advantage of the narrow-stencil second-derivative operators over the wide-stencil operators, which results from applying first-derivative SBP operators twice, Figure \ref{fig:convergence} shows the convergence behavior for both methods. We use exactly the same SATs in both cases. It is straightforward to verify that the wide-stencil operator also is a fully compatible SBP operator and that the resulting scheme is energy-stable and self-adjoint. In the second-order accurate case, the narrow-stencil method is slightly more accurate. For higher orders, the narrow-stencil method is more than an order of magnitude more accurate. The spectral radius of the spatial operator is for this problem slightly larger for the narrow scheme than for the wide. For the grid corresponding to $h=0.01$ the relative differences in spectral radius are:
\begin{equation*}
  \mbox{second order: }5.34 \%, \quad \mbox{fourth order: } 1.28 \%, \quad \mbox{sixth order: } 1.23 \%.
\end{equation*}
Note that the largest stable time-step is approximately proportional to the square root of the spectral radius. Hence, compared to the big difference in accuracy, the slight increase in spectral radius has very little impact on performance.
}

Although \cite{Duru201437} did not observe any accuracy reduction for the adapted operators applied to isotropic materials, we can hereby conclude that schemes based on the adapted operators of orders $2q=4$ and $2q=6$ both suffer a reduction by half an order, at least for general anisotropic materials. \reftwo{To investigate also the isotropic case, we use the same exact solution and domain as for the anisotropic problem but with spatially uniform isotropic material properties $\rho = 1$, $\lambda=\mu=1$. We impose traction conditions on all boundaries. The results are shown in Table \ref{table:convIso}. We observe a clear reduction for $2q=6$, for an average rate of $4.54$. For order $2q=4$, it is not entirely obvious whether the asymptotic rate (average $3.70$) is $3.5$ or $4$. We conclude that reductions in convergence rate can manifest even in the isotropic case and the rates of $q+2$, as observed in \cite{Duru201437}, cannot be expected in general. Fully understanding this matter is, however, out of the scope of the present study.
}
\begin{table}[h]
\centering
\begin{tabular}{c|c|cc|cc|cc} & &
\multicolumn{2}{|c}{second order}& \multicolumn{2}{|c}{fourth order}& \multicolumn{2}{|c}{sixth order} \\
\hline
$h^{-1}$& PPWL &$log_{10}(\mbox{error})$ & $r$& $log_{10}(\mbox{error})$ & $r$& $log_{10}(\mbox{error})$ & $r$ \\
\hline
40 & 22 &   -1.20 &          &    -2.87 &          &    -3.62 &         \\
60 & 33 &   -1.52 &     1.80 &    -3.49 &     3.49 &    -4.38 &     4.30\\
80 & 44 &   -1.76 &     1.95 &    -3.96 &     3.78 &    -4.95 &     4.53\\
100 & 55 &   -1.94 &     1.84 &    -4.31 &     3.59 &    -5.36 &     4.23\\
120 & 67 &   -2.10 &     1.98 &    -4.61 &     3.85 &    -5.73 &     4.66\\
150 & 83 &   -2.28 &     1.90 &    -4.97 &     3.71 &    -6.16 &     4.47\\
200 & 111 &   -2.52 &     1.94 &    -5.44 &     3.78 &    -6.74 &     4.66\\
250 & 139 &   -2.71 &     1.95 &    -5.81 &     3.77 &    -7.20 &     4.72\\
300 & 166 &   -2.87 &     1.95 &    -6.11 &     3.78 &    -7.57 &     4.76\\
350 & 194 &   -3.00 &     1.96 &    -6.36 &     3.76 &    -7.90 &     4.79\\
\hline
\multicolumn{2}{c|}{avg.\ rate} & & 1.91 & & 3.70 & & 4.54
\end{tabular}
\caption{\reftwo{$\ell^2$ errors and convergence rates $r$ for the isotropic problem, using the narrow stencil.}}
\label{table:convIso}
\end{table}

\reftwo{
\subsection{Stability and self-adjointness}
To verify that the schemes are energy conserving and self-adjoint, we again use the domain in Figure \ref{fig:convDomain}. We use random material properties. Let $\widetilde{\rho}$ be a grid function of random numbers drawn from the standard uniform distribution $\mathcal{U}(0,1)$. Similarly, let all independent components of $\widetilde{C}_{\is \js \ks \ls}$ be drawn from $\mathcal{U}(0,1)$ (remaining components are determined by the major symmetry). We then set the discrete material properties
\begin{equation}
  \rho = 1 + \widetilde{\rho}, \quad \stiffphys_{\il \jl \kl \elll} = \left\{ \begin{array}{ll}
  \medskip
  \widetilde{C}_{\is\js\ks\ls} + 4, & \il=\kl \mbox{ and } \jl=\elll \\
  \widetilde{C}_{\is\js\ks\ls}, & \mbox{otherwise} \\
  \end{array}
  \right. .
\end{equation}
Theorems \ref{thm:selfadjoint_boundary} and \ref{thm:selfadjoint_interface} prove that the total discrete elastic operator $\cD_{\js \ls}^{tot}$ is self-adjoint in the inner product defined by the physical quadrature $JH$. In two spatial dimensions, this is equivalent to the matrix $A$ being symmetric, where
\begin{equation}
  A = \mathcal{H} \begin{bmatrix}  \cD_{11}^{tot} &  \cD_{12}^{tot} \\
                       \cD_{21}^{tot} &  \cD_{22}^{tot}\end{bmatrix} \quad \mbox{and} \quad \mathcal{H} = \begin{bmatrix} JH & \\ & JH \end{bmatrix}.
\end{equation}
We set the smallest grid spacing in the reference domain to $h=0.01$, which leads to a total of $19796$ grid points.
The relative deviations from symmetry $\norm{A-A^T}_{\text{max}}/\norm{A}_{\text{max}}$ for this problem are:
\begin{equation*}
\begin{aligned}
  &\mbox{second order: }1.61\times10^{-16}, \quad \mbox{fourth order: }3.64\times10^{-16}, \\
  &\; \mbox{   sixth order: }2.08\times10^{-16},
\end{aligned}
\end{equation*}
which verifies that the schemes are self-adjoint to machine precision.

In the absence of external forces and boundary data, the semidiscrete equations take the form
\begin{equation}
  \mathcal{P} \mathcal{H} \ddot{\vec{\bfu}} = A \vec{\bfu}, \quad \mbox{where } \mathcal{P} = \begin{bmatrix} \rho & \\ & \rho \end{bmatrix}.
\end{equation}
Since $A$ is symmetric, the semidiscrete problem preserves the quantity
\begin{equation}
  \varepsilon = \frac{1}{2} \left( \left(\dot{\vec{\bfu}}\right)^T \mathcal{P} \mathcal{H} \dot{\vec{\bfu}} - \vec{\bfu}^T A \vec{\bfu} \right),
\end{equation}
which is precisely the semidiscrete energy given by \eqref{eq:energy_robin}, \eqref{eq:e_d}, and \eqref{eq:energy_intf}. Our stability analysis further guarantees that the semidiscrete energy is non-negative, and hence a seminorm of $\vec{\bfu}$. That is, we have proved that, with proper SATs, $A$ is negative semidefinite. For the random material properties above and $h=0.01$, the largest eigenvalues of $hA$ are:
\begin{equation*}
\begin{aligned}
  &\mbox{second order: } -6.981 \times10^{-5}, \quad \mbox{fourth order: } -6.977\times10^{-5}, \\
  & \; \mbox{   sixth order: } -6.976\times10^{-5},
\end{aligned}
\end{equation*}
which verifies that $A$ is negative semidefinite.
}

\subsection{Elastodynamic cloaking} \label{sec:cloaking}
Elastic cloaking is the art of making an object impossible to detect by means of elastic waves by surrounding the object with carefully chosen materials. \reftwo{These material properties are chosen such that waves, incident from any direction, pass around the object and reform on the other side in such a way that the wavefield outside the cloak is (approximately) the same as if the object were absent.} \refthree{Elastodynamic cloaking may be used to conceal military objects \cite{Olsson2011}, shield buildings from seismic waves from earthquakes, and reduce vibrations in cars \cite{Farhat2009}.  } \reftwo{
To design a cloak, we utilize coordinate transformation theory \cite{Milton2006, Norris2011, craster2018}. As an example, we cloak the impenetrable object shown in black in Figure \ref{fig:cloakOrig}. We assume that the background medium is homogeneous with $\rho=1$ and $\lambda=\mu=1$ and model impenetrability by imposing homogeneous displacement conditions on the object's surface. To construct the cloak we proceed as follows:
\begin{enumerate}
 \item Choose the exterior boundary of the cloak (see Figure \ref{fig:cloakOrig}). Let $\Omega^c$ denote the region that the cloak will occupy. That is, the material parameters will be adjusted only within $\Omega^c$.

 \item Introduce a fictitious object, significantly smaller than the original object, which the cloaked object will mimic. We will refer to the region between this fictictious object and the exterior boundaries of the cloak as $\Omega^f$ (see Figure \ref{fig:cloakFict}).

 \item Given a mapping between $\Omega^f$ and $\Omega^c$ (we discuss how to obtain this mapping later), transform the equations of motion with homogeneous material properties in $\Omega^f$ to equivalent equations posed on $\Omega^c$. That is, repeat the transformation analysis in Section \ref{sec:transform} with $\physdom=\Omega^f$, $\refdom=\Omega^c$. Since the transformed equations are equivalent, filling the cloak with the transformed material guarantees that the cloaked object will be indistinguishable from the fictitious object, when probed from outside the cloak. If the fictitious object is small enough, the cloaked object will be practically undetectable.
\end{enumerate}

\begin{figure}[h]
    \begin{subfigure}[b]{.49\linewidth}
        \centering
        \includegraphics[width= 0.99\linewidth]{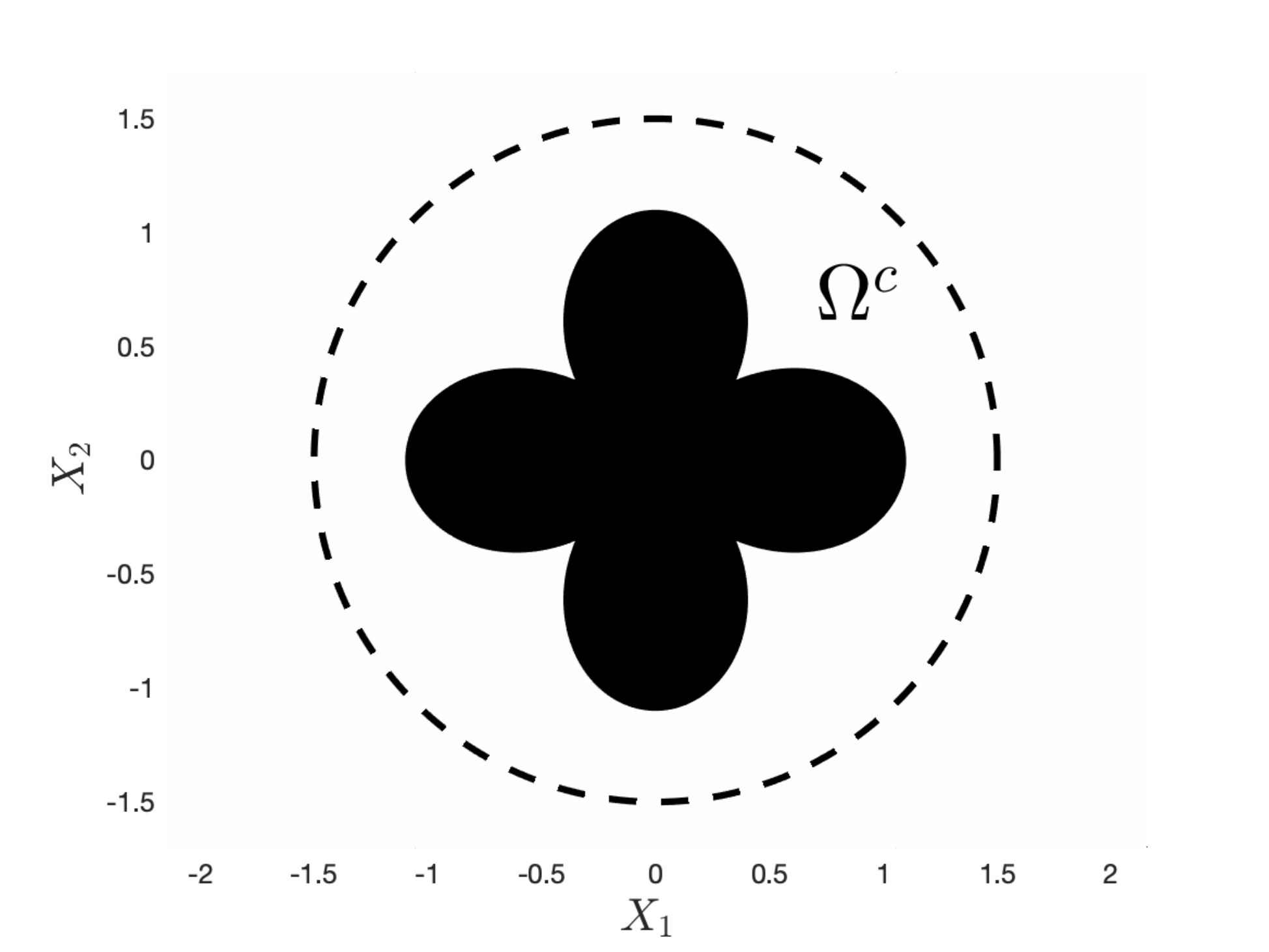}
        \caption{Object to be cloaked}\label{fig:cloakOrig}
    \end{subfigure}
    \begin{subfigure}[b]{.49\linewidth}
        \centering
        \includegraphics[width=0.99\linewidth]{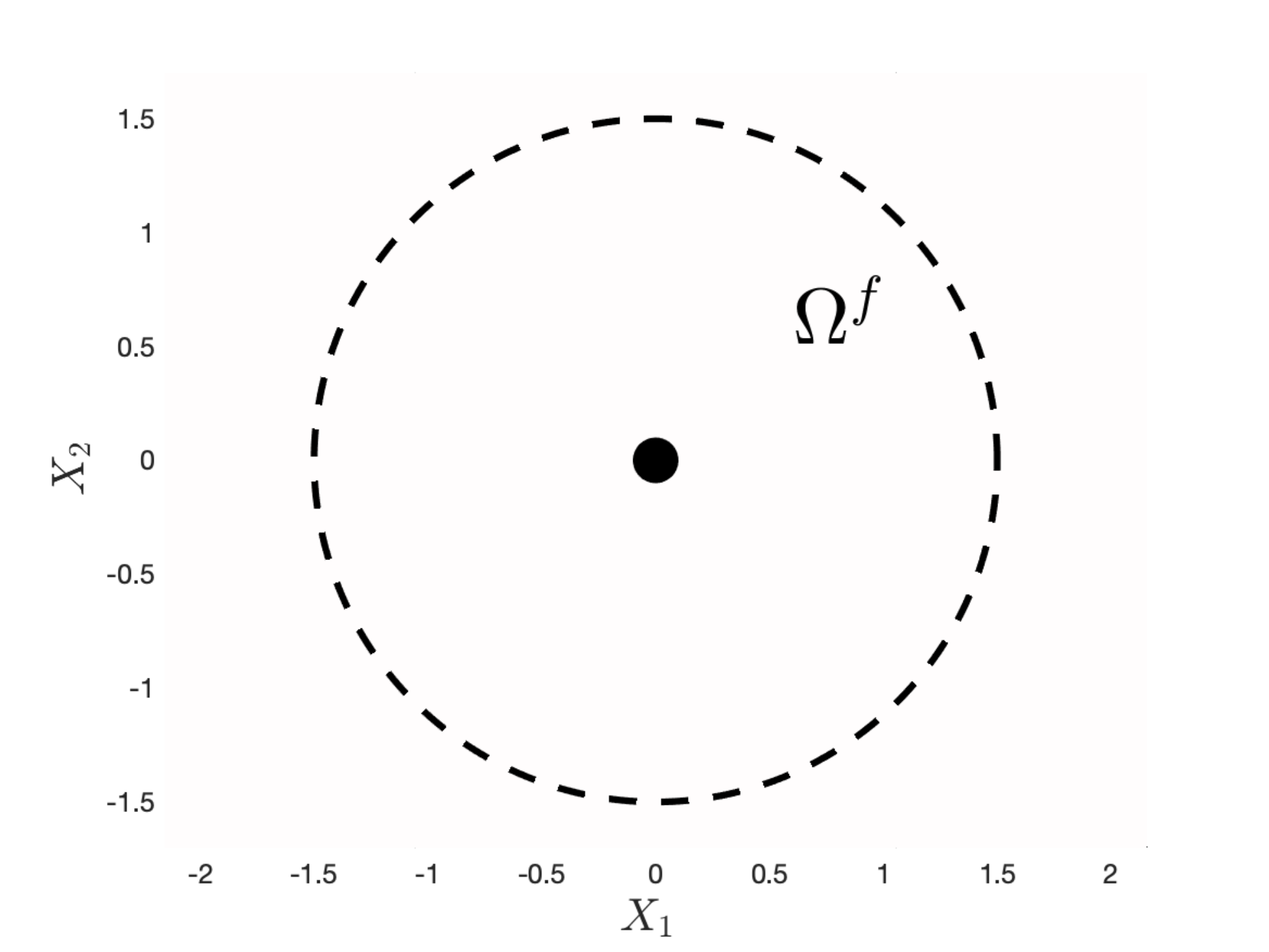}
        \caption{Fictitious object}\label{fig:cloakFict}
    \end{subfigure}
    \caption{(a) The object to be cloaked, with the dashed circle marking the extent of the cloak. $\Omega^c$ denotes the region that the cloak will occupy. (b) The small disk is a fictitious object that the cloaked object will mimic. $\Omega^f$ denotes the region between the fictitious object and the outer boundary of the cloak. }
\end{figure}

In theory we can achieve \emph{perfect} cloaking by choosing the reference object as a point, but such coordinate transformations are singular and would require singular material properties in the cloak. In practice, one usually settles for \emph{partial} or \emph{near} cloaking \cite{craster2018}, where the reference object is finite but much smaller than the original object, as in Figure \ref{fig:cloakFict}. Since the equations of Cosserat materials are invariant under coordinate transformation, they always allow for cloaking \cite{Norris2011}, at least at a mathematical level. Not all nonsingular cloaks are realizable in practice, because the material properties prescribed by the coordinate transformation may be infeasible to engineer \cite{Kadic2013,Kadic2012}.

Let us now discuss how to establish a mapping between $\Omega^c$ and $\Omega^f$ and compute the transformation gradient. If the cloak has a simple shape, for example circular or spherical, the transformation gradient can be computed analytically \cite{Brun2009,Diatta2014}. Here, we allow for more complicated objects and cloak shapes. There may be many ways to do this and our approach is just one option.
We grid $\Omega^c$ and $\Omega^f$ with grids whose block topology match so that each block can be paired with a block in the other grid. We describe the procedure for one such pair of grid blocks. With a slight abuse of notation, let $\Omega^{c,f}$ denote the regions occupied by these blocks in what follows. In the gridding process the blocks have been associated with one-to-one coordinate mappings $\vec{X}^{c,f}$ from the unit square $\refdom$ such that
\begin{equation}
  \Omega^c = \vec{X}^c(\refdom), \quad \Omega^f = \vec{X}^f(\refdom).
\end{equation}
It follows that $G = \vec{X}^{c} \circ (\vec{X}^f)^{-1}$ is a one-to-one mapping from $\Omega^f$ to $\Omega^c$.

To determine the transformed material it remains to compute an approximation of the transformation gradient
\begin{equation}
  \mathsf{F} = \pdd{\vec{X}^c}{\vec{X}^f}.
\end{equation}
We interpolate $\vec{X}^f$ to the cloak grid (this provides flexibility because we do not need to assume anything about the number of grid points in either block). The interpolation is performed between the Cartesian reference grids in $\omega$ and $\vec{X}^f$ is treated as a grid function. Next, we apply the numerical derivative operators defined in \eqref{eq:chain_discrete} (note that the transformation gradient appearing in \eqref{eq:chain_discrete} concerns the mapping to $\refdom$ and not the mapping $G$) to compute an approximation of the inverse transformation gradient
\begin{equation}
  \mathsf{F}^{-1} = \pdd{\vec{X}^f}{\vec{X}^c}.
\end{equation}
Finally, $\mathsf{F}$ is obtained by inverting $\mathsf{F}^{-1}$.
}

 \reftwo{To illustrate the spatial heterogeneity and anisotropy of the resulting cloak, we shall need to introduce some notation. Let $c_{qp}$ and $c_{qs}$ denote quasi-P- and quasi-S-wave speeds in the cloak, and let $c_p$ and $c_s$ denote the isotropic wave speeds in the homogeneous background medium. To illustrate the spatial heterogeneity of the cloak, Figure \ref{fig:quasi} shows $\ln(c_{qp}/c_p)$ in $\Omega^c$, for a wave propagating parallel to the $X_1$-axis. We use the log-scale to better illustrate the fast variations in wave speed near the scatterer.
Figure \ref{fig:slowness} illustrates the anisotropy of the cloak by showing the slowness surface at the point $[X_1,X_2]=[0.5, 1]$, with the slowness surfaces for the isotropic background material included for reference. We remark that the cloak is spatially heterogeneous and the slowness surfaces are significantly different at other points in the cloak.
}
\begin{figure}[h]
    \begin{subfigure}[b]{.49\linewidth}
        \centering
        \includegraphics[width= 0.99\linewidth, trim=0.5cm 0cm 1cm 1cm, clip=true]{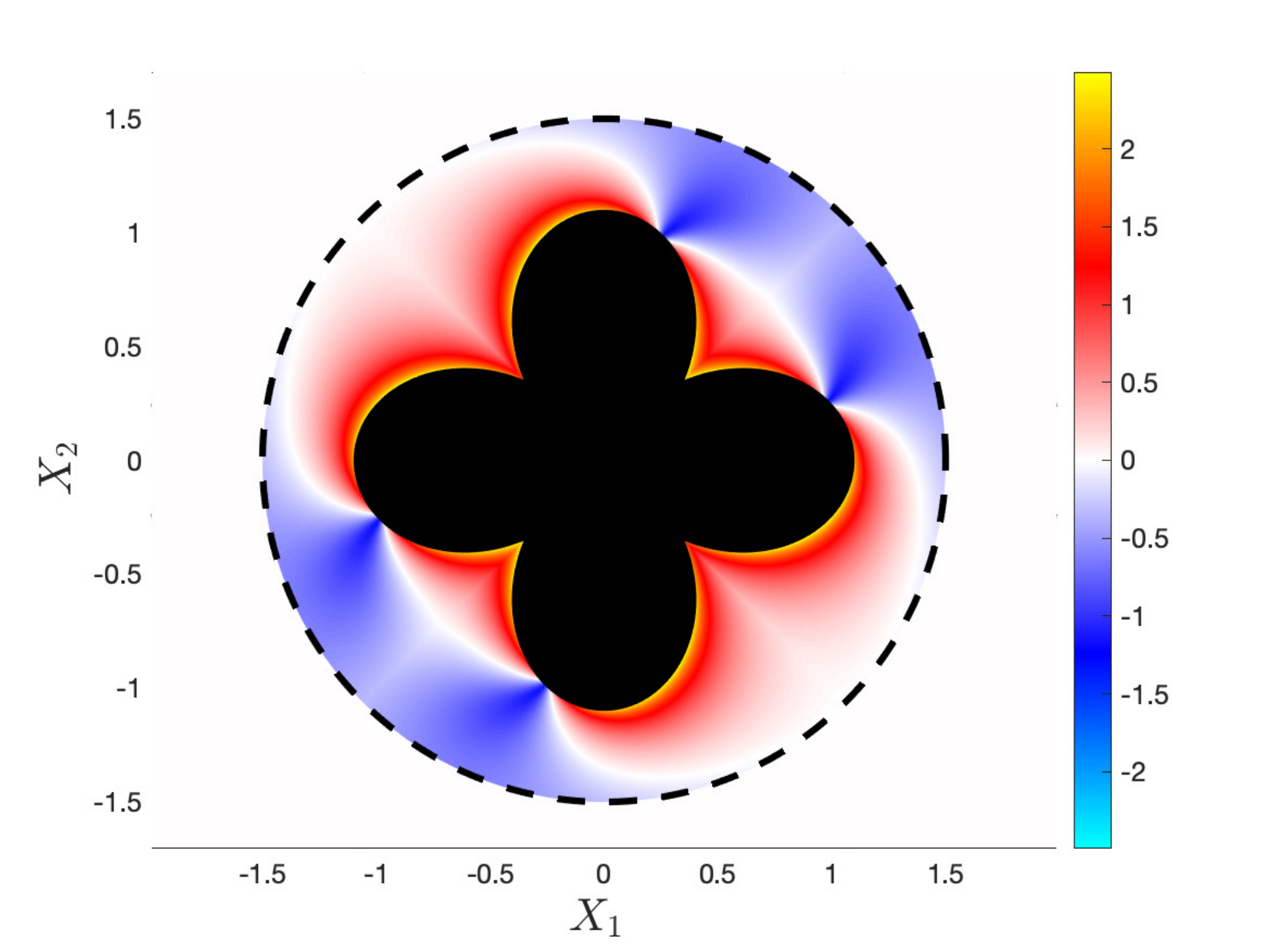}
        \caption{$\ln(c_{qp}/c_p)$}\label{fig:quasi}
    \end{subfigure}
    \begin{subfigure}[b]{.49\linewidth}
        \centering
        \includegraphics[width=0.99\linewidth]{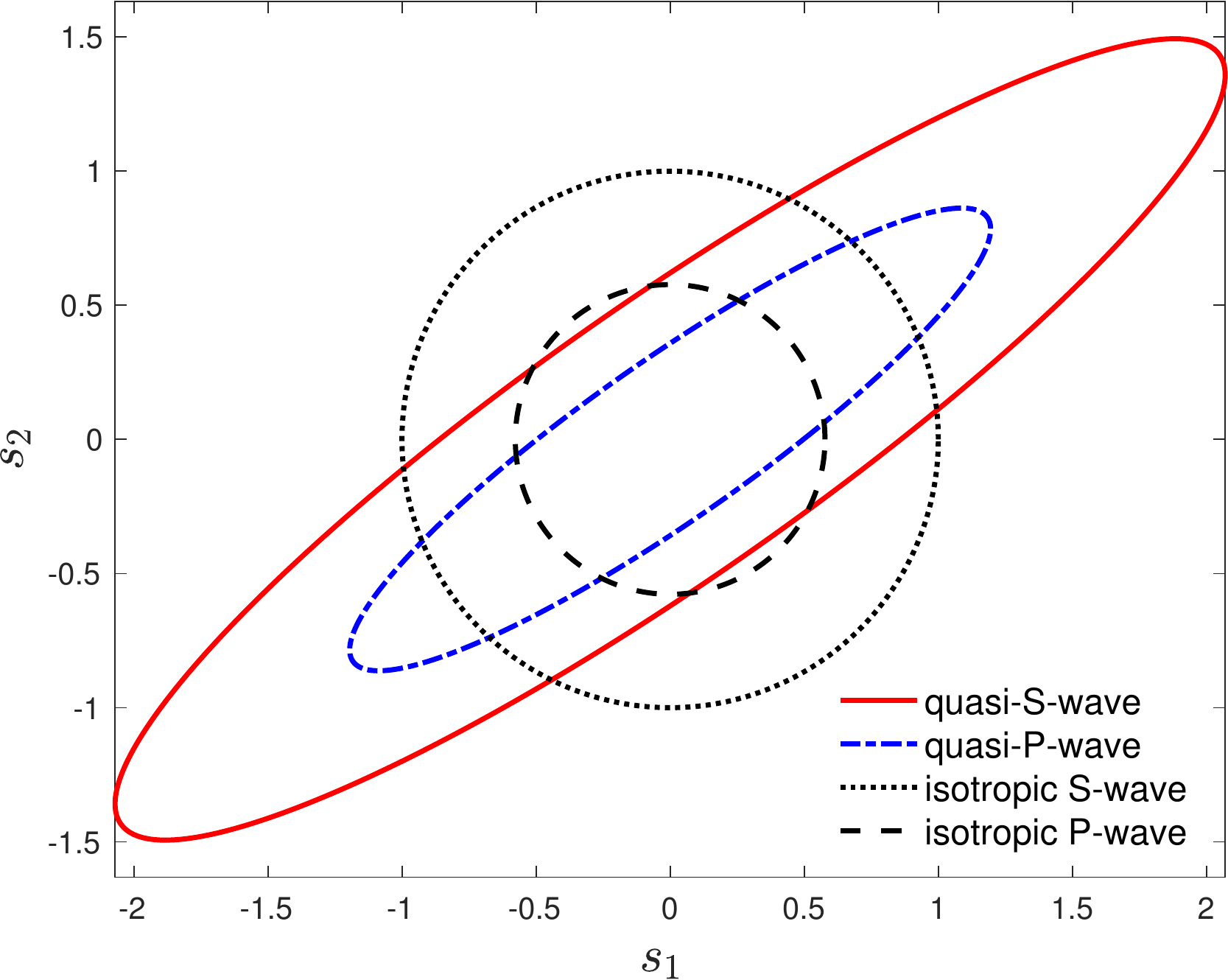}
        \caption{Slowness surfaces}\label{fig:slowness}
    \end{subfigure}
    \caption{(a) Quasi-P-wave speed relative to the background P-wave speed, for a wave propagating parallel to the $X_1$-axis, with colors corresponding to $\ln(c_{qp}/c_p)$. (b) Slowness surface at $[X_1,X_2] = [0.5, 1]$, with the slowness surfaces of the isotropic background medium included for reference. }
\end{figure}

To quantify the performance of the cloak, we probe the scatterer by applying a time-harmonic \otherchange{line} force outside of the cloak. In the presence of a time-harmonic line force, the 2D equations of motion read
\begin{equation} \label{eq:harmonic_continuous}
  \rho \ddot{u}_{\js} = \ddxi \stiffphys_{\is\js\ks\ls} \ddxk u_{\ls} + f_{\js} \delta(\vec{X}-\vec{X}_0) \cos{\alpha t},
\end{equation}
where $f_{\idx{J}}$ here is force per unit distance (not force per unit volume as in \eqref{eq:wave_eq_general}).
We use super-grid absorbing layers \cite{Appelo09_2,petersson_sjogreen_2014} to approximate \eqref{eq:harmonic_continuous} in an unbounded domain. The semidiscrete system of equations then reads
\begin{equation} \label{eq:harmonic_discrete}
  \rho \ddot{\bfu}_{\js} = \left( \cD^{\physdom}_{\is \ks}(\stiffphys_{\is\js\ks\ls}) + \mathbb{S}_{\js \ls} \right) \bfu_{\ls}  + \cE_{\js \ls} \bfut_{\ls} + f_{\js} \ddelta(\vec{X}-\vec{X}_0) \cos{\alpha t},
\end{equation}
where $\mathbb{S}_{\js \ls}$ denotes the SATs, $\ddelta$ is a discrete approximation of the $\delta$-function \cite{Petersson2016}, and $\cE_{\js \ls}$ provides dissipation in the super-grid layers. In the domain of interest, $\cE_{\js \ls}$ is zero. Inside the super-grid layers, $H \cE_{\js \ls}$ is symmetric negative semidefinite.

\reftwo{We choose to compute the time-harmonic solution to \eqref{eq:harmonic_discrete} (rather than solve the time-dependent equations) because it reveals the steady-state response of the system (instead of the response at arbitrarily selected times).} The time-harmonic solution can be written as
\begin{equation} \label{eq:harmonic_ansatz}
  \bfu_{\js} = \bfv_{\js} \cos{\alpha t} + \bfw_{\js} \sin{\alpha t} .
\end{equation}
Inserting the ansatz \eqref{eq:harmonic_ansatz} in \eqref{eq:harmonic_discrete} yields the system of equations
\begin{equation}
\begin{aligned}
  - \rho \alpha^2 \bfv_{\js} &= \left( \cD^{\physdom}_{\is \ks}(\stiffphys_{\is\js\ks\ls}) + \mathbb{S}_{\js \ls} \right) \bfv_{\ls}  + \alpha \cE_{\js \ls} \bfw_{\ls} + f_{\js} \ddelta(\vec{X}-\vec{X}_0), \\
  - \rho \alpha^2 \bfw_{\js} &= \left( \cD^{\physdom}_{\is \ks}(\stiffphys_{\is\js\ks\ls}) + \mathbb{S}_{\js \ls} \right) \bfw_{\ls}  - \alpha \cE_{\js \ls} \bfv_{\ls},
\end{aligned}
\end{equation}
which we solve for $\bfv_{\js}$ and $\bfw_{\js}$.

We choose force position $\vec{X}_0 = [1.5, 1.5]$, force vector $\vec{f} = [-\frac{1}{\sqrt{2}}, \frac{1}{\sqrt{2}}]$, angular frequency $\alpha = 2\pi$, and use the sixth order SBP-SAT method to discretize \eqref{eq:harmonic_continuous}. Figure \ref{fig:cloak_empty} shows the resulting displacement magnitude $\sqrt{\bfv_1 \circ \bfv_1 + \bfv_2 \circ \bfv_2}$, where $\circ$ denotes the Hadamard product, in free space, with no scatterer present (corresponding plots of $\vec{\bfw}$ are qualitatively similar and are omitted here). A perfect cloak would yield the same displacement outside of the cloak. Figure \ref{fig:cloak_orig} shows the displacement field in the presence of the uncloaked scatterer. There are obvious differences compared to the free-space solution---in particular the shadow zone to the southwest of the scatterer. Figure \ref{fig:cloak_cloaked} shows the displacement around the cloaked scatterer. Outside the cloak, the displacement is quite similar to the free-space solution, with minor differences---note in particular the faint shadow zone to the southwest of the scatterer. Outside the cloak, the displacement due to the cloaked scatterer is in fact identical (up to numerical errors) to the displacement produced by the small disk-shaped scatterer in Figure \ref{fig:cloak_ref}, with homogeneous material parameters. In this numerical experiment we could easily improve the performance of the cloak by making the disk in Figure \ref{fig:cloak_ref} even smaller, but that would make the coordinate transformation near-singular and would likely make the prescribed cloak material more difficult to engineer.
\begin{figure}[h]
\centering
  \begin{subfigure}[b]{.45\linewidth}
        \centering
        \includegraphics[width=0.85\linewidth, trim=1.8cm 0.1cm 2cm 0.5cm, clip=true]{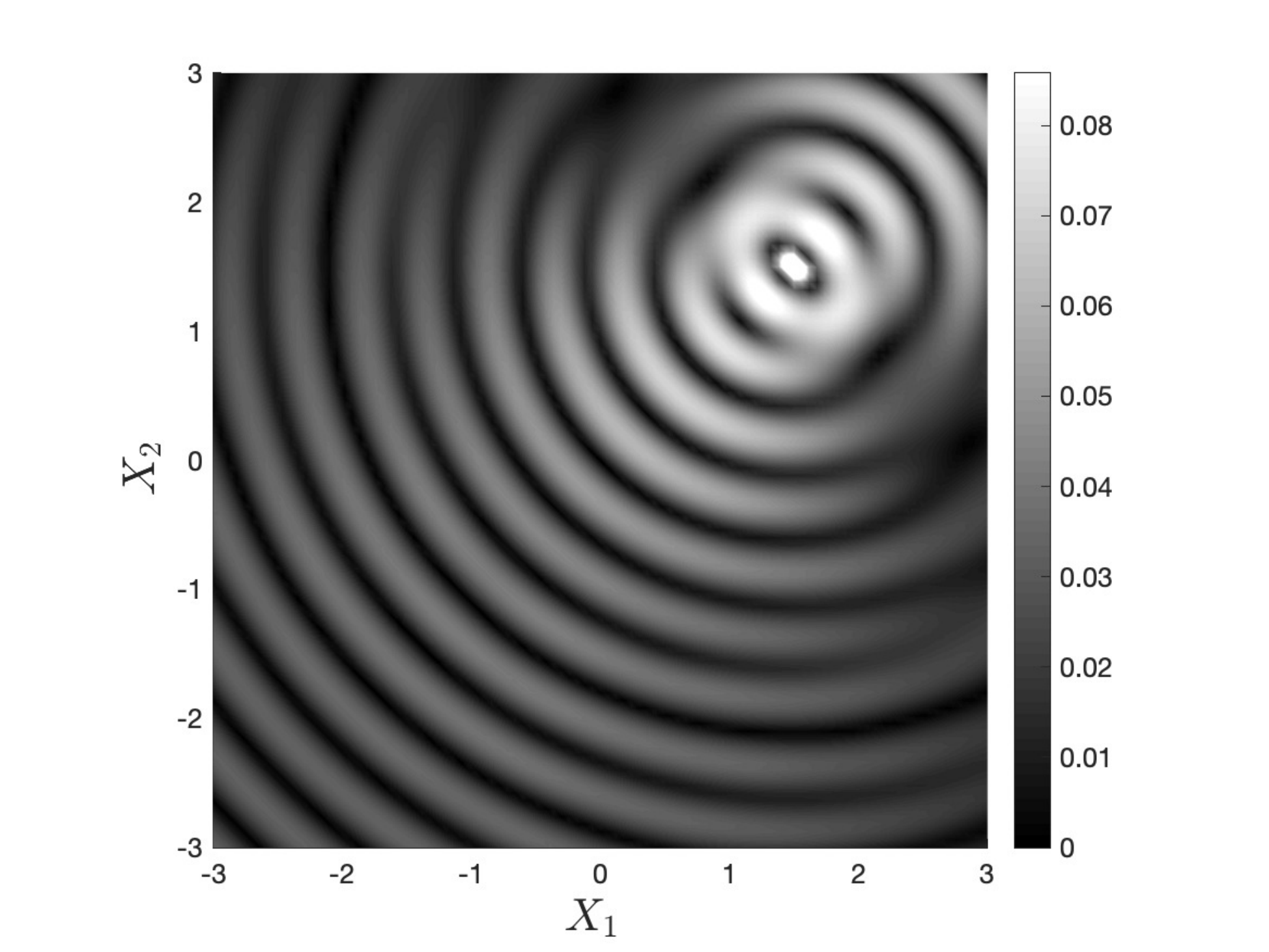}
        \caption{No scatterer}\label{fig:cloak_empty}
    \end{subfigure}
    \begin{subfigure}[b]{.45\linewidth}
        \centering
        \includegraphics[width= 0.85\linewidth, trim=1.8cm 0.1cm 2cm 0.5cm, clip=true]{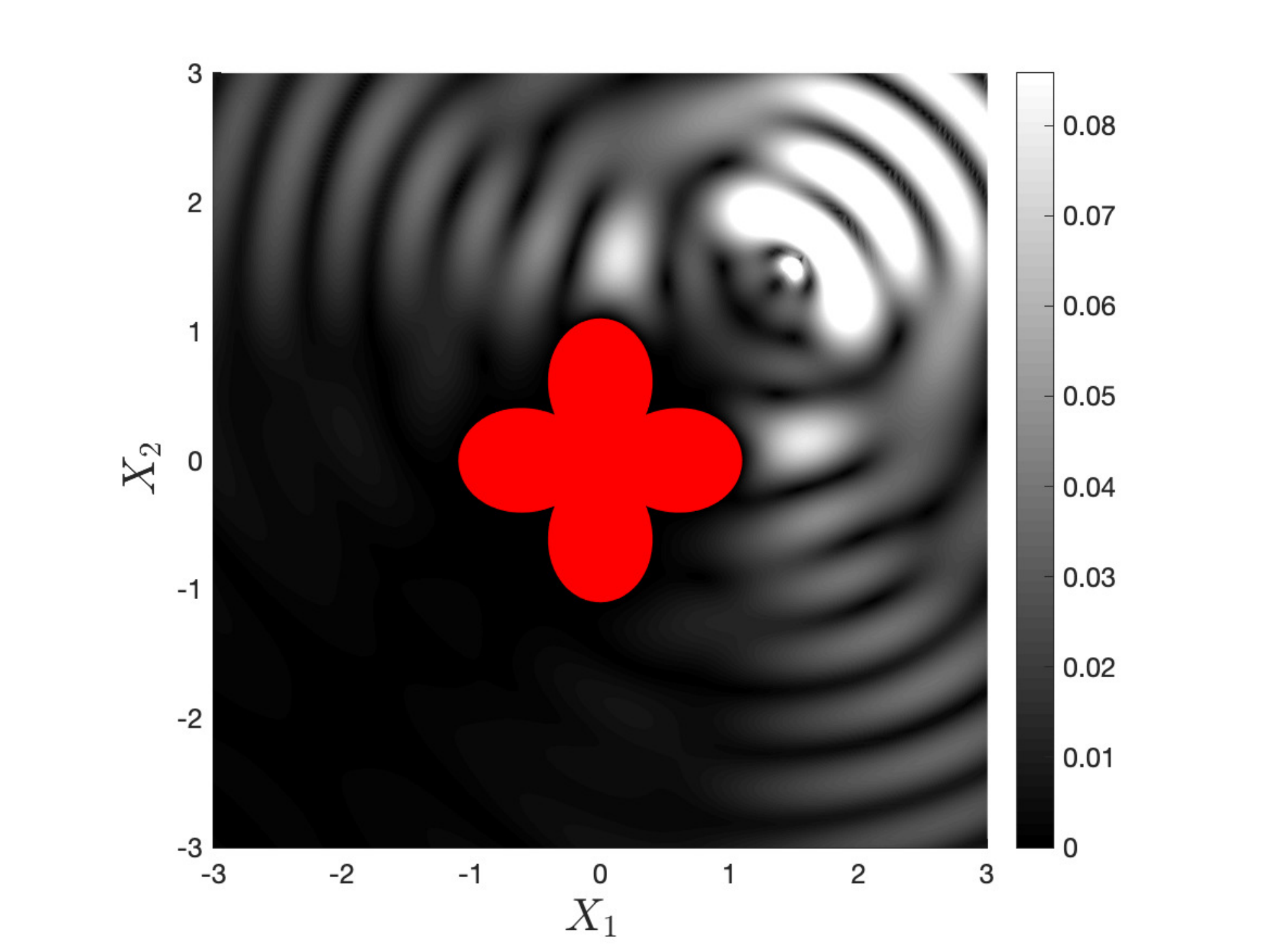}
        \caption{Scatterer without cloak}\label{fig:cloak_orig}
    \end{subfigure}
    \\
    \begin{subfigure}[b]{.45\linewidth}
        \centering
        \includegraphics[width= 0.85\linewidth, trim=1.8cm 0.1cm 2cm 0.5cm, clip=true]{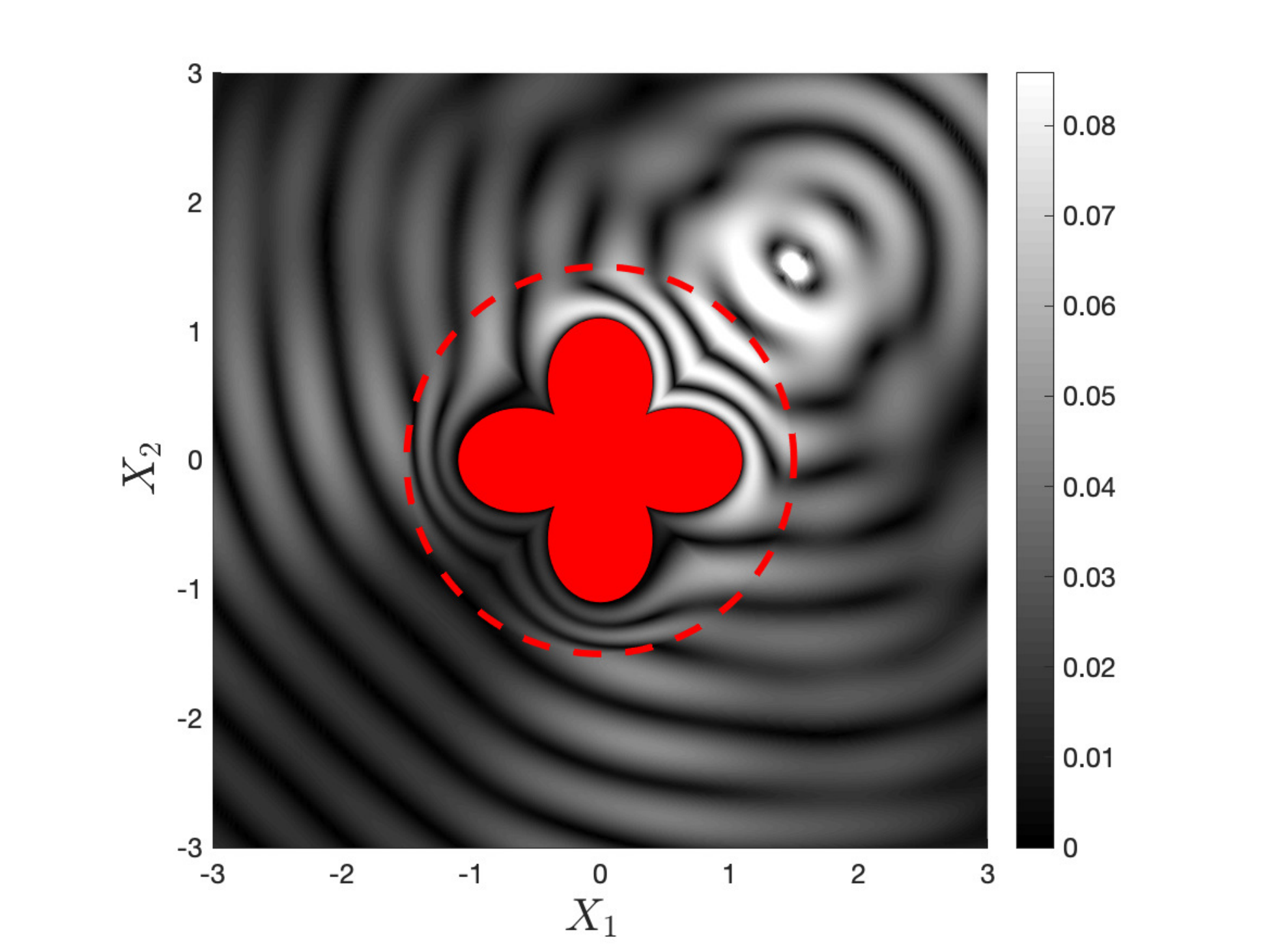}
        \caption{Cloaked scatterer}\label{fig:cloak_cloaked}
    \end{subfigure}
    \begin{subfigure}[b]{.45\linewidth}
        \centering
        \includegraphics[width=0.85\linewidth, trim=1.8cm 0.1cm 2cm 0.5cm, clip=true]{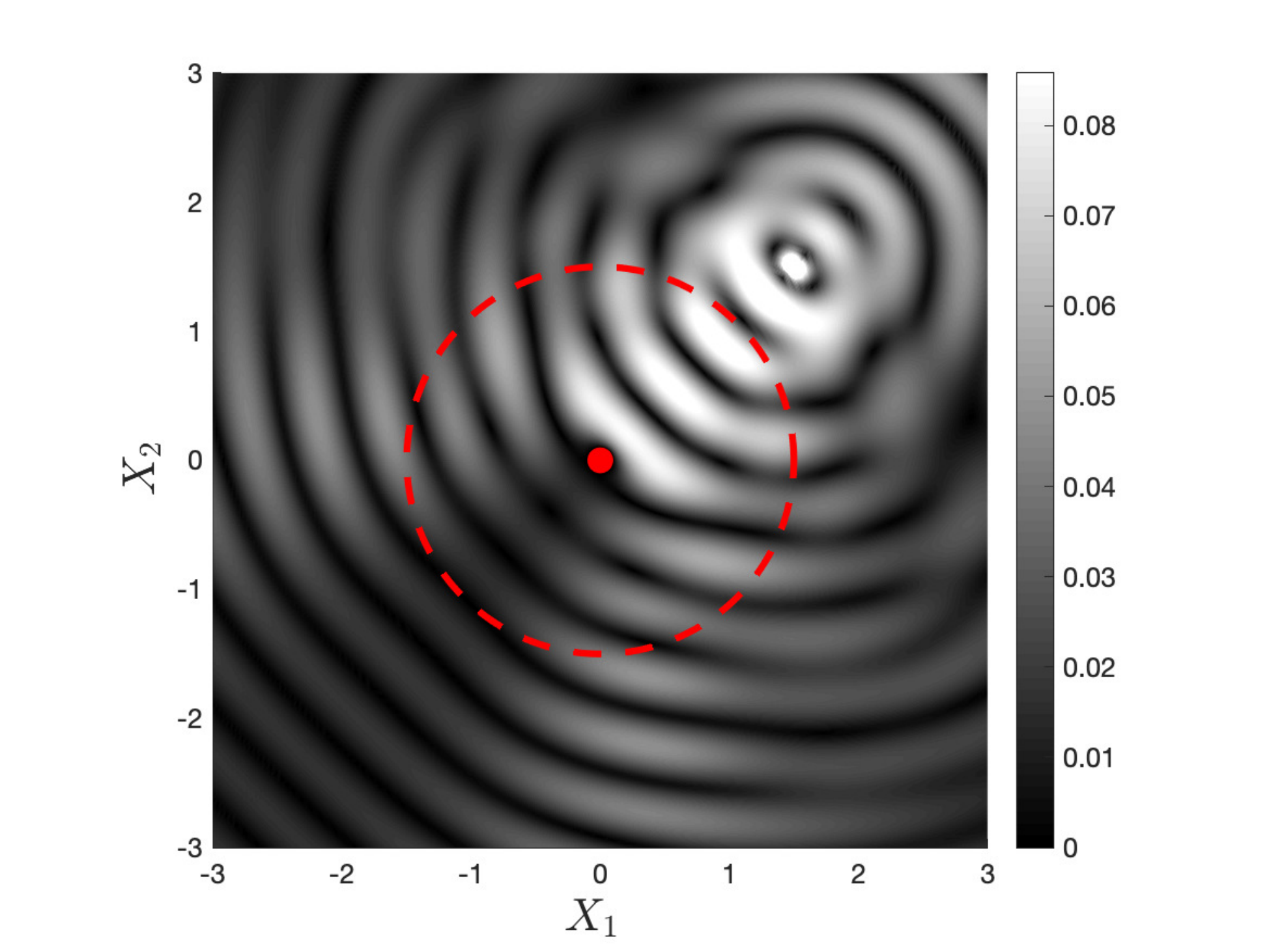}
        \caption{\reftwo{Fictitious} scatterer }\label{fig:cloak_ref}
    \end{subfigure}
    \caption{Plots of displacement magnitude $\sqrt{\bfv_1 \circ \bfv_1 + \bfv_2 \circ \bfv_2}$ caused by a time-harmonic point force applied at $\vec{X}=[1.5, 1.5]$ with (a) no scatterer;  (b) an uncloaked scatterer; (c) a cloaked scatterer; and (d) the small \reftwo{fictitious} scatterer that is equivalent to the cloaked scatterer. }
    \label{fig:cloak}
\end{figure}

\subsection{Seismic imaging in mountainous regions}
The topic of the second application problem is seismic imaging on land, in particular in mountainous regions where topographical variations may be large. Other studies that have developed finite difference methods on curvilinear grids for use in seismic imaging in the presence of topography include \cite{Shragge2016,Shragge2017}. As a structural model representative of mountainous regions we choose the SEG SEAM Foothills model \cite{Oristaglio2016}, which is an isotropic model with heterogeneous material properties and very pronounced topography. We select a vertical cross section of the original 3D structural model with pressure and shear wave speeds as shown in Figures \ref{fig:vp} and \ref{fig:vs}.
\begin{figure}[h]
  \begin{subfigure}[b]{.49\linewidth}
    \centering
    \includegraphics[width=1\linewidth, trim=1cm 1cm 1cm 1cm, clip=true]{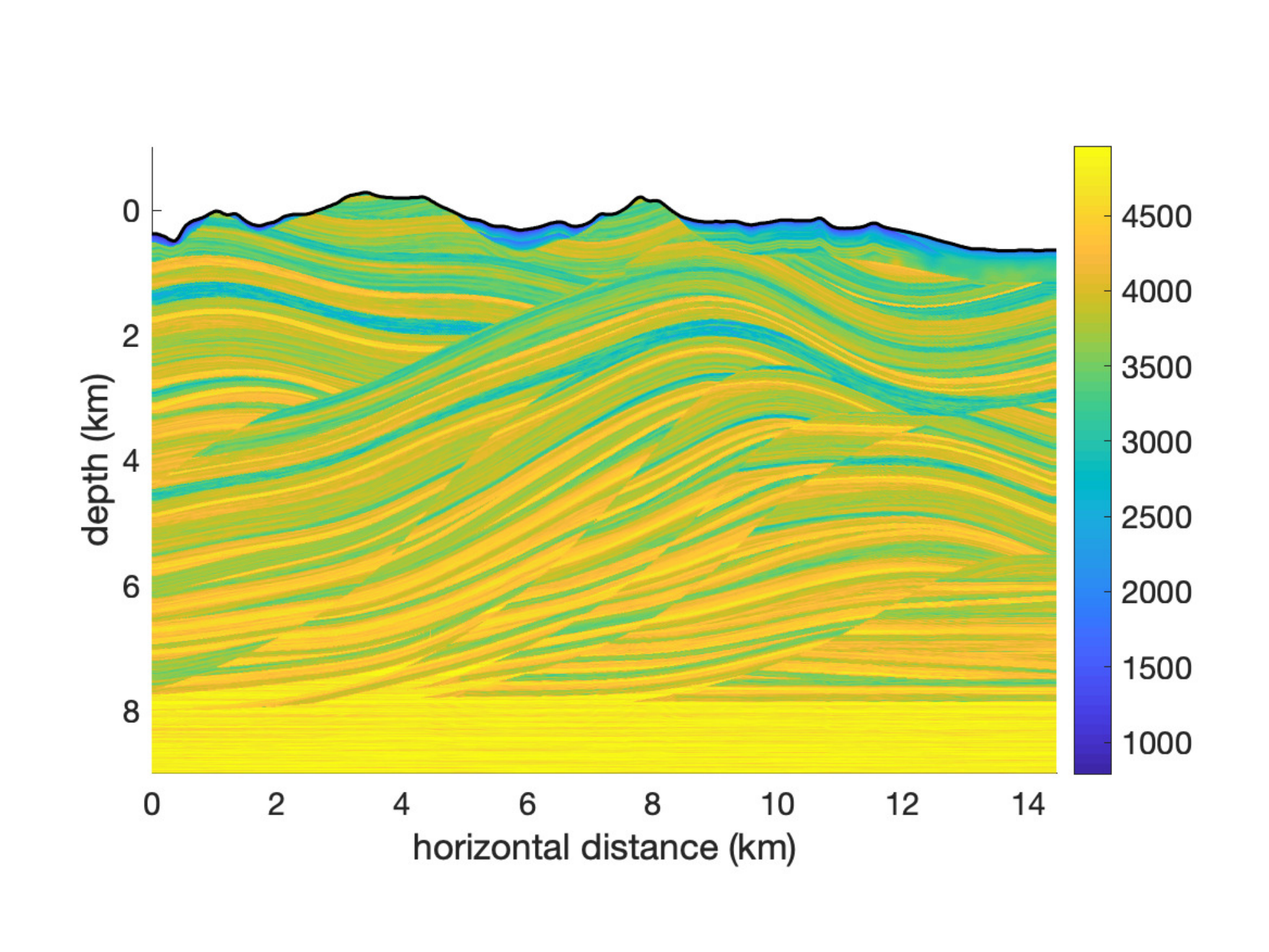}
    \caption{Pressure wave speed (m/s)}\label{fig:vp}
\end{subfigure}
\begin{subfigure}[b]{.49\linewidth}
    \centering
    \includegraphics[width=1\linewidth, trim=1cm 1cm 1cm 1cm, clip=true]{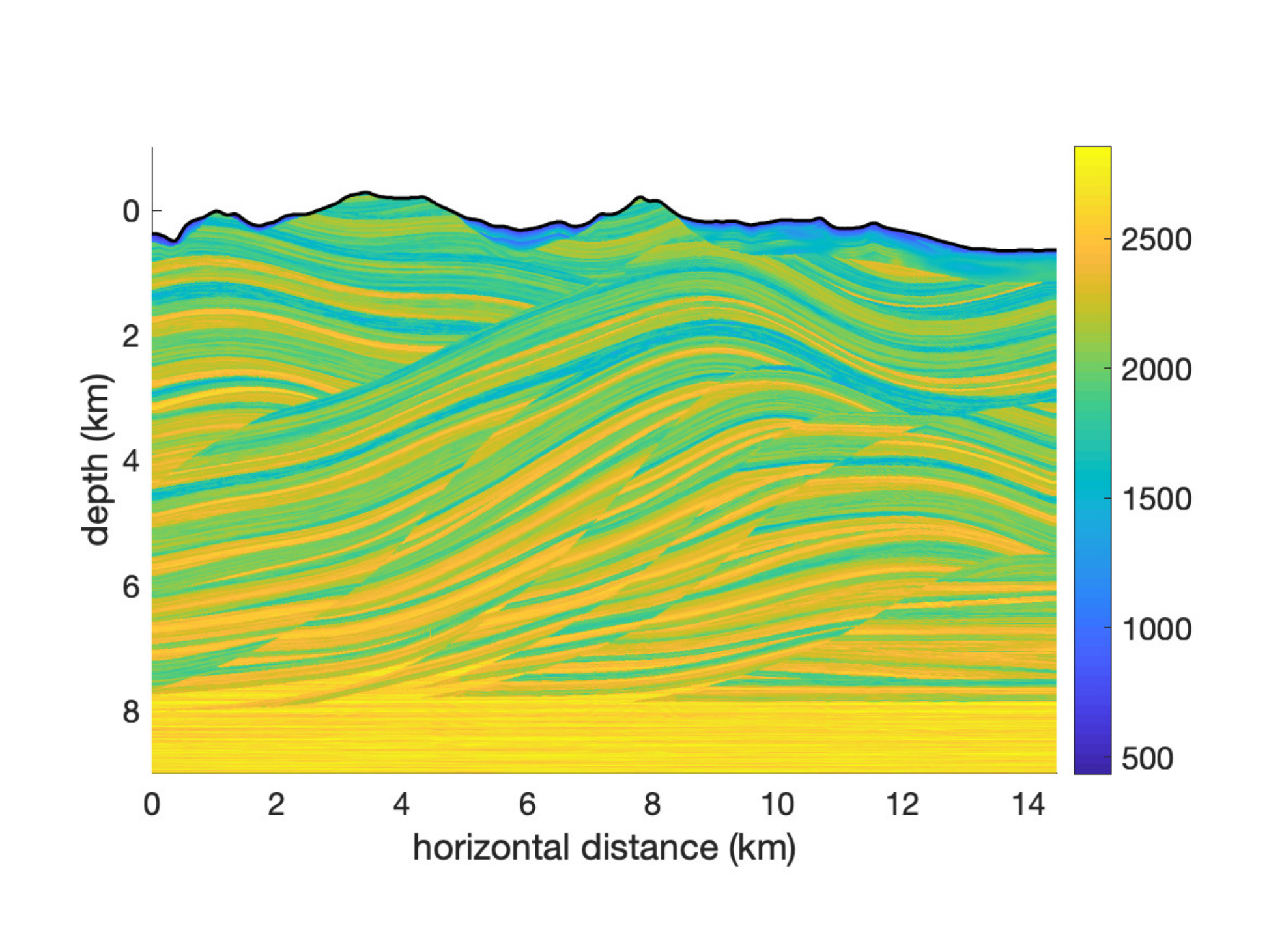}
    \caption{Shear wave speed (m/s)}\label{fig:vs}
\end{subfigure}
\caption{\otherchange{Wave speeds in a vertical cross section of the SEG SEAM Foothills model}}
\end{figure}
To mimic a vibrator source, we impose homogeneous traction boundary conditions on the free surface and apply a vertical point force at the surface (alternatively, one could impose inhomogeneous traction boundary conditions, which, for a particular choice of the discrete delta function, yields an identical semi-discrete problem). \otherchange{Note that with the wide-stencil method the discrete $\delta$-function must satisfy appropriate \emph{smoothness conditions} \cite{Petersson2016}, which we have incorporated}. The force vector is (note that we use the symbol $\delta$ to denote both the Kronecker delta and the Dirac delta function)
\begin{equation}
   f_\js =  -\delta_{\js 2} \hat{f} W(t) \delta(\vec{X} - \vec{X}_0),
\end{equation}
where $\hat{f}$ is a scalar force amplitude and $W(t)$ denotes the Ricker wavelet \cite{Ricker1943, Ricker1944} with peak frequency $\alpha_P$ centered at time $t_0$, i.e.,
\begin{equation}
    W(t) = (1-2 \pi^2 \otherchange{\alpha_{P}^2} (t-t_0)^2 ) e^{-\pi^2 \otherchange{\alpha_{P}^2} (t-t_0)^2}.
\end{equation}
\refthree{To further characterize the source we define the maximum source frequency $\alpha_{M}>\alpha_{P}$ as the frequency for which the amplitude spectrum is 5\% of peak amplitude, i.e.,
 \begin{equation}
   \abs{\mathcal{F}[W](\alpha_{M})} = 0.05\abs{\mathcal{F}[W](\alpha_{P})},
 \end{equation}}
where $\mathcal{F}[W]$ denotes the Fourier transform of $W$. This definition yields $\alpha_M \approx 2.40 \alpha_P$. We think of $\alpha_M$ as the highest frequency that needs to be resolved for accurate simulation results.
We choose \otherchange{$\alpha_{P} = 4$ Hz, which yields $\alpha_M = 9.59$ Hz. We further set} $t_0 = \otherchange{\alpha_P^{-1}}$ and let the horizontal position of the point force be $X_1 = 6$ km. We select $\hat{\rho} = 1340$ kg/m$^{3}$ and $\hat{c}_s = 600$ m/s as reference values for density and shear wave speed near the source and define nondimensional particle velocity $\dot{\tilde{u}}_\is$ as
\begin{equation}
\dot{\tilde{u}}_\is = \frac{\hat{\rho} \hat{c}_s^2}{\hat{f} \otherchange{\alpha_{P}}} \dot{u}_\is.
\end{equation}

Our implementation utilizes the PETSc \cite{petsc-efficient,petsc-user-ref,petsc-web-page} implementation of the classical fourth order Runge--Kutta method in the TS ODE/DAE solver library \cite{abhyankar2018petsc2}. \otherchange{We set $\mbox{CFL}=0.4$ and} use the sixth order SBP-SAT method with grid spacing $\approx 7$ m (in the physical domain $\Omega$), \refthree{which corresponds to 7.2 points per wavelength ($\mbox{PPWL}$). We compute PPWL based on the maximum frequency $\alpha_M$ and the minimum shear velocity $c_s^{min}$ (here equal to 500 m/s) according to
\begin{equation}
  \mbox{PPWL} = \frac{c_s^{min}}{\alpha_M \Delta X_1},
\end{equation}
where $\Delta X_1$ denotes the horizontal grid spacing.} The grid is generated by transfinite interpolation with uniform spacing in the horizontal direction. We again use super-grid absorbing layers at the artificial boundaries.  We use only one grid block to discretize the domain shown in Figure \ref{fig:vp} and hence differentiate across the discontinuities in material parameters associated with the many media layers. While this constitutes a first order error, we remark that the method remains energy stable.

The top three rows of Figure \ref{fig:snapshot} show snapshots of particle velocity in the vertical direction. The bottom panel shows a space-time plot (shot gather) of vertical particle velocity recorded at the surface. \refthree{Figure \ref{fig:seismogramNarrow} shows seismograms, recorded at the surface at horizontal position $X_1=10$ km. With 1.8 PPWL, the computations are under-resolved. The narrow-stencil simulations with 3.6 PPWL and 7.2 PPWL show good agreement,  indicating that with 7.2 PPWL the numerical errors are small. This is further corroborated by the fact that the wide- and narrow-stencil seismograms with 7.2 PPWL are practically indistinguishable. To assess the performance of the wide- and narrow-stencil methods in marginally resolved simulations, Figure \ref{fig:seismogramWideNarrow} shows seismograms generated with 3.6 PPWL along with a 7.2 PPWL reference seismogram. The wide- and narrow-stencil methods produce slightly different seismograms.
}

To assess the influence of the structural model, \otherchange{we repeat the experiments above} with constant material parameters $\rho = 2300$ kg/m$^3$, pressure wave speed $c_p = 3500$ m/s, and shear wave speed $c_s = 2000$ m/s (note that PPWL values for this example are based on this value of $c_s$). Figure \ref{fig:snapshot_const} shows snapshots of vertical particle velocity. Dashed vertical lines in the bottom panel relate scattering of waves to topographical features. \refthree{Figure \ref{fig:seismogramNarrowConst} shows seismograms, recorded at the surface at horizontal position $X_1=10$ km. The 7.2 PPWL simulation shows excellent agreement with the 28.9 PPWL simulations. Figure \ref{fig:seismogramWideNarrowConst} compares the wide- and narrow-stencil seismograms generated with 2.9 PPWL. In this case, the narrow-stencil method is a clear winner; the wide-stencil method significantly underpredicts the amplitude of the largest peak and produces a tail of waves of much larger amplitude than in the reference solution.
}
\begin{figure}[h]
  \begin{subfigure}[b]{.49\linewidth}
        \centering
        \includegraphics[width=0.8\linewidth, trim=2.5cm 0.8cm 2.0cm 0.5cm, clip=true]{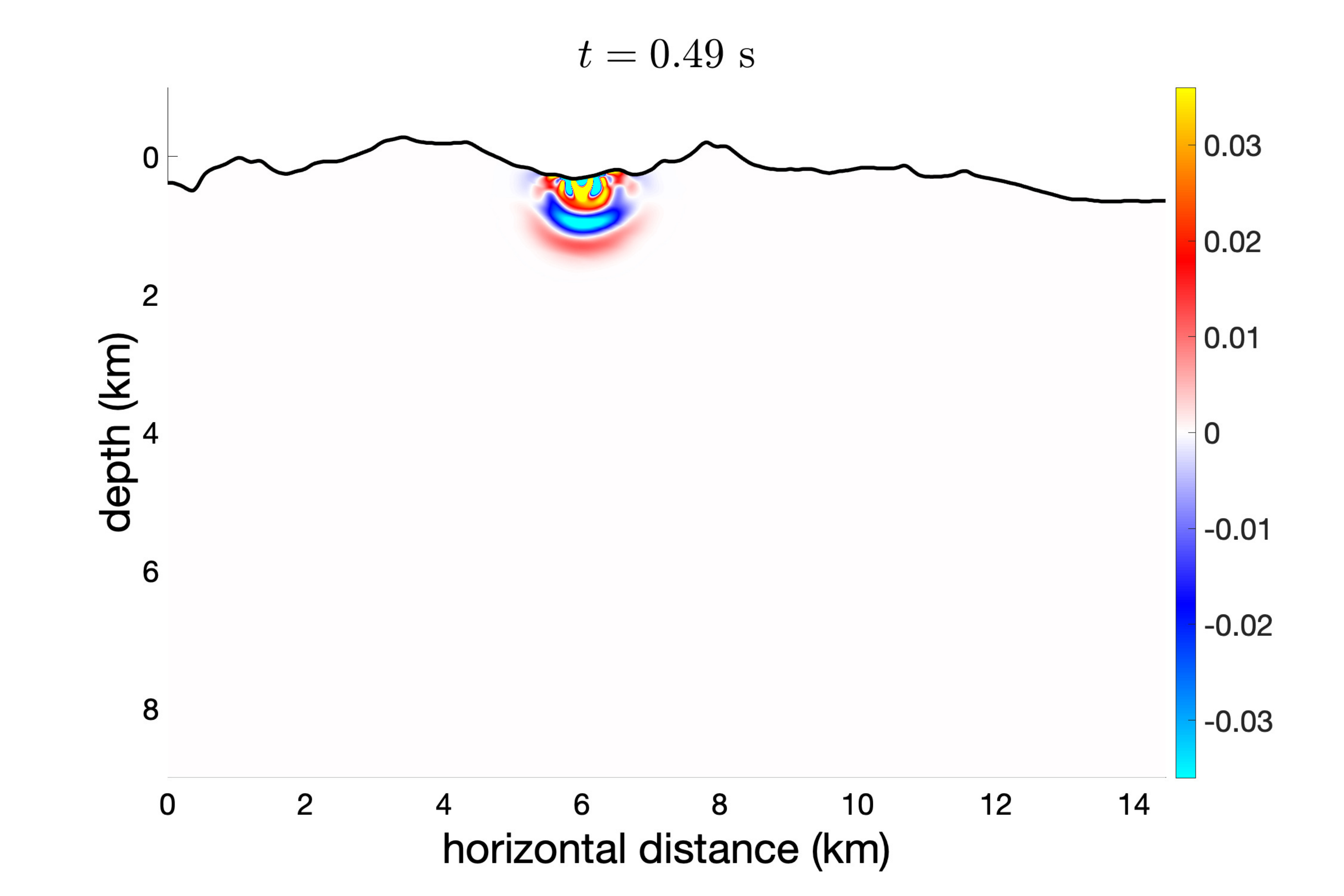}
        \label{fig:fh1}
    \end{subfigure}
    \begin{subfigure}[b]{.49\linewidth}
        \centering
        \includegraphics[width=0.8\linewidth, trim=2.5cm 0.8cm 2.0cm 0.5cm, clip=true]{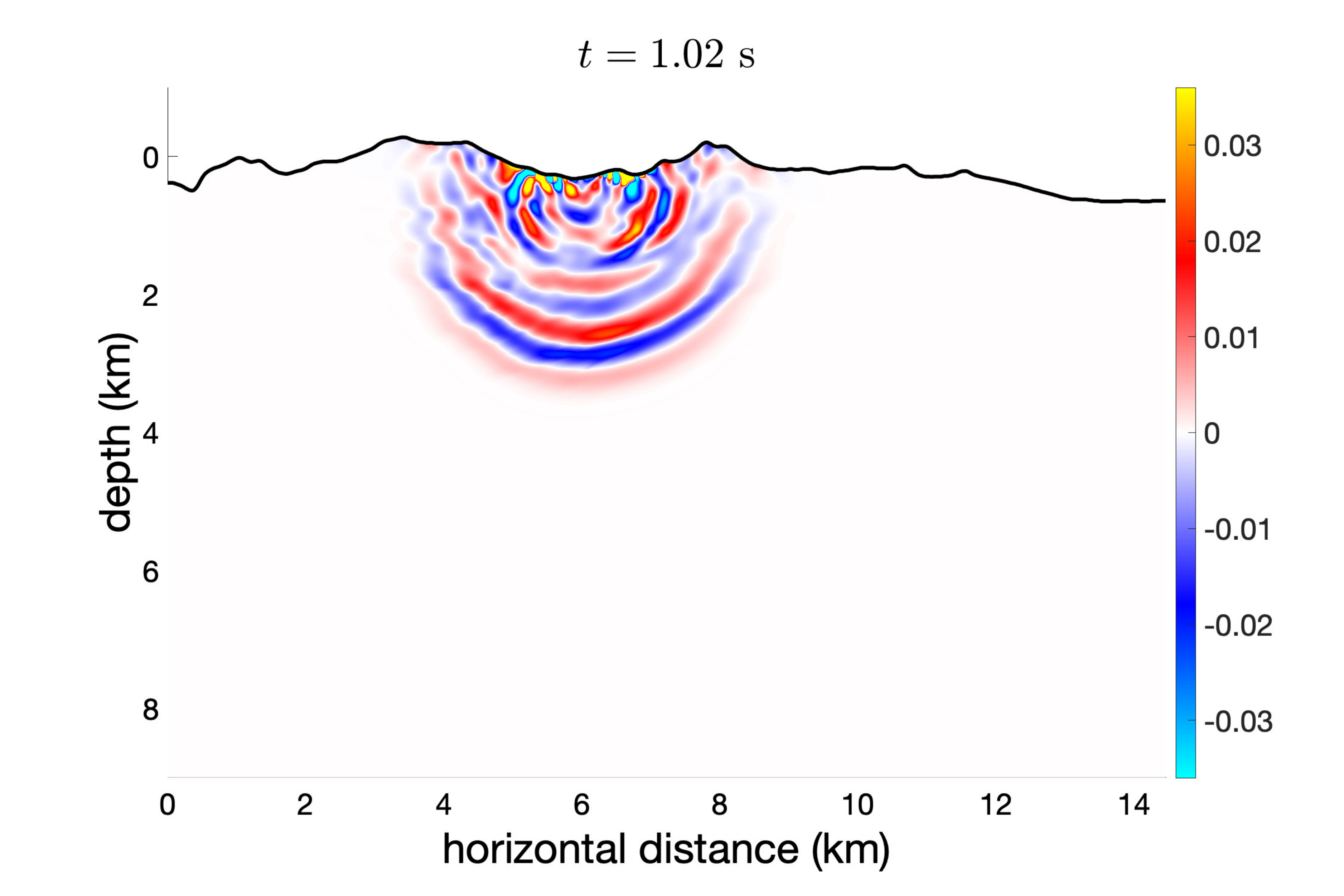}
        \label{fig:fh2}
    \end{subfigure}
    \\
    \begin{subfigure}[b]{.49\linewidth}
        \centering
        \includegraphics[width=0.8\linewidth, trim=2.5cm 0.8cm 2.0cm 0.5cm, clip=true]{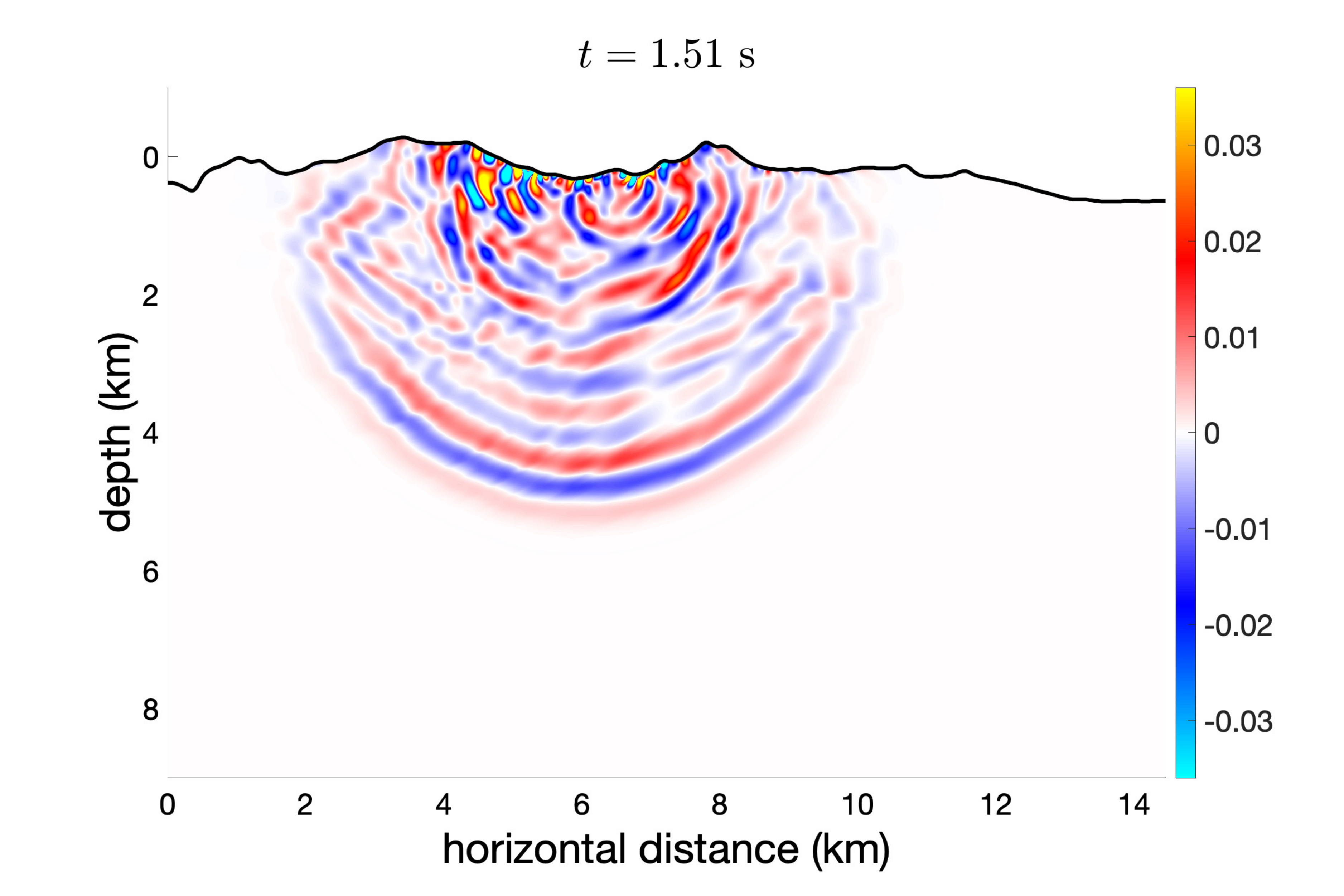}
        \label{fig:fh3}
    \end{subfigure}
    \begin{subfigure}[b]{.49\linewidth}
        \centering
        \includegraphics[width=0.8\linewidth, trim=2.5cm 0.8cm 2.0cm 0.5cm, clip=true]{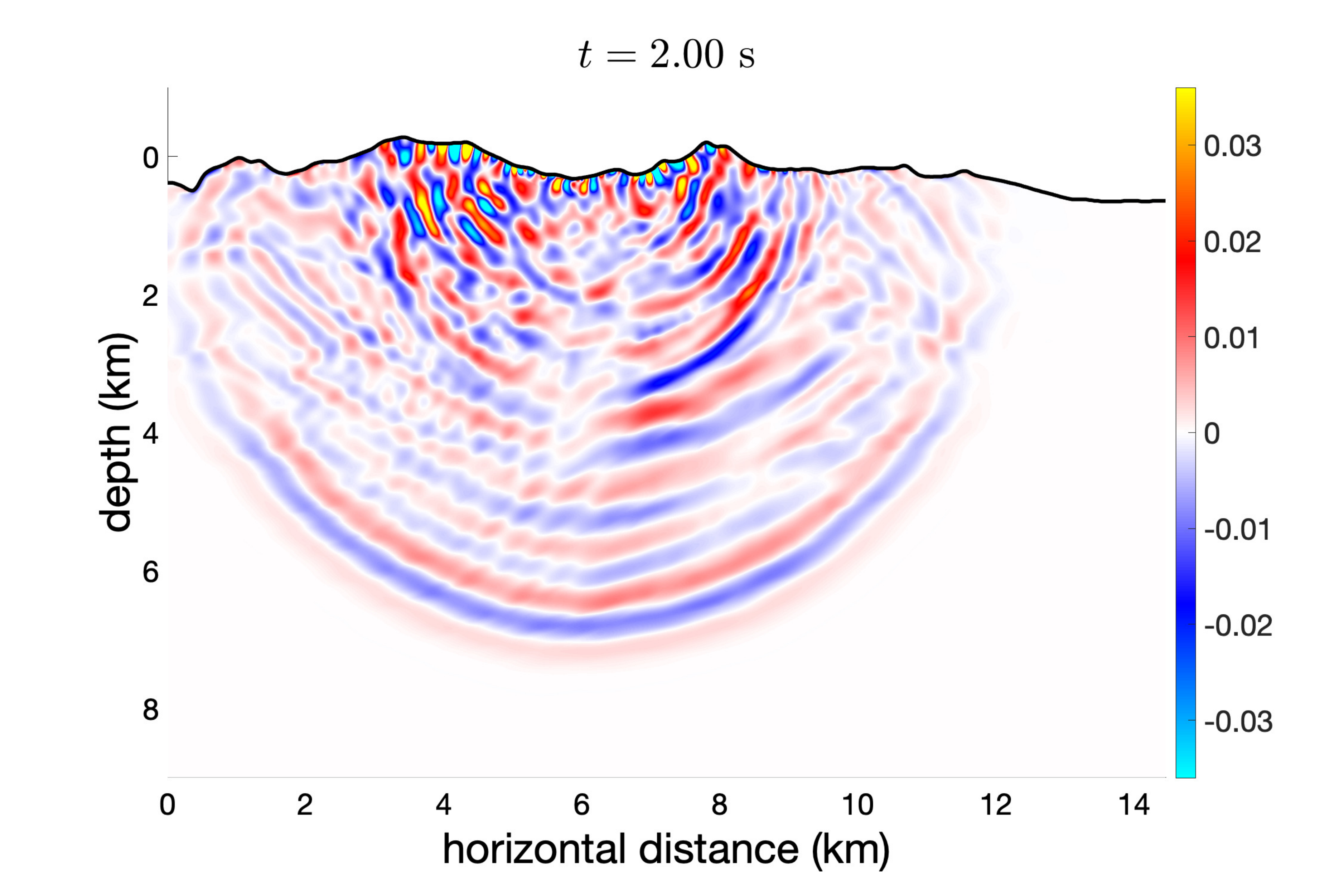}
        \label{fig:fh4}
    \end{subfigure}
    \\
    \begin{subfigure}[b]{.49\linewidth}
        \centering
        \includegraphics[width=0.8\linewidth, trim=2.5cm 0.8cm 2.0cm 0.5cm, clip=true]{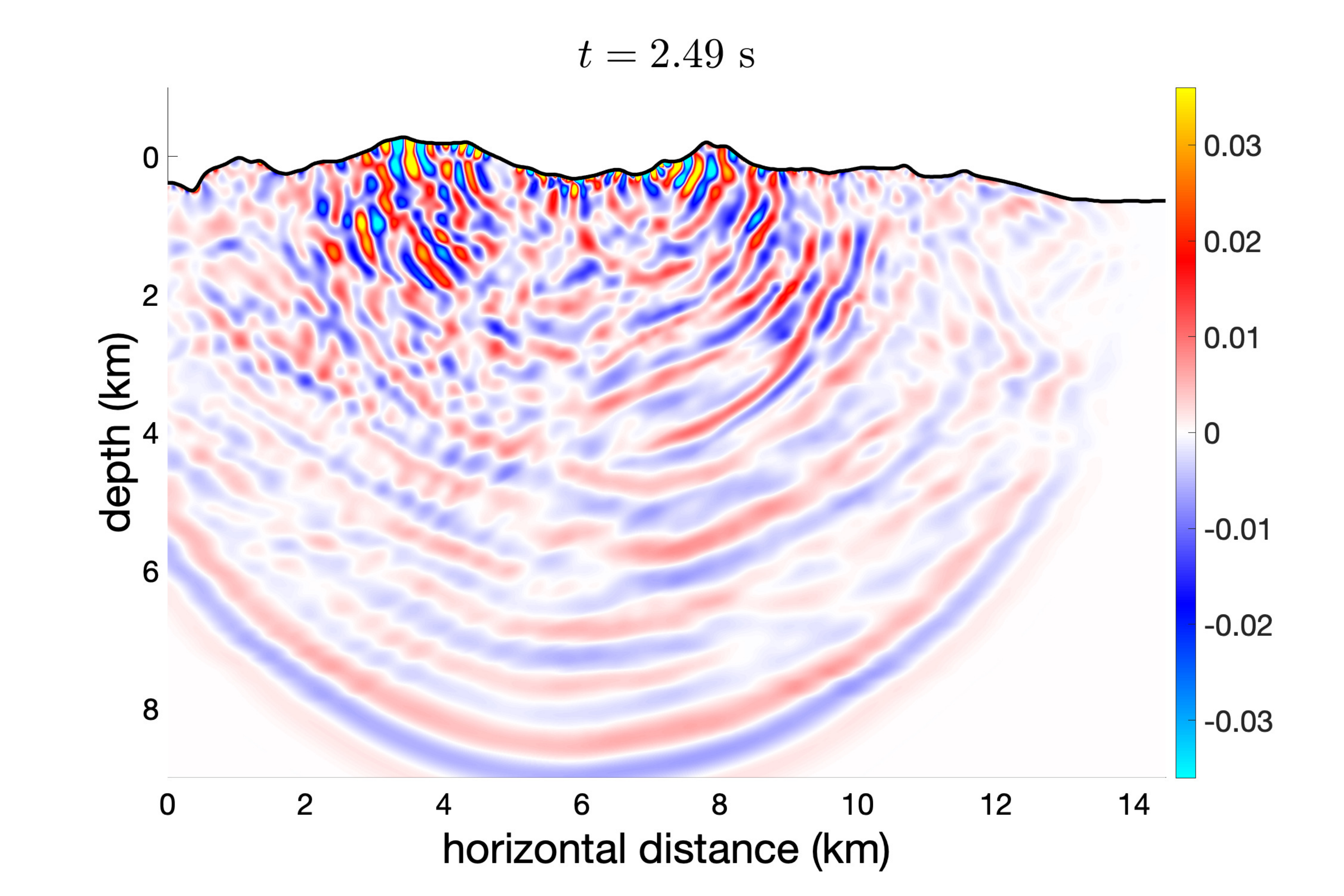}
        \label{fig:fh5}
    \end{subfigure}
    \begin{subfigure}[b]{.49\linewidth}
        \centering
        \includegraphics[width=0.8\linewidth, trim=2.5cm 0.8cm 2.0cm 0.5cm, clip=true]{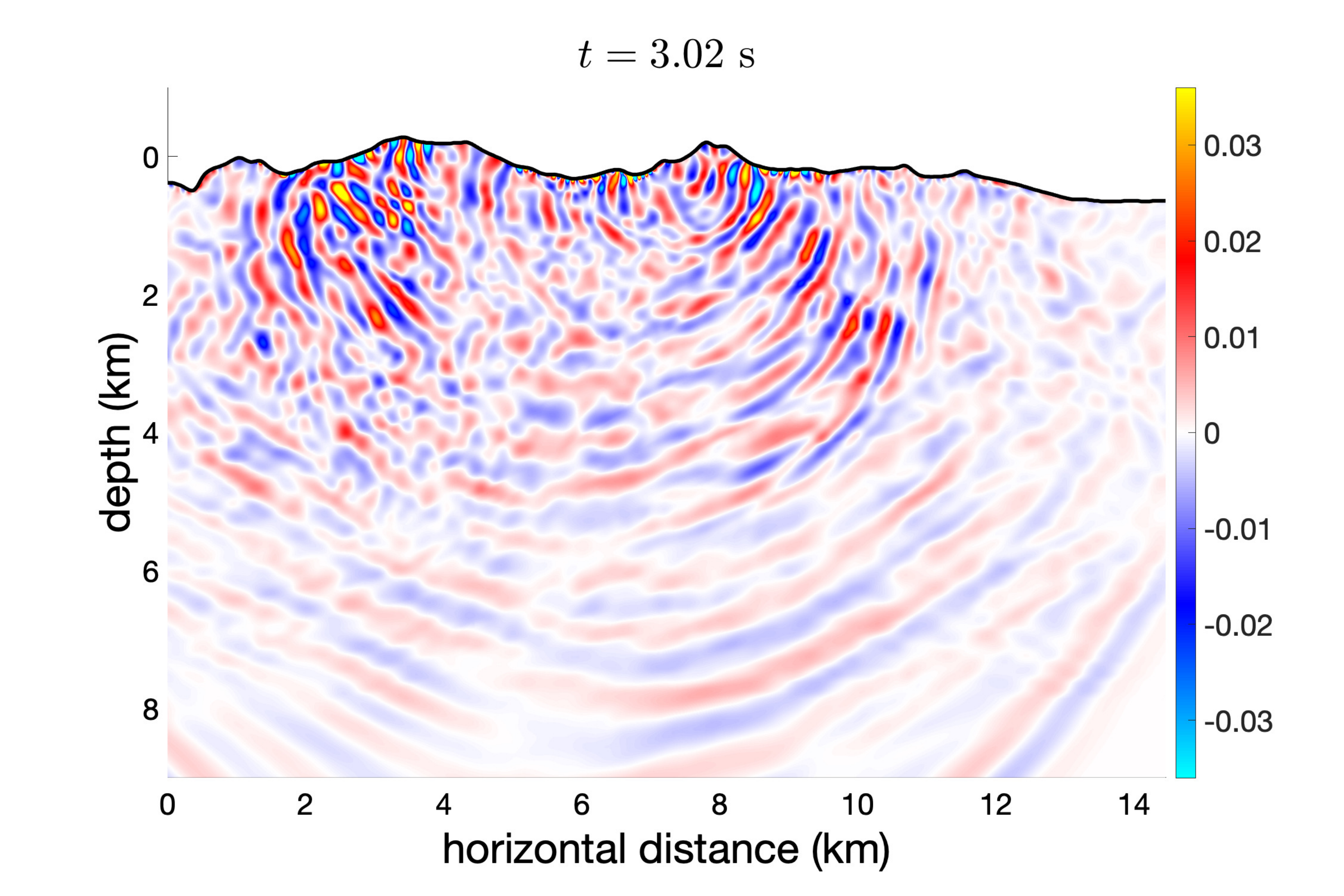}
        \label{fig:fh6}
    \end{subfigure}
    \\
    \centering
    \begin{subfigure}[b]{0.8\linewidth}
        \centering
        \includegraphics[width=1\linewidth, trim=0cm 0.1cm 0cm 0cm, clip=true]{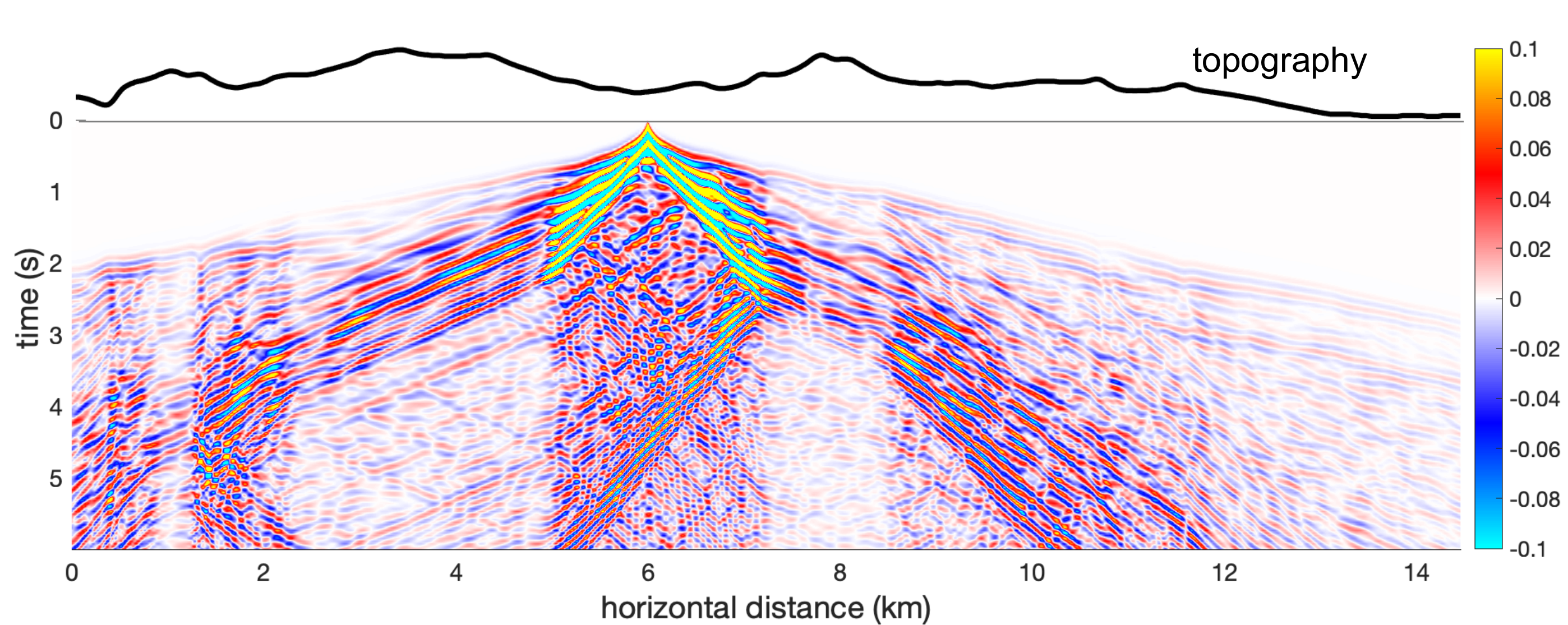}
        \label{fig:spaceTime_f10}
    \end{subfigure}
    \caption{Plots of $\dot{\tilde{u}}_2$, the vertical component of particle velocity, with the Foothills structural model. The top three rows show snapshots of $\dot{\tilde{u}}_2$ at different times. The bottom panel shows a space-time plot (shot gather) of $\dot{\tilde{u}}_2$, recorded at the surface.}
    \label{fig:snapshot}
\end{figure}
\begin{figure}[h]
  \begin{subfigure}[b]{.49\linewidth}
    \centering
    \includegraphics[width=1\linewidth, trim=0cm 0cm 0cm 0cm, clip=true]{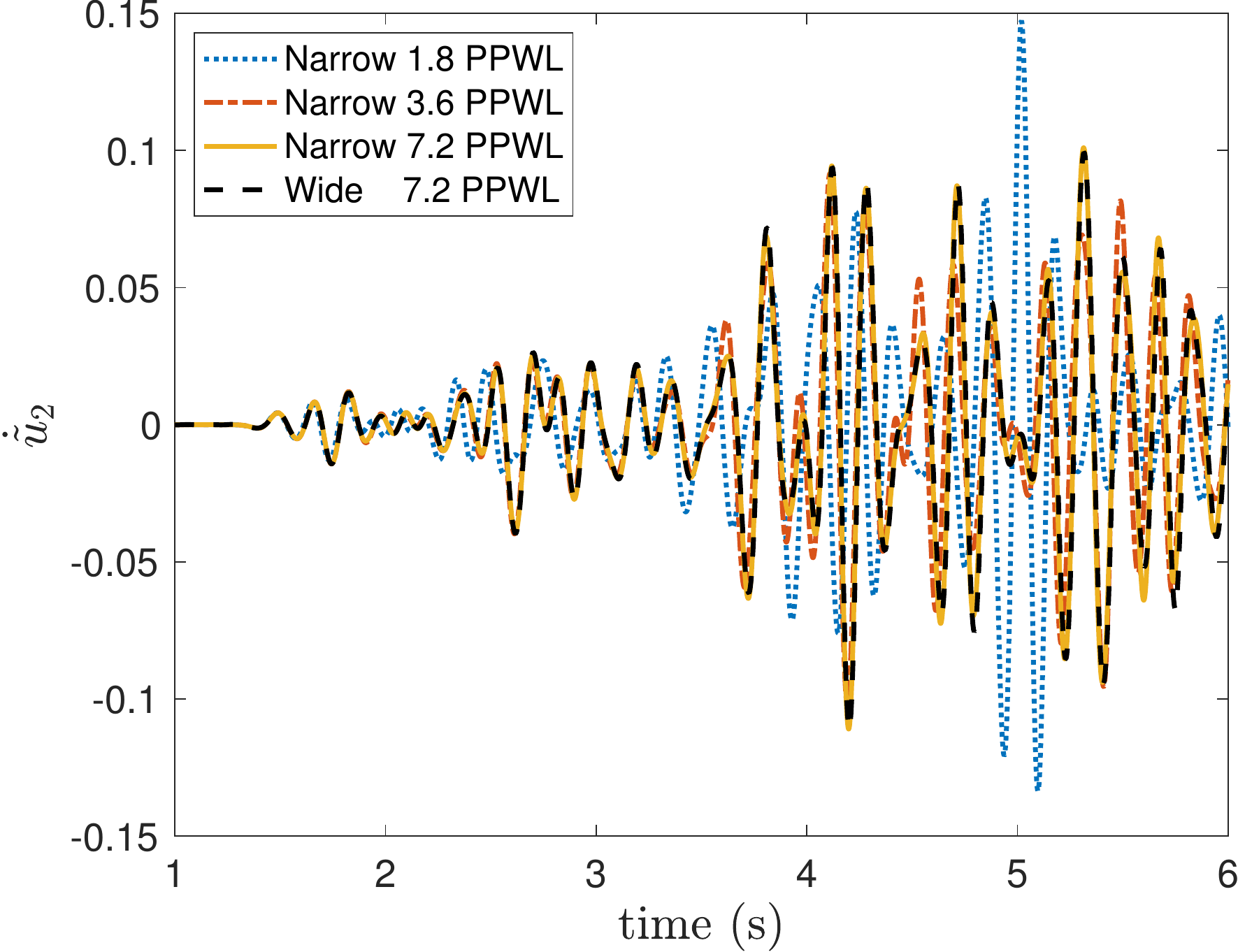}
    \caption{\refthree{Self-refinement (narrow scheme)}}\label{fig:seismogramNarrow}
\end{subfigure}
\begin{subfigure}[b]{.49\linewidth}
    \centering
    \includegraphics[width=1\linewidth, trim=0cm 0cm 0cm 0cm, clip=true]{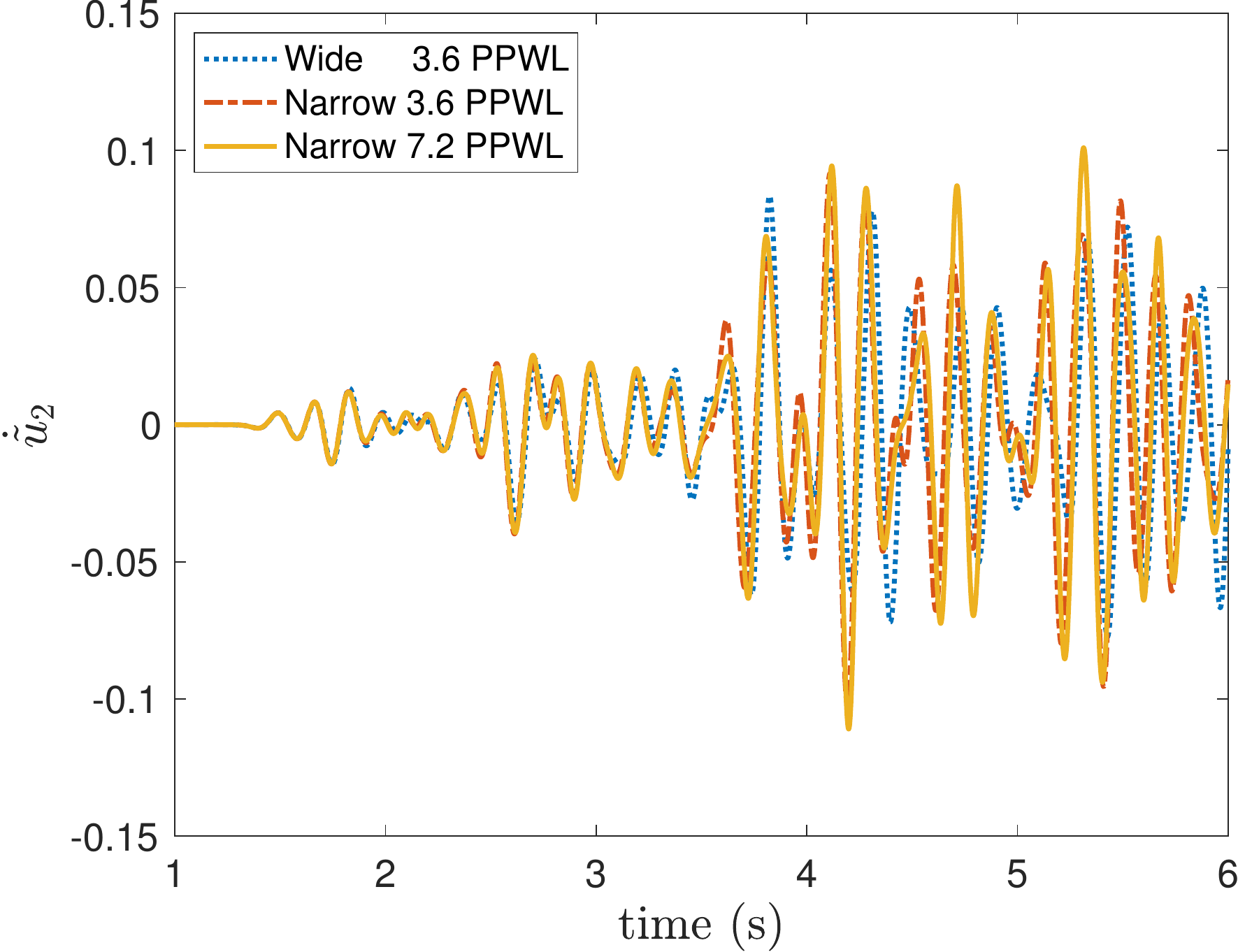}
    \caption{\refthree{Comparing wide and narrow schemes}}\label{fig:seismogramWideNarrow}
\end{subfigure}
\caption{\refthree{Seismograms of vertical particle velocity $\dot{\tilde{u}}_2$, recorded at the surface at $X_1= 10$ km, with the Foothills structural model. PPWL denotes points per wavelength, estimated using the minimum shear wave speed and maximum source frequency. (a) Seismograms generated by the narrow-stencil method at different levels of grid-refinement. (b) Seismograms generated by the wide- and narrow-stencil methods on a coarse grid, compared to a reference solution on a fine grid.}}
\end{figure}
\begin{figure}[h]
  \begin{subfigure}[b]{.49\linewidth}
    \centering
    \includegraphics[width=1\linewidth, trim=0cm 0cm 0cm 0cm, clip=true]{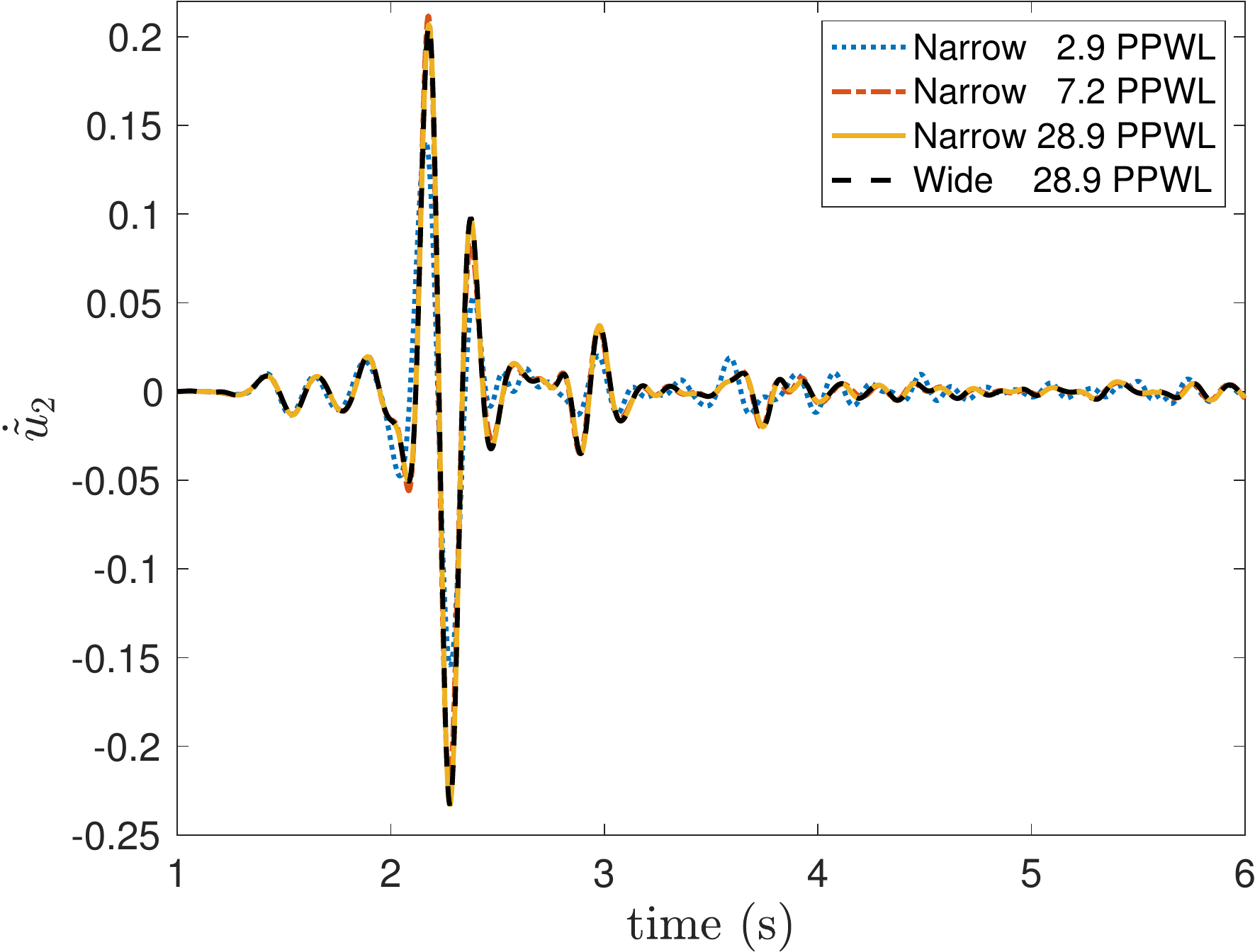}
    \caption{\refthree{Self-refinement (narrow scheme)}}\label{fig:seismogramNarrowConst}
\end{subfigure}
\begin{subfigure}[b]{.49\linewidth}
    \centering
    \includegraphics[width=1\linewidth, trim=0cm 0cm 0cm 0cm, clip=true]{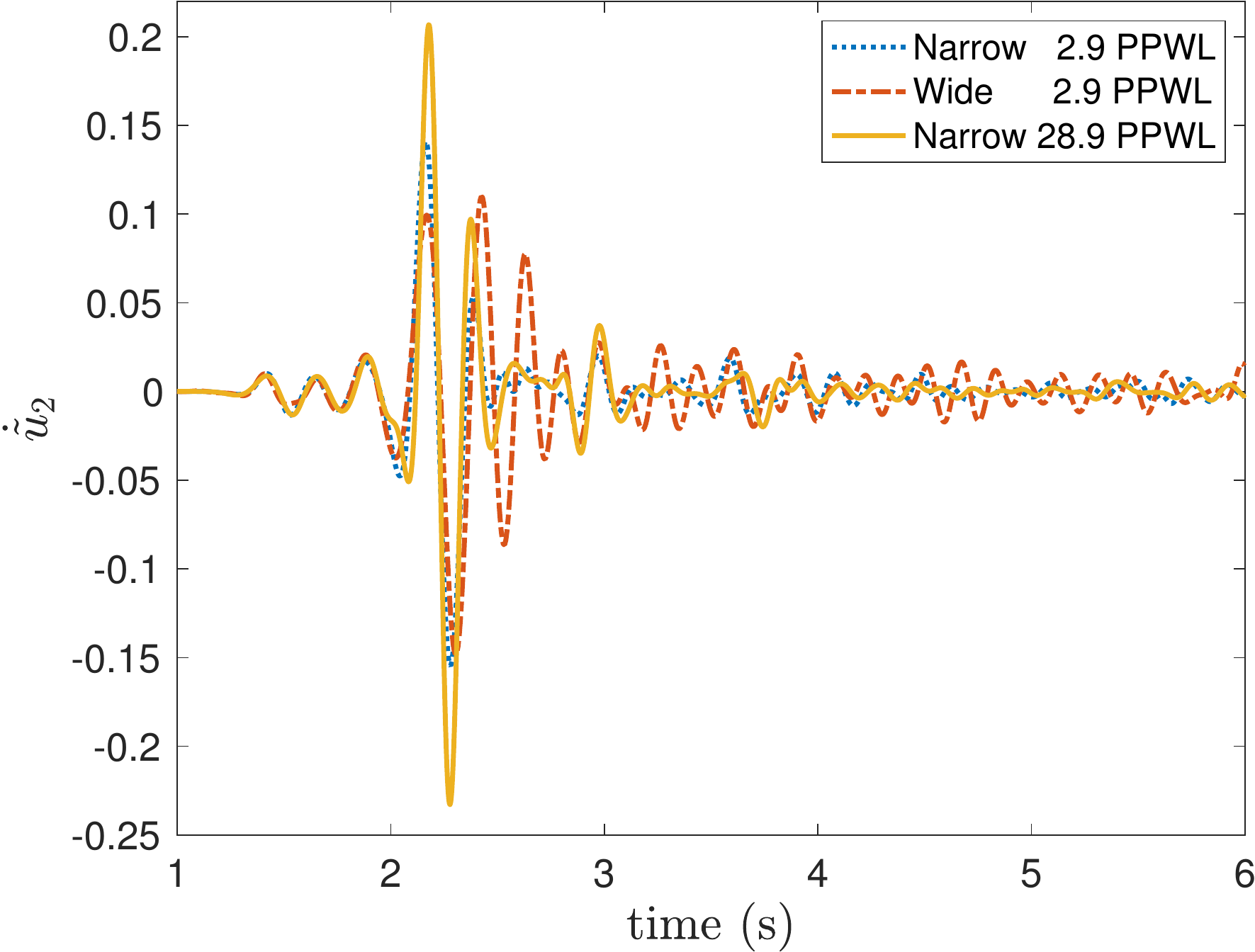}
    \caption{\refthree{Comparing wide and narrow schemes}}\label{fig:seismogramWideNarrowConst}
\end{subfigure}
\caption{\refthree{Seismograms of vertical particle velocity $\dot{\tilde{u}}_2$, recorded at the surface at $X_1= 10$ km, with constant material parameters. PPWL denotes points per wavelength, estimated using the shear wave speed and maximum source frequency. (a) Seismograms generated by the narrow-stencil method at different levels of grid-refinement. (b) Seismograms generated by the wide- and narrow-stencil methods on a coarse grid, compared to a reference solution on a fine grid.}}
\end{figure}
\begin{figure}[h]
  \begin{subfigure}[b]{.49\linewidth}
        \centering
        \includegraphics[width=0.8\linewidth, trim=2.5cm 0.7cm 2.0cm 0.2cm, clip=true]{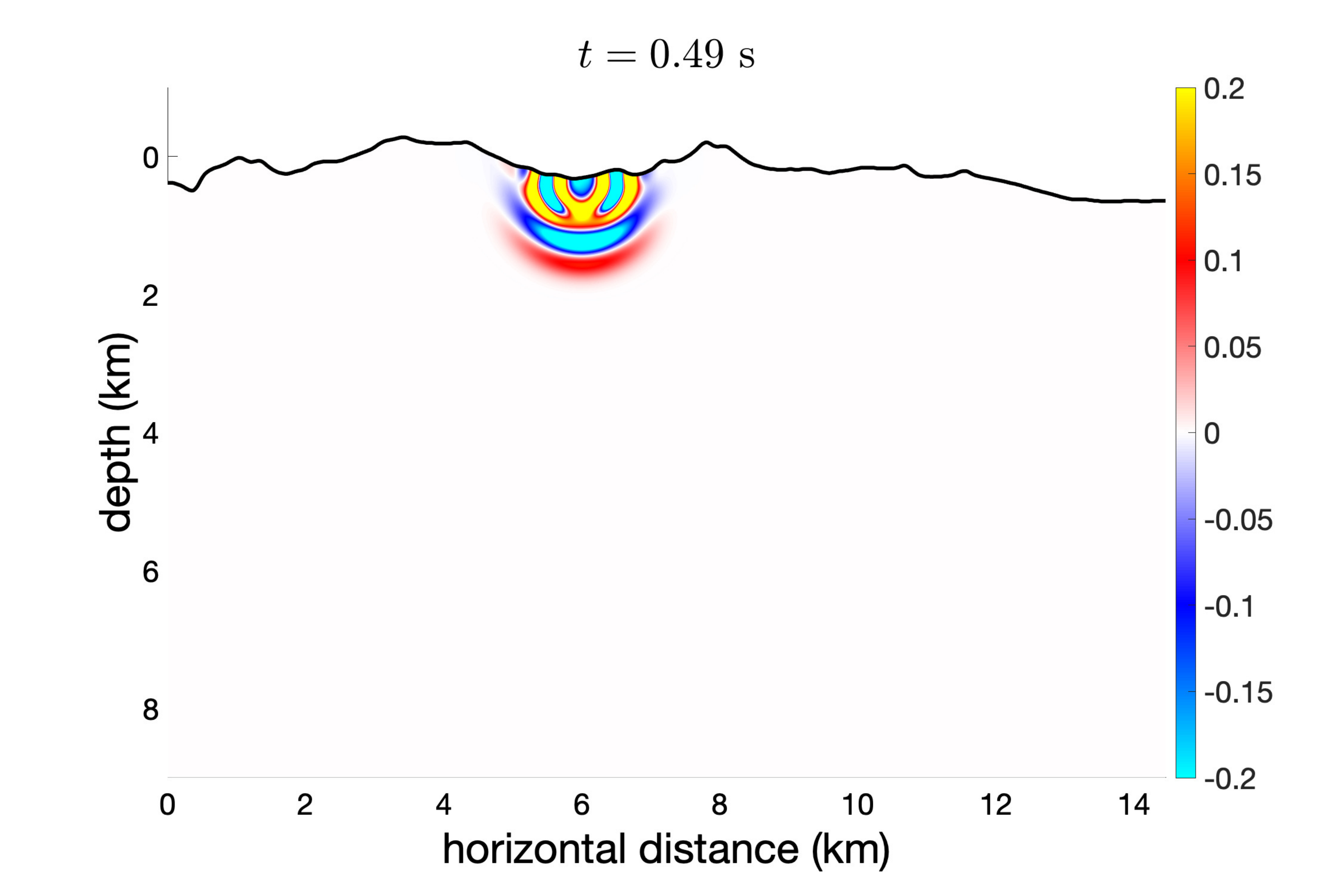}
    \end{subfigure}
    \begin{subfigure}[b]{.49\linewidth}
        \centering
        \includegraphics[width=0.8\linewidth, trim=2.5cm 0.7cm 2.0cm 0.2cm, clip=true]{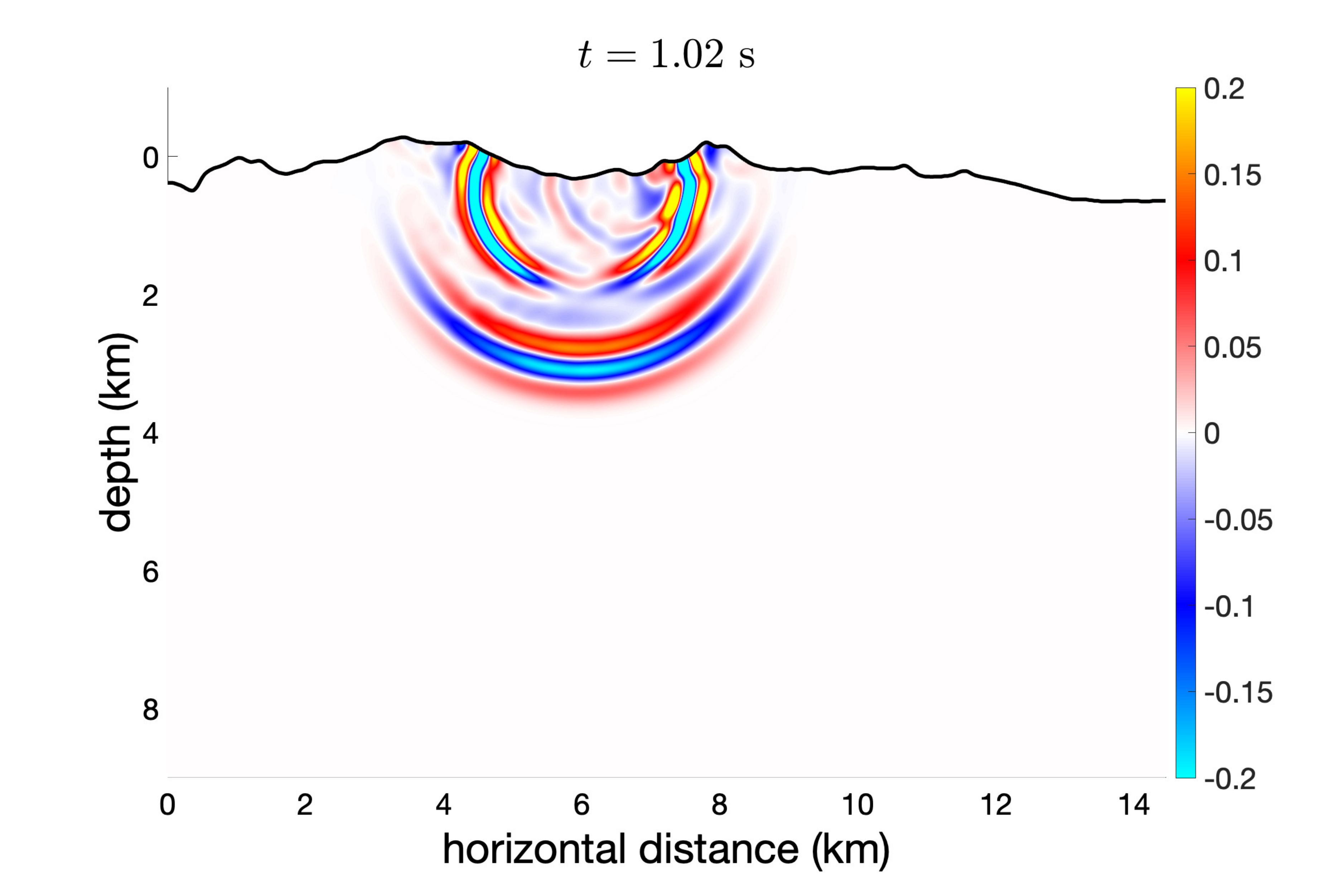}
    \end{subfigure}
    \\
    \begin{subfigure}[b]{.49\linewidth}
        \centering
        \includegraphics[width=0.8\linewidth, trim=2.5cm 0.7cm 2.0cm 0.2cm, clip=true]{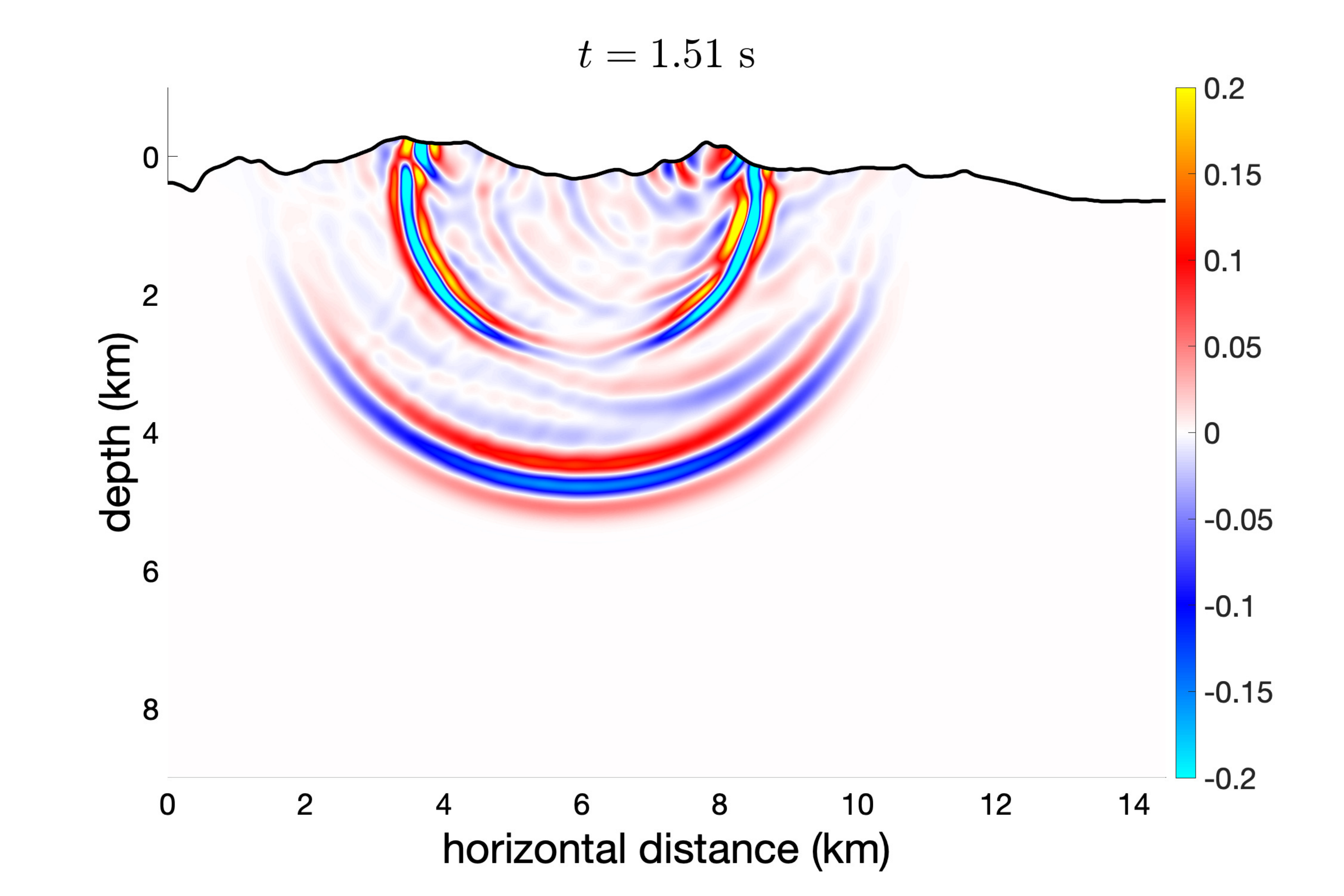}
    \end{subfigure}
    \begin{subfigure}[b]{.49\linewidth}
        \centering
        \includegraphics[width=0.8\linewidth, trim=2.5cm 0.7cm 2.0cm 0.2cm, clip=true]{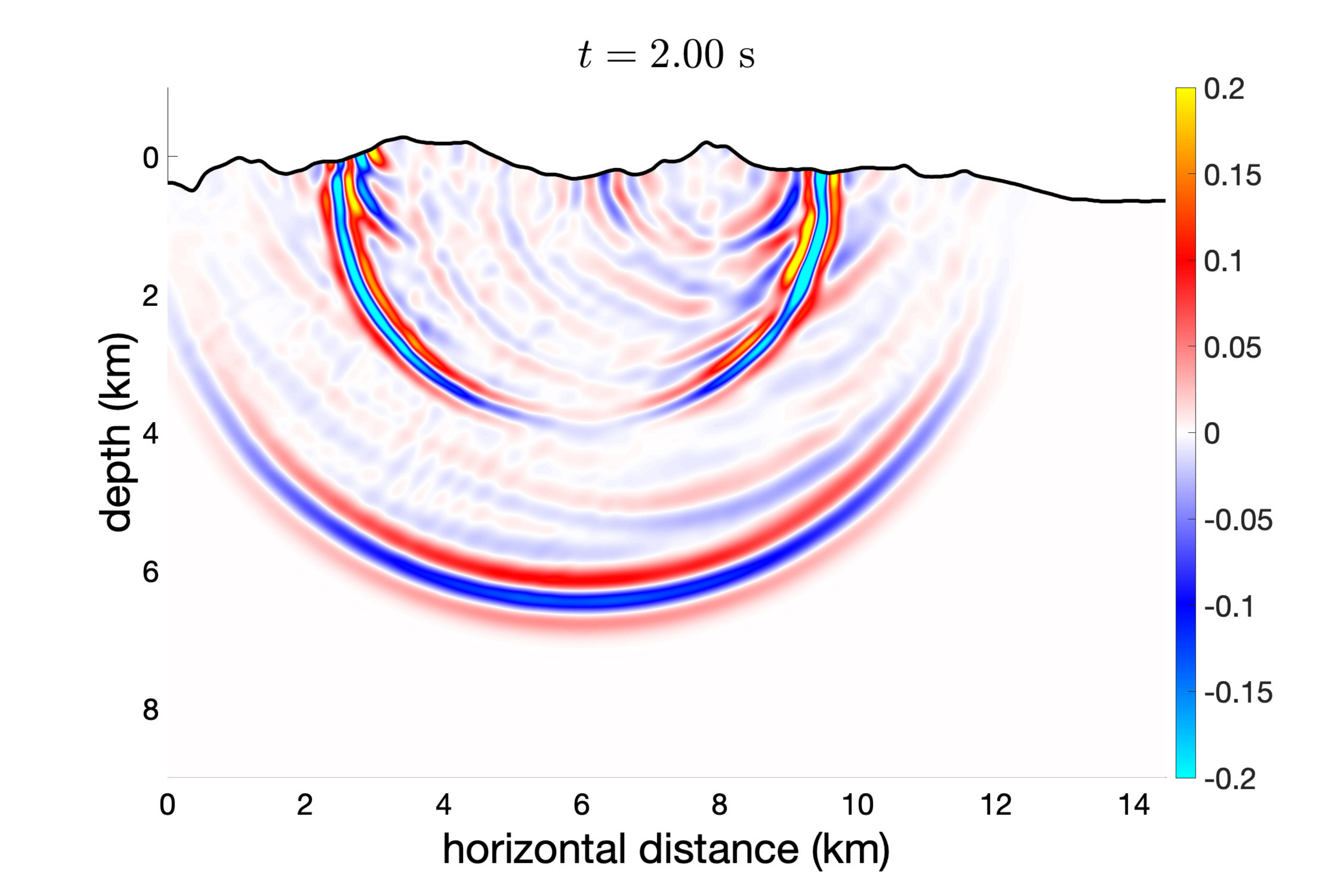}
    \end{subfigure}
    \\
    \begin{subfigure}[b]{.49\linewidth}
        \centering
        \includegraphics[width=0.8\linewidth, trim=2.5cm 0.7cm 2.0cm 0.2cm, clip=true]{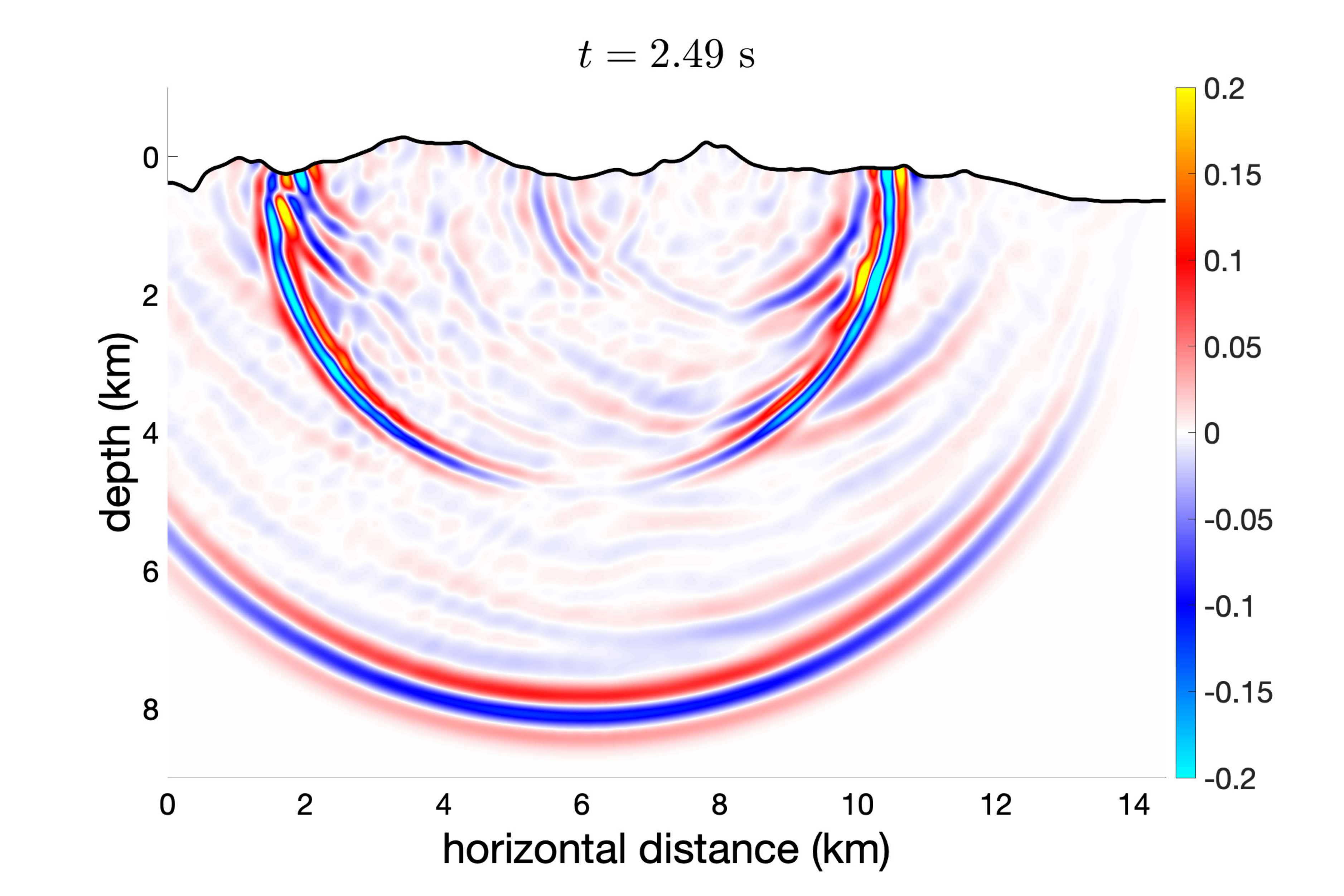}
    \end{subfigure}
    \begin{subfigure}[b]{.49\linewidth}
        \centering
        \includegraphics[width=0.8\linewidth, trim=2.5cm 0.7cm 2.0cm 0.2cm, clip=true]{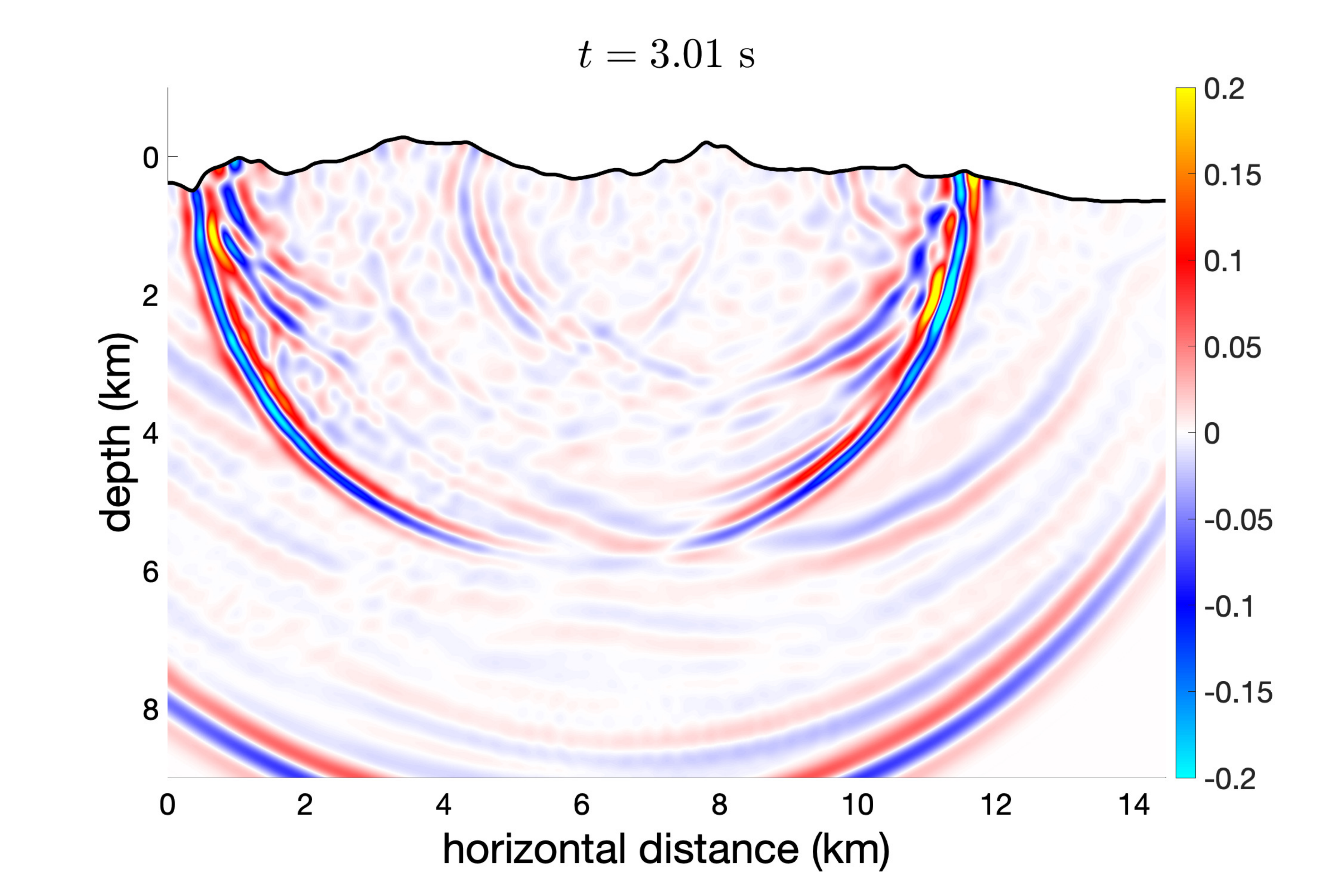}
    \end{subfigure}
    \\
    \centering
    \begin{subfigure}[b]{0.8\linewidth}
        \centering
        \includegraphics[width=1\linewidth, trim=0cm 0cm 0cm 0cm, clip=true]{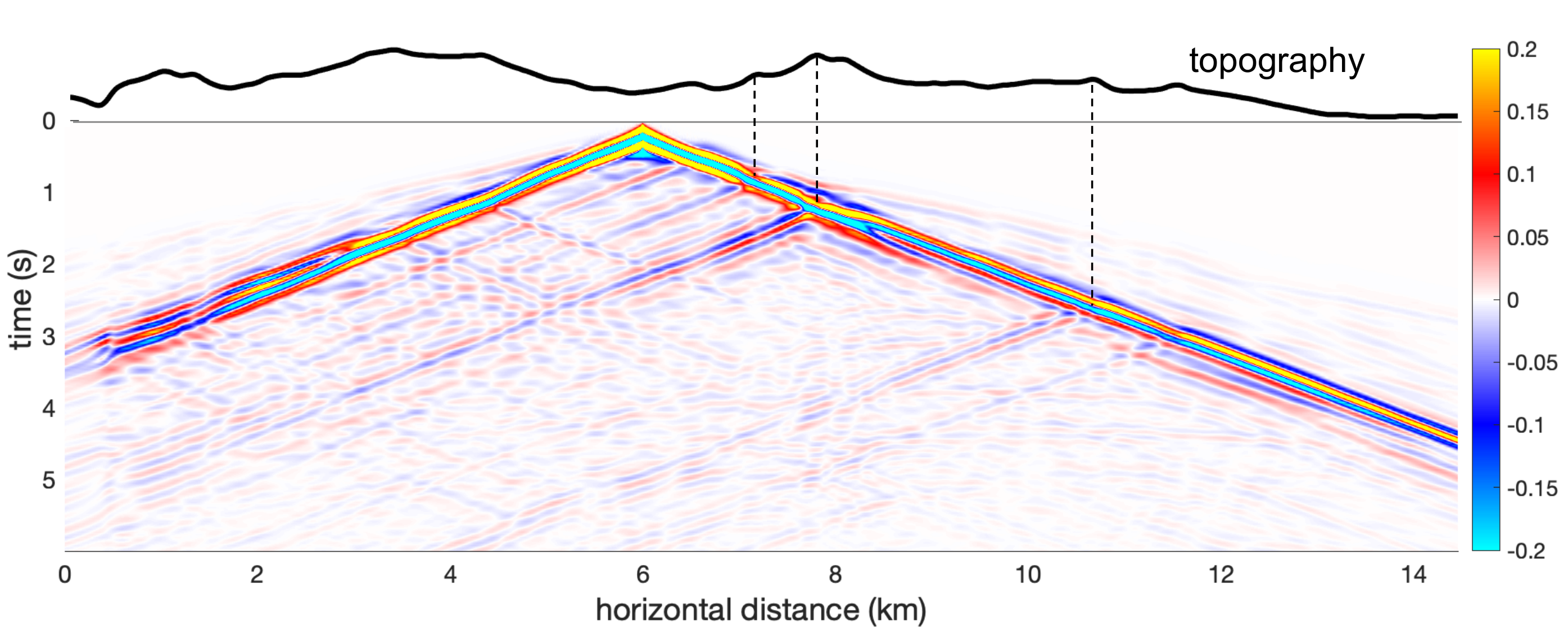}
    \end{subfigure}
    \caption{Plots of $\dot{\tilde{u}}_2$, the vertical component of particle velocity, with constant material parameters. The top three rows show snapshots of $\dot{\tilde{u}}_2$ at different times. The bottom panel shows a space-time plot (shot gather) of $\dot{\tilde{u}}_2$, recorded at the surface. Dashed vertical lines in the space-time plot relate wave scattering to topographical features.}
    \label{fig:snapshot_const}
\end{figure}

\section{Conclusions} \label{sec:conclusions}
We have developed an SBP-SAT method for the anisotropic elastic wave equation on curvilinear multiblock grids in $d$ dimensions. Robin boundary conditions, displacement boundary conditions, and interface conditions are all imposed using SATs, which are designed so that the spatial discretization is energy-stable and self-adjoint. The method assumes fully compatible diagonal-norm SBP operators for variable coefficients. In the numerical experiments, we formed fully compatible operators (here referred to as \emph{adapted} fully compatible operators) by adding a correction to the compatible operators constructed by Mattsson \cite{Mattsson11}. Although the resulting fully compatible operators are one order less accurate at grid end points, our numerical experiments indicate that the global convergence rate is reduced by only half an order, for orders four and six, and not at all for order two. The convergence rates are 2, 3.5, and 4.5, for interior orders two, four and six.

We have applied the new method to problems inspired by elastodynamic cloaking and seismic imaging. In elastodynamic cloaking, anisotropic materials are essential. Hence methods such as ours, which can handle general anisotropy, are the key to evaluating the performance of proposed cloaks via numerical simulation. In the seismic imaging experiment we considered the SEAM Foothills velocity model \cite{Oristaglio2016}, which features large variations in elevation. Our method offers accurate approximation of the topography and the free surface boundary condition, both of which are necessary to model the highly complex surface waves accurately.

MATLAB code that reproduces figures \otherchange{1-5} is available at \url{https://sourceforge.net/projects/elastic-curvilinear/} .


\section*{Acknowledgments}
\otherchange{This research was supported by the Southern California Earthquake Center (Contribution No.\ 10787). SCEC is funded by NSF Cooperative Agreement EAR-1600087 \& USGS Cooperative Agreement G17AC00047.} M.\ Almquist gratefully acknowledges support from the Knut and Alice Wallenberg Foundation (Dnr.\ KAW 2016.0498).
We thank Joe Stefani for useful discussions of seismic imaging and help with the Foothills model.


\end{document}

%% file: preamble_ms.tex
\usepackage{graphicx}
\usepackage[T1]{fontenc}
\usepackage[utf8]{inputenc}
\usepackage{textcomp}
\usepackage[]{amsmath, amssymb}
\usepackage{fancyvrb}
\usepackage{float}
\usepackage{url}
\usepackage{pdfpages}
\usepackage{listings}
\usepackage{amsthm}
\usepackage{grffile}
\usepackage{multirow}
\usepackage{pgfplots}
\usepackage{silence}
\usepackage{todonotes}
\usepackage{subcaption}
\usepackage{gensymb}
\usepackage{epstopdf}
\usepackage{accents}

\usepackage[pdftitle={Anisotropic elastic waves},
  pdfauthor={Almquist, Dunham},
  pdffitwindow=true,
  breaklinks=true,
  colorlinks=true,
  urlcolor=blue,
  linkcolor=red,
  citecolor=blue,
  anchorcolor=red]{hyperref}

\newcommand{\refone}[1]{{#1}}
\newcommand{\reftwo}[1]{{#1}}
\newcommand{\refthree}[1]{{#1}}

\newcommand{\otherchange}[1]{{#1}}

\newcommand{\munderbar}[1]{\underaccent{\bar}{#1}}

\DefineVerbatimEnvironment{code}{Verbatim}{fontsize=\small, xleftmargin=5mm}
\newcommand{\R}{{\mathbb R}}

\newcommand{\idx}[1]{\text{\normalfont\tiny #1}}

\newcommand{\is}{\idx{I}}
\newcommand{\js}{\idx{J}}
\newcommand{\ks}{\idx{K}}
\newcommand{\ls}{\idx{L}}

\newcommand{\il}{\text{\normalfont I}}
\newcommand{\jl}{\text{\normalfont J}}
\newcommand{\kl}{\text{\normalfont K}}
\newcommand{\elll}{\text{\normalfont L}}

\newcommand{\abs}[1]{\left|#1\right|}

\newcommand{\dr}{\mathrm{d}}
\newcommand{\pd}{\partial}
\newcommand{\dd}[2]{\frac{\dr#1}{\dr#2}}
\newcommand{\pdd}[2]{\frac{\partial #1}{\partial #2}}

\newcommand{\bip}[3]{\left( #1,#2 \right)_{#3}}
\newcommand{\ip}[3]{\left( #1, #2 \right)_{#3}}
\newcommand{\norm}[1]{\left\lVert #1 \right\rVert}
\newcommand{\seminorm}[1]{{\left\vert\kern-0.25ex\left\vert\kern-0.25ex\left\vert #1
    \right\vert\kern-0.25ex\right\vert\kern-0.25ex\right\vert}}

\newcommand{\myint}[4]{ \int \limits_{#1}^{#2} \! #3 \, \mathrm{d} #4}

\newcommand{\refdom}{\omega}
\newcommand{\physdom}{\Omega}
\newcommand{\physdomfull}{\bar{\Omega}}
\newcommand{\refboundary}{\pd \refdom}
\newcommand{\physboundary}{\pd \physdom}
\newcommand{\physinterface}{\Gamma}

\newcommand{\physboundaryfull}{\pd \physdomfull}

\newcommand{\volint}[2]{ \ip{#1}{#2}{\refdom} }
\newcommand{\volintphys}[2]{ \ip{#1}{#2}{\physdom} }
\newcommand{\volintphysfull}[2]{ \ip{#1}{#2}{\physdomfull} }
\newcommand{\surfint}[2]{ \bip{#1}{#2}{\refboundary} }
\newcommand{\surfintphys}[2]{ \bip{#1}{#2}{\physboundary} }

\newcommand{\intintphys}[2]{ \bip{#1}{#2}{\physinterface} }

\newcommand{\dintintphys}[2]{ \bip{#1}{#2}{\physinterface} }

\newcommand{\dvolint}[2]{ \ip{#1}{#2}{\refdom} }
\newcommand{\dvolintphys}[2]{ \ip{#1}{#2}{\physdom} }
\newcommand{\dvolintphysfull}[2]{ \ip{#1}{#2}{\physdomfull} }
\newcommand{\dsurfint}[2]{ \bip{#1}{#2}{\refboundary}}
\newcommand{\dsurfintphys}[2]{ \bip{#1}{#2}{\physboundary}}

\newcommand{\dsurfintindexsign}[4]{ \bip{#1}{#2}{\refboundary_{#3}^{#4}} }

\newcommand{\volintphysu}[2]{ \ip{#1}{#2}{\physdom_u} }
\newcommand{\volintphysv}[2]{ \ip{#1}{#2}{\physdom_v} }

\newcommand{\dvolintphysu}[2]{ \ip{#1}{#2}{\physdom_u} }
\newcommand{\dvolintphysv}[2]{ \ip{#1}{#2}{\physdom_v} }

\newcommand{\bv}[1]{\mathbf{#1}}

\newcommand{\db}{\hat{D}}
\newcommand{\ddxi}{\partial_{\idx{I}}}

\newcommand{\ddxk}{\partial_{\idx{K}}}
\newcommand{\ddxii}{\partial_{i}}

\newcommand{\ddxik}{\partial_{k}}

\newcommand{\ddi}{\partial_i}
\newcommand{\ddj}{\partial_j}
\newcommand{\ddk}{\partial_k}

\newcommand{\dil}{\delta_{i \ell}}
\newcommand{\djl}{\delta_{j \ell}}

\newcommand{\dijx}{\delta_{\is \js}}
\newcommand{\dikx}{\delta_{\is \ks}}

\newcommand{\djkx}{\delta_{\js \ks}}
\newcommand{\dklx}{\delta_{\ks \ls}}
\newcommand{\dilx}{\delta_{\is \ls}}
\newcommand{\djlx}{\delta_{\js \ls}}

\newcommand{\sa}{\sum \limits_{\alpha}}
\newcommand{\sar}{\sum \limits_{\alpha,r}}
\newcommand{\sars}{\sum \limits_{\alpha,r,s}}

\newcommand{\sess}{\sum \limits_{s}}

\newcommand{\sk}{\sum \limits_{k}}

\newcommand{\bfu}{{\bv{u}}}

\newcommand{\bfv}{{\bv{v}}}
\newcommand{\bfw}{{\bv{w}}}

\newcommand{\bff}{{\bv{f}}}

\newcommand{\bfs}{{\bv{s}}}
\newcommand{\bfphi}{{\boldsymbol{\phi}}}

\newcommand{\bfpsi}{{\boldsymbol{\psi}}}
\newcommand{\bfchi}{{\boldsymbol{\chi}}}

\newcommand{\bfut}{\dot{\bfu}}
\newcommand{\bfutt}{\ddot{\bfu}}

\newcommand{\stiffref}{c}
\newcommand{\stiffphys}{C}

\newcommand{\tractionop}{\mathbb{T}}
\newcommand{\ctractionop}{T}

\newcommand{\satsym}{\mathbb{Z}}
\newcommand{\rhoref}{\varrho}
\newcommand{\Rmultid}{\mathbb{W}}

\newcommand{\contOpInt}{\mathcal{D}}

\newcommand{\cD}{\mathbb{D}}

\newcommand{\cE}{\mathbb{E}}

\newcommand{\K}{F}
\newcommand{\Kapprox}{\munderbar{\K}}
\newcommand{\Japprox}{\munderbar{J}}
\newcommand{\stiffrefapprox}{\munderbar{\stiffref}}

\newcommand{\gauge}{A}
\newcommand{\ddelta}{\mathbf{d}}

\newcommand{\jump}[1]{\ensuremath{[\![#1]\!]} }
\newcommand{\bfju}{\jump{\bfu}}

\newcommand{\surfjacobian}{\hat{J}}

\newcommand{\setoffaces}{\widehat{\refboundary}}

\newcommand{\surfjacobianapprox}{\munderbar{\surfjacobian}}

\theoremstyle{plain}
\newtheorem{theorem}{Theorem}
\newtheorem{lemma}{Lemma}

\theoremstyle{remark}
\newtheorem{remark}{Remark}

\theoremstyle{definition}
\newtheorem{defi}{Definition}